\documentclass[letterpaper, 11pt,  reqno]{amsart}

\usepackage{amsmath,amssymb,amscd,amsthm,amsxtra, esint}
\usepackage{tikz} 
\usepackage{color}
\usepackage{hyperref}

\usepackage{cases}

\usepackage[left=32mm, right=32mm, 
bottom=27mm]{geometry}

\makeatletter
\renewcommand{\subjclassname}{%
  \textup{1991} Mathematics Subject Classification}
\@xp\let\csname subjclassname@1991\endcsname \subjclassname
\@namedef{subjclassname@2000}{%
  \textup{2000} Mathematics Subject Classification}
\@namedef{subjclassname@2010}{%
  \textup{2010} Mathematics Subject Classification}
  \@namedef{subjclassname@2020}{%
  \textup{2020} Mathematics Subject Classification}
\makeatother

\allowdisplaybreaks[2]

\usepackage{color}

\definecolor{gr}{rgb}   {0.,   0.69,   0.23 }
\definecolor{bl}{rgb}   {0.,   0.5,   1. }
\definecolor{mg}{rgb}   {0.85,  0.,    0.85}
\definecolor{yl}{rgb}   {0.8,  0.7,   0.}
\definecolor{or}{rgb}  {0.7,0.2,0.2}

\usepackage{tikz}
\usepackage{xcolor}
\usetikzlibrary{calc,intersections,through,backgrounds}
\usetikzlibrary{decorations.pathreplacing,decorations.pathmorphing,decorations.markings,decorations.shapes}
\usetikzlibrary{shapes}
\usetikzlibrary{external}

\newtheorem{theorem}{Theorem} [section]

\newtheorem{lemma}[theorem]{Lemma}
\newtheorem{proposition}[theorem]{Proposition}
\newtheorem{remark}[theorem]{Remark}

\newtheorem{corollary}[theorem]{Corollary}

\DeclareMathOperator*{\supp}{supp}

\newcommand{\1}{\hspace{0.5mm}\text{I}\hspace{0.5mm}}
\newcommand{\II}{\text{I \hspace{-2.8mm} I} }
\newcommand{\III}{\text{I \hspace{-2.9mm} I \hspace{-2.9mm} I}}

\newcommand{\IV}{\text{I \hspace{-2.9mm} V}}

\renewcommand{\i}{\iota}
\newcommand{\noi}{\noindent}
\newcommand{\Z}{\mathbb{Z}}
\newcommand{\R}{\mathbb{R}}

\newcommand{\Cb}{\mathcal{C}}
\newcommand{\T}{\mathbb{T}}

\newcommand{\deff}{\stackrel{\textup{def}}{=}}
\newcommand{\U}{\Theta}
\newcommand{\B}{\mathbf{B}}

\let\P= \undefined
\newcommand{\P}{\mathbf{P}}

\newcommand{\Q}{\mathbf{Q}}
\renewcommand{\AA}{\mathbf{A}}
\newcommand{\AB}{\mathbb{A}}

\newcommand{\E}{\mathbb{E}}
\newcommand{\En}{\mathcal{E}}

\renewcommand{\L}{\mathcal{L}}

\newcommand{\Rg}{\mathcal{R}}

\newcommand{\K}{\mathcal{K}}

\newcommand{\F}{\mathcal{F}}

\newcommand{\al}{\alpha}
\newcommand{\be}{\beta}
\newcommand{\dl}{\delta}

\newcommand{\too}{\longrightarrow}

\newcommand{\Dl}{\Delta}
\newcommand{\eps}{\varepsilon}
\newcommand{\kk}{\kappa}
\newcommand{\g}{\gamma}
\newcommand{\G}{\Gamma}
\newcommand{\ld}{\lambda}
\newcommand{\Ld}{\Lambda}
\newcommand{\s}{\sigma}

\newcommand{\ft}{\widehat}

\newcommand{\wt}{\widetilde}
\newcommand{\cj}{\overline}

\newcommand{\dt}{\partial_t}
\newcommand{\dd}{\partial}

\newcommand{\Dlg}{\Delta_\gm}

\newcommand{\diag}{\bigtriangleup}
\newcommand{\diam}{\mathfrak{d}(\M)}

\newcommand{\om}{\omega}
\renewcommand{\O}{\Omega}

\newcommand{\A}{\mathcal{A}}
\newcommand{\les}{\lesssim}

\newcommand{\jb}[1]
{\langle #1 \rangle}

\renewcommand{\S}{\mathcal{S}}
\newcommand{\Sp}{\mathbb{S}}

\newcommand{\gm}{\mathbf{\mathrm{g}}}
\newcommand{\Pg}{P_\gm}

\newcommand{\Prob}{\mathbb{P}}

\newcommand{\M}{\mathcal{M}}

\newcommand{\N}{\mathbb{N}}

\newcommand{\X}{\mathcal{X}}
\newcommand{\Y}{\mathcal{Y}}
\newcommand{\NN}{\mathcal{N}}
\newcommand{\ZZ}{\mathcal{Z}}
\newcommand{\GG}{\mathcal{G}}

\renewcommand{\H}{\mathcal{H}}
\newcommand{\D}{\mathcal{D}}

\newcommand{\Id}{\textup{Id}}

\makeatletter
\def\DeclareSymbol#1#2#3{\expandafter\gdef\csname MH@symb@#1\endcsname{\tikz[baseline=#2, scale=.18]{#3}}}
\def\<#1>{\ensuremath{\mathchoice{\tikzsetnextfilename{macros#1}{\color{black}\csname MH@symb@#1\endcsname}}{\tikzsetnextfilename{macros#1}{\color{black}\csname MH@symb@#1\endcsname}}{\tikzsetnextfilename{macros#1}\scalebox{.7}{\color{black}\csname MH@symb@#1\endcsname}}
{\tikzsetnextfilename{macros#1}\scalebox{.5}{\color{black}\csname MH@symb@#1\endcsname}}}} 
\makeatother 
\DeclareSymbol{1}{0}{\path (-.4,0) -- (.4,0); \draw[line width=0.2mm] (0,0) -- (0,1.2) 
node[circle, fill, draw, solid, inner sep=0pt, minimum size=1.8pt] {};}

\newcommand{\dg}{\mathbf{d}_\gm}
\newcommand{\dgg}{\mathbf{d}_0}
\newcommand{\Gg}{G_\gm}
\newcommand{\Vg}{V_\gm}
\newcommand{\VV}{\mathcal{V}}
\newcommand{\XX}{\mathbf{X}}

\newtheorem*{ackno}{Acknowledgements}

\numberwithin{equation}{section}
\numberwithin{theorem}{section}

\begin{document}
\baselineskip = 15pt

\title[Stochastic quantization of LCFT]
{Stochastic quantization of Liouville conformal field theory}

\author[T.~Oh, T.~Robert, N.~Tzvetkov and Y.~Wang]
{Tadahiro Oh, Tristan Robert, Nikolay Tzvetkov, and Yuzhao Wang}

\address{
Tadahiro Oh, School of Mathematics\\
The University of Edinburgh\\
and The Maxwell Institute for the Mathematical Sciences\\
James Clerk Maxwell Building\\
The King's Buildings\\
Peter Guthrie Tait Road\\
Edinburgh\\ 
EH9 3FD\\
 United Kingdom}

\email{hiro.oh@ed.ac.uk}

\address{
Tristan Robert\\
Fakult\"at f\"ur Mathematik\\
Universit\"at Bielefeld\\
Postfach 10 01 31\\
33501 Bielefeld\\
Germany}

\email{trobert@math.uni-bielefeld.de}

\address{
Nikolay Tzvetkov, Laboratoire AGM\\
Universit\'e de Cergy-Pontoise,  Cergy-Pontoise, F-95000, UMR 8088 du CNRS
}
\email{nikolay.tzvetkov@u-cergy.fr}

\address{
Yuzhao Wang\\
School of Mathematics, 
University of Birmingham, 
Watson Building, 
Edgbaston, 
Birmingham\\
B15 2TT, 
United Kingdom}

\email{y.wang.14@bham.ac.uk}



\begin{abstract}
We study a nonlinear stochastic heat equation forced by a space-time white noise on closed surfaces, with nonlinearity $e^{\be u}$. This equation corresponds to the stochastic quantization of the Liouville quantum gravity (LQG) measure. (i) We first revisit the construction of the LQG measure in Liouville conformal field theory (LCFT) in the $L^2$ regime $0<\be<\sqrt{2}$. This uniformizes in this regime the approaches of David-Kupiainen-Rhodes-Vargas (2016), David-Rhodes-Vargas (2016) and Guillarmou-Rhodes-Vargas (2019) which treated the case of a closed surface with genus 0, 1 and $> 1$ respectively. Moreover, our argument shows that this measure is independent of the approximation procedure for a large class of smooth approximations. (ii) We prove almost sure global well-posedness of the parabolic stochastic dynamics, and invariance of the measure under this stochastic flow. In particular, our results improve previous results obtained by Garban (2020) in the cases of the sphere and the torus with their canonical metric, and are new in the case of closed surfaces with higher genus.

\end{abstract}



\maketitle
%


\tableofcontents

\baselineskip = 14pt

\section{Introduction}

\subsection{Stochastic quantization of LCFT}\label{SUBS:intro}
In this work, we discuss the well-posedness of some stochastic dynamics preserving the Liouville quantum gravity (LQG) measure appearing in the Liouville conformal field theory (LCFT) on a general compact surface. Given a connected, closed (compact, boundaryless), orientable, two-dimensional Riemannian manifold $(\M,\gm)$, the Liouville action $S_\L$ is defined on paths $u:\M\to\R$ by
\begin{align}\label{SL}
S_\L(u;\gm) \deff \frac1{4\pi}\int_{\M}\Big\{|\nabla_\gm u|^2 + Q\Rg_\gm u+ 4\pi\nu e^{\be u}\Big\}d\Vg,
\end{align}
where $\Rg_\gm$ is the Ricci scalar curvature and $d\Vg$ is the volume form. There are three positive parameters, namely the cosmological constant $\nu>0$, the coupling constant $\be >0$ and the charge $Q=\frac2{\be}+\frac{\be}2$. The goal of LCFT is then to compute the $L$-points correlation functions
\begin{align}\label{correlation1}
\Big\langle\prod_{\ell=1}^L\VV_{a_\ell}(x_\ell)\Big\rangle \deff \int \prod_{\ell=1}^L\VV_{a_\ell}(x_\ell)(u) e^{-S_\L(u)}Du,
\end{align}
where the so-called vertex operators $\VV_{a_\ell}(x_\ell)$ are given by
\begin{align*}
\VV_{a_\ell}(x_\ell)(u) \deff e^{a_\ell u(x_\ell)},
\end{align*}
for some points $x_\ell\in\M$ and some coefficients $a_\ell\in\R$.
More generally, one wants to give a meaning to
\begin{align}\label{correlation2}
\int F(u)\prod_{\ell=1}^L\VV_{a_\ell}(x_\ell)(u) e^{-S_\L(u)}Du
\end{align}
for suitable test functions $F$.

The stochastic quantization of LCFT  then consists in constructing a parabolic dynamics given by the stochastic nonlinear heat equation
\begin{align}\label{SLQG1}
\begin{cases}
\dt u - \frac1{4\pi}\Dlg u + \NN(u) = \xi_\gm,~~(t,x)\in \R_+\times\M\\
u_{|t=0} = u_0
\end{cases}
\end{align} 
for some nonlinear interaction $\NN(u)$ and a stochastic forcing term $\xi_\gm : (\O,\Prob)\to \S'(\R\times\M)$ given by a space-time white noise (with $\S'(\R\times\M)$ being the space of space-time tempered distributions), such that the corresponding stochastic flow map $\Phi(t) : (u_0,\xi_\gm) \mapsto u$ satisfies for any $t\ge 0$
\begin{align}
&\int \E\bigg[F\Big(\Phi(t,u_0,\xi_\gm)\big)\Big)\prod_{\ell=1}^L\VV_{a_\ell}(x_\ell)\big(\Phi(t,u_0,\xi_\gm)\big) e^{-S_\L\big(\Phi(t,u_0,\xi_\gm)\big)}\bigg]Du_0 \notag\\
&\qquad= \int F(u_0) \prod_{\ell=1}^L\VV_{a_\ell}(x_\ell)(u_0) e^{-S_\L(u_0)}Du_0.
\label{invariance}
\end{align}

\vspace*{2mm}
LCFT is a special case of Euclidean quantum field theory (QFT), which aims at reconciling quantum mechanics with special relativity. During the 70's and the 80's, a systematic investigation of a minimal set of axioms ensuring the existence of such a theory was performed. In particular, a constructive approach to QFT has been developed through the lens of probability theory, which proved to be widely applicable. Namely, provided that the correlation functions \eqref{correlation1} satisfy some particular properties, there is then a roadmap allowing one to recover a QFT on Minkowski space. This program was particularly successful to treat QFT with polynomial interactions on $\R^{1+1}$ space-time \cite{Simon,Glimm}. 

In this context, the way to build a ``uniform'' measure on the set of (random) quantum relativistic fields is naturally to look at Gibbs type measures by putting a Boltzmann weight on the contribution of each admissible field, thus yielding to \eqref{correlation1}. Indeed, at the classical level, the functional $S_\L$ in \eqref{SL} is well-known for its role in the proof of the uniformization theorem for compact Riemannian 2-manifolds \cite{Berger}: in the case of a closed surface $\M$ of genus greater than 2, it is possible to find a metric with constant negative curvature on $\M$ by looking at the critical points of $S_\L$ when $Q$ is replaced by $\frac2{\be}$. The formal measure $e^{-S_\L(u)}du$ can then be seen as a natural generalization of the classical Wiener measure on the set of paths $u:[0,1]\to\R$. This latter is indeed formally given by $e^{-E(u)}du$ for the energy functional $E(u)=\frac12\int_0^1\big(u'(x)\big)^2dx$, and corresponds to the quantum analogue of the classical trajectories $u : [0,1]\to\R$ given by $u=$constant, which are the critical points of $E$. The value of $Q=\frac2{\be}+\frac{\be}2$ in \eqref{SL} can then be seen as a quantum correction of the classical value $\frac2{\be}$; see also the discussion in the introduction of \cite{LRV2}.

LCFT is then a Euclidean QFT in $1+1$ dimension, which moreover possesses some extra symmetries, to wit, \emph{conformal invariance}. It turns out that its importance goes beyond that of only QFT \cite{Nakayama}. Although it was introduced by Polyakov in his seminal work \cite{Polyakov} as a theory of random surfaces used to describe string theory and quantum gravity, it proved to be also deeply related to probability, geometry and algebra, as it is conjectured to be the scaling limit of random planar maps \cite{MS1}; it also appears in super-symmetric Yang-Mills \cite{SVa,MaO}. This explains why it has attracted a lot of attention in the past decades in both the physics and mathematics communities. 

\subsection{Construction of the LQG measure}
Although natural, the expression \eqref{correlation2} for the LQG measure is merely formal, since $Du$ represents a uniform measure on the set of paths $u : \M \to \R$, which is not properly well-defined. The rigorous construction of expressions such as \eqref{correlation1} has recently attracted a lot of attention \cite{DS11,DKRV,DRV,HRV,Remy,GRV,DOZZ}. It is now well-known that a way to define properly \eqref{correlation1} is to start by interpreting $e^{-S_\L(u)}Du$ as a measure with density with respect to a base Gaussian measure. Indeed, let us only consider the quadratic part of the action and look at the expression
\begin{align*}
e^{-\frac1{4\pi}\int_{\M}|\nabla_\gm u|^2d\Vg}Du = \Xi \prod_{n\ge 1}\frac{\ld_n}{2\pi}e^{-\frac{\ld_n^2 }{4\pi}u_n^2}du_n
\end{align*}
with the normalisation constant formally given by
\begin{align}\label{Xi}
\Xi = \prod_{n\ge 1}\frac{2\pi}{\ld_n}.
\end{align}
Here $0=\ld_0^2 <\ld_1^2\le \ld_2^2\le ...$ are the eigenvalues of $-\Dlg$ associated with an orthonormal basis $\{\varphi_n\}_{n\ge 0}$ of $L^2(\M,\gm)$ of eigenfunctions, so that we can decompose $u:\M\to \R$ as 
\begin{align*}
u = \sum_{n\ge 0}u_n\varphi_n.
\end{align*}
This allows us to interpret for $n\ge 1$ the one-dimensional measure $\frac{\ld_n}{2\pi}e^{-\frac{\ld_n^2 }{4\pi}u_n^2}du_n$ as the density of a normal distribution with variance $\frac{2\pi}{\ld_n^2}$. However, due to $\ld_0=0$, we see that the zero-th frequency is distributed uniformly on $\R$. Thus we can interpret the formal expression $e^{-\frac1{4\pi}\int_{\M}\big|\nabla_\gm u\big|_\gm^2d\Vg}Du$ as $\Xi d\mu_\gm\otimes d\cj X$ by decomposing $u= X_\gm + \cj X$, were $d\cj X$ is the Lebesgue measure on $\R$, and  $\mu_\gm$ is the law of the random variable $X_\gm$ given by the so-called {\it (mass-less) Gaussian free field} (GFF). Namely, $\mu_\gm$ is the Gaussian measure on 
\begin{align}\label{H0}
H^s_0(\M,\gm) \deff \Big\{u\in H^s(\M,\gm)=(1-\Dlg)^{-\frac{s}2}L^2(\M),~\int_{\M}u d\Vg=0\Big\},
\end{align} 
$s<0$, with covariance\footnote{ The law of the Gaussian free field (GFF) \eqref{GFF} is more often referred to as a Gaussian measure on $L_0^2(\M,\gm)$ with covariance $2\pi(-\Dlg)^{-1}$ in the probability literature; here we want to emphasize that this latter operator is not trace class, which makes the support of $\mu_\gm$ actually strictly larger than $L_0^2(\M,\gm)$.} operator $2\pi(-\Dlg)^{s-1}$. In particular we can realise $X_\gm$ as
\begin{align}\label{GFF}
X_\gm(\om) \deff \sum_{n\ge 1}\frac{\sqrt{2\pi}h_n(\om)}{\ld_n}\varphi_n,
\end{align}
where $\{h_n\}_{n\ge 1}$ are iid random variables $\sim \NN(0,1)$ on a probability space $(\O,\Prob)$.

This allows us to interpret \eqref{correlation2} as
\begin{align*}
\int_{H^s_0(\M)}\int_{\R} F(X_\gm+\cj X)d\rho_{\{a_\ell,x_\ell\},\gm}(X_\gm,\cj X),
\end{align*}
where\footnote{which is not to be confused with the random measure $e^{\be X_\gm} d\Vg$, which is the Gaussian multiplicative chaos associated with the GFF $X_\gm$, and is sometimes also referred to as the LQG measure in the probability literature.} the LQG measure $\rho_{\{a_\ell,x_\ell\},\gm}$ is then formally given by
\begin{align}
``d\rho_{\{a_\ell,x_\ell\},\gm}(X_\gm,\cj X) &= \Xi \exp\bigg\{\sum_{\ell=1}^La_\ell(X_\gm+\cj X)(x_\ell) - \frac{Q}{4\pi}\int_{\M}\Rg_\gm (X_\gm+\cj X)d\Vg\notag\\
&\qquad\qquad - \nu \int_{\M}e^{\be (X_\gm+\cj X)}d\Vg\bigg\}d\mu_\gm(X_\gm)\otimes d\cj X".
\label{LQG}
\end{align}

There are still some issues to deal with in order to make sense of \eqref{LQG}. The first one comes from the normalisation constant $\Xi$ in \eqref{Xi}. Indeed, in view of Weyl's law (see \eqref{Weyl} below) the infinite product in \eqref{Xi} does not converge. Still, it is possible to interpret it as 
\begin{align}\label{Xi2}
\Xi = \frac{\sqrt{\Vg(\M)}}{\sqrt{\det'\Dlg}},
\end{align}
where $\det'(\Dlg)$ is the determinant of the Laplace-Beltrami operator, defined as a suitable expression of the spectral zeta function $\zeta(s)=\sum_{n\ge 1}\ld_n^{-s}$; see for example \cite{OPS} or the discussion in \cite[Section 2.3]{GRV}.

The next issue comes from the roughness of the support of $\mu_\gm$: we see that the expressions $X_\gm(x_\ell)$ and $e^{\be X_\gm}$ are still not well-defined. This requires a proper {\it renormalization} procedure. For any $N\in\N$ we therefore look at the truncated measure given by
\begin{align}\label{LQGN}
d\rho_{N,\gm}(X_\gm,\cj X) \deff \ZZ_N^{-1}(\gm)R_N(X_\gm+\cj X)d\mu_\gm(X_\gm)\otimes d\cj X,
\end{align}
where $\ZZ_N$ is the truncated partition function, and the renormalized truncated density in \eqref{LQGN} is given by
\begin{align}
R_N (u) &\deff \Xi\exp\bigg\{ \sum_{\ell=1}^L\Big(a_\ell\P_Nu(x_\ell) -\frac{a_\ell^2}{2}\big(\log N+2\pi C_\P\big)\Big)\notag\\
&\qquad -\frac{Q}{4\pi}\int_\M \Rg_\gm ud\Vg- \nu \int_{\M}e^{-\pi\be^2C_\P}N^{-\frac{\be^2}2}e^{\be \P_Nu}d\Vg\bigg\},
\label{RN}
\end{align}
where the regularization operator is
\begin{align}\label{PN}
\P_N = e^{N^{-2}\Dlg},
\end{align}
and $C_\P$ is a constant which only depends on the choice of the regularization procedure; see Lemma \ref{LEM:GN3} below.

The difference with \eqref{LQG} comes from the introduction of the ``ultraviolet cut-off'' $\P_N$ in \eqref{RN} as well as the removal of the divergent terms $\frac{a_\ell^2}2\log N$ and $e^{\frac{\be^2}2\log N}$. We can thus hope to have cured all the small scales divergences in the model, and to recover a non trivial limit when letting the cut-off parameter $N\to\infty$.

Before stating our first result, we need to make several assumptions on the parameters entering the model.
\begin{align}
&\textbf{($L^2$ regime)} &0<&\be<\sqrt{2}\label{A1}\\
&\textbf{(First Seiberg bound)} & \chi(\M)Q<&\sum_{\ell=1}^L a_\ell \label{Seiberg1}\\
&\textbf{(Integrable insertions)}& \max_{\ell=1,...,L}a_\ell <& \frac{2}{\be},\label{Seiberg2b}
\end{align}
where $\chi(\M)$ in \eqref{Seiberg1} is the Euler characteristic of $\M$. We then have the following result.
\begin{theorem}\label{THM:LQG}
Let\footnote{When $L=0$, we simply do not consider any insertion operator in \eqref{correlation2}. See also Remark~\ref{REM:genus} below.} $L\ge 0$ and $a_\ell\in\R$ and $x_\ell \in\M$, $\ell=1,...,L$. Let also $\nu,\be >0$ and $Q=\frac{2}{\be}+\frac{\be}2$ satisfying \eqref{A1}-\eqref{Seiberg1}-\eqref{Seiberg2b}. Then:\\
\textup{(i)} the truncated measure $\rho_{N,\gm}$ in \eqref{LQGN} converges weakly towards a probability measure $\rho_{\{a_\ell,x_\ell\},\gm}$ on $H^s_0(\M)\oplus\R$, for any $s<0$, which is absolutely continuous with respect to $d\cj\mu_\gm\otimes d\cj  X$, where $\cj\mu_\gm$ is the non-centred Gaussian measure on $H^s_0(\M)$ with covariance $(-\Dlg)^{s-1}$ and mean $\sum_{\ell=1}^L a_\ell \Gg(x_\ell,\cdot)\in H^s_0(\M)$, with $\Gg$ being the Green's function of the Laplace-Beltrami operator on $\M$;\\
\textup{(ii)} the limit measure $\rho_{\{a_\ell,x_\ell\},\gm}$ is independent of the approximation procedure. More precisely, if we replace $\P_N$ in \eqref{PN} by any Schwartz multiplier $\psi(-N^{-2}\Dlg)$ with $\psi\in\S(\R)$ and $\psi(0)=1$, then the same result holds, and the limit $\rho_{\{a_\ell,x_\ell\},\gm}$ obtained this way coincides with the limit $\rho_{\{a_\ell,x_\ell\},\gm}$ obtained by the approximations \eqref{LQGN};\\
\textup{(iii)} the measure $\rho_{\{a_\ell,x_\ell\},\gm}$ is invariant under conformal change of the metric, in the sense that for any continuous bounded test function $F\in C_b\big(H^s_0(\M)\oplus\R\big)$ and any smooth metrics $\gm,\gm_0$ on $\M$ such that $\gm = e^{f_0}\gm_0$ for some $f_0\in C^{\infty}(\M)$, we have 
\emph{Weyl's anomaly}
\begin{align}
&\int_{H^s_0(\M,\gm))}\int_{\R} F\big(X_\gm+\cj X\big)d\rho_{\{a_\ell,x_\ell\},\gm}(X_\gm,\cj X)\notag\\
&\qquad =\exp\Big(\frac{c_L}{96\pi}\int_{\M}\big(|\nabla_0f_0|^2+\Rg_0f_0\big)dV_0-\sum_{\ell=1}^L\big(\frac{Qa_\ell}2-\frac{a_\ell^2}4\big)f_0(x_\ell)\Big)\notag\\
&\qquad\qquad\times \int_{H^s_0(\M,\gm_0))}\int_{\R} F\big(X_0+\cj X-\frac{Q}{2}f_0\big)d\rho_{\{a_\ell,x_\ell\},\gm_0}(X_0,\cj X),
\label{anomaly}
\end{align}
where $c_L = 1+6Q^2$ is the central charge of LCFT.
\end{theorem}
Note that the $\dot{H}^1$ norm of $f_0$ in Weyl's anomaly does not depend on the metrics $\gm$ or $\gm_0$. See Subsection \ref{SUBS:conf}.

Several comments are in order. First, concerning the assumptions \eqref{A1}-\eqref{Seiberg1}-\eqref{Seiberg2b}, note that we only state convergence of the truncated measure in the ``$L^2$ regime'' \eqref{A1} for the coupling constant $\be$, whereas it is known \cite{DKRV,DRV,GRV} that one can define it in the entire \emph{sub-critical regime} $0<\be<2$, and even at the endpoint $\be=2$ \cite{GRV,DRSV1,DRSV2,APS}. However, for the construction of the dynamics in \eqref{invariance}, our argument requires to control the second moment of the Gaussian multiplicative chaos, which gives the restriction \eqref{A1}. See also Proposition \ref{PROP:U} and \ref{PROP:flow} below. 

As for the constraint \eqref{Seiberg2b}, it is more restrictive than the \emph{second Seiberg bound}
\begin{align}\label{Seiberg2}
\max_{\ell=1,...,L} a_{\ell} < Q = \frac{2}{\be} + \frac{\be}2
\end{align}
for which the measure is constructed in \cite{DKRV}, with even the endpoint case being tractable \cite{DKRV2}. Although the full regime $0<\be<2$ and $\max_{\ell}a_\ell <Q$ can be obtained from the arguments in \cite{DKRV,DRV,GRV}, our argument for the construction of the dynamics \eqref{heat0} does not seem to extend beyond \eqref{A1}-\eqref{Seiberg1}-\eqref{Seiberg2b} at this point. See also Remark~\ref{REM:U} below.

Let us also mention that we stated the uniqueness of the measure only with respect to the class of approximations described in Theorem \ref{THM:LQG} (ii), which includes the natural regularizations by the heat kernel \eqref{PN} or the ``smooth'' projection on the finite-dimensional subspace $\mathrm{Vect}\{\varphi_n,~\ld_n\le N\}$ corresponding to a smooth truncation of the eigenfunctions expansion \eqref{GFF}. Still our argument also extends to smoothing operators with kernel $N^2\chi\big(N\dg(x,y)\big)$, $\chi\in C^{\infty}_0\big(0,\i(\M))\big)$ with $\i(\M)$ being the injectivity radius; see Remark~\ref{REM:kernels}.

\begin{remark}\label{REM:genus}\rm
The Seiberg bounds \eqref{Seiberg1} and \eqref{Seiberg2} (and so a fortiori the assumption \eqref{Seiberg2b}) imply that the formal measure $``e^{-S_\L(u)}Du"$ (i.e. without insertions) can only be made sense of in the case of a surface with negative curvature, for which $L=0$ does not violate the condition \eqref{Seiberg1}. On the other hand, we see that for the sphere or the torus, this latter measure is not finite since the Seiberg bound \eqref{Seiberg1} cannot be satisfied for $L=0$ in these cases (see \eqref{bd-RN(X,m)} below). Combining \eqref{Seiberg1} with \eqref{Seiberg2}, we see that we need at least $L\ge 3$ in order to have a non trivial probability measure $\rho_{\{a_\ell,x_\ell\},\gm}$ on $\M=\Sp^2$, whereas in the case of the torus, we need $L\ge 1$. Namely, the uniform distribution of the zero mode $d\cj X$ raises an issue in defining the probability measure $``e^{-S_\L(u)}Du"$ in the case of the sphere and the torus, and we fix this issue by inserting punctures in the definition of the measure \eqref{LQG}. On the other hand, the negative curvature is favorable for the measure construction, in the sense that this issue does not appear. Interestingly, this deepens the analogy with the classical situation encountered in the uniformization theorem as mentioned in the introduction. Indeed, due to Gauss-Bonnet theorem (see \eqref{GB} below) it is not possible to find a smooth metric with constant negative curvature on $\Sp^2$. However, this can be bypassed by considering metrics with conical singularities \cite{Troyanov}, and the Seiberg bound \eqref{Seiberg1} can then be seen at the analogue of the necessary condition on the solid angles of the singularities in order for such a metric to exist.
\end{remark}
\begin{remark}\rm
From the perspective of Remark~\ref{REM:genus}, the LQG measure corresponds to the minimal choice $L=3$ on $\Sp^2$, $L=1$ on $\T^2$ or $L=0$ on hyperbolic surfaces. In the case of the sphere, fixing three punctures $x_1,x_2,x_3\in\Sp^2$ to define the LQG measure, the expression \eqref{correlation1} then actually corresponds to the $(L-3)$-correlation function of the LQG measure. Note that the stochastic quantization procedure as introduced in \cite{PW} then only corresponds to \eqref{correlation4} for the minimal choice of $L$ used to define the LQG measure.
\end{remark}
\subsection{Stochastic dynamics and invariance of the measure}
We now move on to the construction of a stochastic parabolic dynamics leaving the measure $\rho_{\{a_\ell,x_\ell\},\gm}$ invariant.

Indeed, as mentioned in Subsection \ref{SUBS:intro}, the measure \eqref{correlation1} arises in the probabilistic construction of LCFT as a Euclidean QFT. In recent years, we have also seen a rapid development of the (stochastic) PDE approach to constructive Euclidean QFT. Motivated by the stochastic quantization \cite{PW} of the QFT models, a lot of attention has then been devoted to the understanding of singular stochastic parabolic PDEs, with the recent breakthroughs of Hairer \cite{Hairer} through the introduction of regularity structures, and Gubinelli and his collaborators \cite{GIP} through the development of paracontrolled calculus. Let us mention a recent success of this latter approach \cite{GH} where the authors follow through the PDE construction of the Euclidean QFT on $\R^{1+2}$ Minkowski space with quartic interaction potential (the so-called $\Phi^4_3$ model).

More recently, other stochastic quantizations procedures have been investigated, namely the elliptic \cite{ADG1,ADG2} and hyperbolic \cite{GKOT,ORTz,ORSW2,ORW} ones, consisting in looking at the elliptic or hyperbolic counterparts of the stochastic parabolic dynamics \eqref{SLQG1}. In particular, in \cite{ORW}, three of the authors of the present paper investigated both the parabolic and hyperbolic stochastic quantizations of the $\exp(\Phi)_2$ model\footnote{by analogy with the $P(\Phi)_2$ model \cite{BSim} dealing with polynomial interactions. This model is also known as the H\o egh-Krohn model \cite{Hoegh}.} on $\M=\T^2$, corresponding to \eqref{correlation1} without insertions and with a mass term $mu^2$ in the action, $m>0$. This latter in particular destructs the conformal invariance property of the measure. We were then able to prove invariance of the measure under the parabolic stochastic dynamics in the  $L^2$ regime \eqref{A1}, and under the hyperbolic dynamics for some regime of $\be>0$. The same result in the parabolic case also appeared in \cite{HKK}.

In \cite{Garban}, Garban studied a parabolic stochastic dynamics which formally preserves \eqref{LQG}, and discussed its well-posedness on both $\M=\T^2$ and $\M=\Sp^2$ (with $L=1$ and $L=3$ respectively; see Remarks \ref{REM:genus} and \ref{REM:Garban}). However the regime of $\be>0$ covered in \cite{Garban} is somehow more restrictive, and in particular convergence of the smooth approximations as well as rigorous invariance of the measure under the flow are only established in some smaller regime. 

More recently, Dub\'edat and Shen \cite{DS19} studied the stochastic Ricci flow, which describes the evolution of the conformal factor for the metric $\gm = e^{f_0}\gm_0$ with respect to a fixed metric $\gm_0$ with constant scalar curvature (see also Subsection \ref{SUBS:conf})
\begin{align}\label{SRic}
\dt f_0= e^{-f_0}\Dl_0 f_0 + \nu e^{-f_0}\xi_0
\end{align}
where $\Dl_0$ is the Laplace-Beltrami operator for the metric $\gm_0$ and $\xi_0$ is a space-time white noise with respect to $\gm_0$. See also Remark~\ref{REM:Ricci} for further discussion.

Motivated by these recent developments, and in view of the formal expression \eqref{correlation2} with the definition of the Liouville action \eqref{SL}, we then look at
\begin{align}\label{heat0}
\begin{cases}
\dt u - \frac1{4\pi}\Dlg u +\frac{Q}{8\pi}\Rg_\gm + \frac12\nu\be e^{\be u} = \xi_\gm,~~(t,x)\in\R_+\times\M\\
u_{|t=0}=u_0,
\end{cases}
\end{align}
where $u_0$ is an $H^s_0(\M)\oplus\R$-valued random variable with law $\rho_{\{a_\ell,x_\ell\},\gm}$, and the space-time white noise $\xi_\gm$ is a centred Gaussian process on $\D'(\R\times\M)$ with covariance
\begin{align}\label{covxi}
\E\Big[\xi_\gm(\psi_1)\xi_\gm(\psi_2)\Big] = \langle \psi_1,\psi_2\rangle_{t,\gm}
\end{align}
for any $\psi_1,\psi_2\in \D(\R\times\M)$, where $\langle\cdot,\cdot\rangle_{t,\gm}$ is the usual inner product on $L^2(\R;L^2(\M,\gm))$.

In view of the discussion in the previous subsection, in order to make sense of the dynamics \eqref{heat0}, we look at an approximate one leaving the truncated measure $\rho_{N,\gm}$ invariant. Since we have formally $d\rho_{N,\gm} (u)=e^{-\wt E_N (u)}du$ where the renormalized energy reads
\begin{align*}
\wt E_N (u)&\deff  \frac1{4\pi}\int_{\M}\Big\{|\nabla_\gm \P_Nu|^2 + Q\Rg_\gm u + 4\pi\nu e^{-\pi\be^2C_\P}N^{-\frac{\be^2}2} e^{\be \P_Nu }\Big\}d\Vg\notag\\
 &\qquad\qquad -\sum_{\ell=1}^L\Big(a_\ell\P_Nu(x_\ell) -\frac{a_\ell^2}{2}\big(\log N+2\pi C_\P\big)\Big),
\end{align*}
we thus consider the associated truncated stochastic parabolic equation
\begin{multline}
\dt \wt u_N -\frac1{4\pi}\Dlg \wt u_N +\frac{Q}{8\pi}\Rg_\gm + \frac12\nu\be e^{-\pi\be^2C_\P}N^{-\frac{\be^2}2}\P_N\Big\{e^{\be \P_N\wt u_N}\Big\}\\
 =\frac12\sum_{\ell=1}^La_\ell \P_N\dl_{x_\ell} + \xi_\gm,
\label{heat1}
\end{multline}
with initial data $\wt u_0$ distributed by the truncated LQG measure $\rho_{N,\gm}$. As pointed out by Garban \cite{Garban}, the equation \eqref{heat1} is difficult to handle as it is because of the rough \emph{deterministic} term $\sum_{\ell=1}^La_l \P_N\dl_{x_\ell}$. So, we first do a Girsanov transform in \eqref{LQGN} to express the $L$-points correlation function $\ZZ_N(\gm)$ as (see \eqref{Girsanov} below)
\begin{align}
\ZZ_N(\gm)&=\GG_N(\gm)\Xi\int_{\H^s_0(\M,\gm)}\int_{\R}\exp\bigg\{\sum_{\ell=1}^La_\ell\cj X - \frac{Q}{4\pi}\int_{\M}\Rg_\gm(X_\gm+\cj X)d\Vg\notag\\
& - \nu e^{-\pi\be^2C_\P}N^{-\frac{\be^2}2}e^{\be \cj X}\int_{\M}e^{\be \P_NX_\gm+2\pi\be\sum_{\ell=1}^L a_\ell(\P_N\otimes\P_N)\Gg(x_\ell,x)}d\Vg\bigg\}d\mu_\gm(X_\gm)d\cj X
\label{correlation4}
\end{align}
for some constant $\GG_N(\gm)$. Here $\Gg$ is the mean zero Green's function for $(-\Dlg)$ on $(\M,\gm)$ (see \eqref{Green} below), and $(\P_N\otimes\P_N)\Gg$ is then the regularization of $\Gg$ in both variables.

Thus, in order to remove the {\it deterministic} singular part in \eqref{heat1}, we first do the change of variable
\begin{align}\label{Girsanov2}
\wt u_N = u_N + 2\pi\sum_{\ell=1}^L a_\ell(\P_N\otimes\Id)\Gg(x_\ell,x),
\end{align}
so that $ u_N$ now solves the stochastic equation
\begin{multline}\label{heat2}
\dt u_N -\frac1{4\pi}\Dlg u_N  + \frac12\nu\be e^{-\pi\be^2C_\P}N^{-\frac{\be^2}2}\P_N\Big\{e^{\be \P_Nu_N + 2\pi\be\sum_{\ell=1}^La_\ell(\P_N\otimes\P_N)\Gg(x_\ell,x)}\Big\}\\ = -\frac{Q}{8\pi}\Rg_\gm+\frac1{2\Vg(\M)}\sum_{\ell=1}^La_\ell + \xi_\gm,
\end{multline}
with initial data
\begin{align*}
u_N\,_{|t=0} = \wt u_0 - 2\pi\sum_{\ell=1}^L a_\ell(\P_N\otimes\Id)\Gg(x_\ell,x).
\end{align*}
Note that this change of variable cannot be seen as a Da Prato - Debussche trick, since the remainder $\wt u_N$ is not smoother than the original unknown $u_N$. Instead, it can be seen at the equivalent, at the level of the dynamics, of the Girsanov transform \eqref{correlation4} performed at the level of the LQG measure. In particular, writing
\begin{align}\label{sN}
\s_N(x)\deff \int_{H^s_0(\M,\gm)}\big|\P_NX_\gm(x)\big|^2d\mu_\gm,
\end{align} 
we have under the new measure
\begin{align*}
\exp\Big(\sum_{\ell=1}^La_\ell\P_NX_\gm(x_\ell)-\frac{a_\ell^2}{2}\s_N(x_\ell)\Big)d\mu_\gm\otimes d\cj X
\end{align*} 
that the law of 
\begin{align*}
u_0\deff \wt u_0 - 2\pi\sum_{\ell=1}^L a_\ell(\P_N\otimes\Id)\Gg(x_\ell,x)
\end{align*}
 is given by the integrand in \eqref{correlation4}, which is now absolutely continuous with respect to $d\mu_\gm\otimes d\cj X$.

\begin{theorem}\label{THM:GWP}
Let $a_\ell\in\R$, $\ell=1,...,L$, and  $\nu,\be>0$ and $Q=\frac{2}{\be}+\frac{\be}2$ satisfy the assumptions \eqref{A1}-\eqref{Seiberg1}-\eqref{Seiberg2b}. Assume also that 
\begin{align}\label{A2}
0<\be < \sqrt{a_{\ell_\textup{max}}^2+4}-a_{\ell_\textup{max}},
\end{align}
with $a_{\ell_\textup{max}} = \max_{\ell=1,...,L}a_\ell$. Then the equation \eqref{heat0} is almost surely globally well-posed and the law of its solution is invariant. More precisely:\\
\textup{(i)} for any $T>0$ and all $N\in\N$, there exists a unique solution $u_N\in C([0,T];H^s_0(\M)\oplus\R)$ to \eqref{heat2} for $d\mu_\gm(X_\gm)\otimes d\cj X\otimes\Prob$-almost every $u_0 = X_\gm + \cj X$ and $\xi_\gm$, and the solution $u_N$ converges in measure to some non trivial process $u=u(t,X_\gm,\cj X,\om)\in C([0,T];H^s_0(\M)\oplus\R)$;\\
\textup{(ii)} for any test function $F\in C_b\big(H^s_0(\M)\oplus\R\big)$ and any $t\ge 0$ it holds
\begin{multline*}
\int_{H^s_0(\M,\gm)}\int_{\R} \E\Big[ F\big(\wt u(t,X_\gm,\cj X,\om)\big)\Big]d\rho_{\{a_\ell,x_\ell\},\gm}(X_\gm,\cj X)\\ = \int_{H^s_0(\M,\gm)}\int_\R F(X_\gm+\cj X)d\rho_{\{a_\ell,x_\ell\},\gm}(X_\gm,\cj X),
\end{multline*}
where 
\begin{align*}
\wt u(t,X_\gm,\cj X,\om) \deff u(t,X_\gm,\cj X,\om)+2\pi\sum_{\ell=1}^La_\ell\Gg(x_\ell,x) \in C(\R_+;H^s_0(\M)\oplus\R)
\end{align*}
is the limit in law of the solution $\wt u_N$ to \eqref{heat1}.
\end{theorem}

\begin{remark}\rm
We used the approximate equation \eqref{heat2} with a truncated nonlinearity (but without truncating the noise nor the initial data) in order for the truncated dynamics \eqref{heat2} to preserve the truncated Gibbs measure \eqref{LQGN}. However, Theorem \ref{THM:GWP} also holds by replacing \eqref{heat2} with
\begin{multline*}
\dt u_N -\frac1{4\pi}\Dlg u_N  + \frac12\nu\be e^{-\pi\be^2C_\P}N^{-\frac{\be^2}2}e^{\be u_N + 2\pi\be\sum_{\ell=1}^La_\ell(\P_N\otimes\P_N)\Gg(x_\ell,x)}\\ = -\frac{Q}{8\pi}\Rg_\gm+\frac1{2\Vg(\M)}\sum_{\ell=1}^La_\ell + \P_N\xi_\gm,
\end{multline*}
with truncated initial data
\begin{align*}
u_N\,_{|t=0} = \P_N\wt u_0 - 2\pi\sum_{\ell=1}^L a_\ell(\P_N\otimes\P_N)\Gg(x_\ell,x).
\end{align*}
\end{remark}

Let us point out that, apart from the use of the Seiberg bound \eqref{Seiberg1} to ensure that the measure $\rho_{\{a_\ell,x_\ell\}}$ is finite, our analysis is completely insensitive to the particular geometry of $\M$. In this aspect, our result unifies the different treatments of \cite{DKRV,DRV,GRV} for the measure construction\footnote{However, note that in the case of a hyperbolic surface, our results only treat the case of a fixed conformal class for the metric, and we do not average on the space of Riemannian metrics with negative curvature compared to \cite{GRV}. See Remark~\ref{REM:Ricci} below.}, and of \cite{Garban} for the SPDE construction. In order to specialise the regime that is covered by Theorem~\ref{THM:GWP} to the different possible geometries, we state the following.
\begin{corollary}
Almost sure global well-posedness and invariance of the measure in the sense of Theorem~\ref{THM:GWP} hold under the following condition:\\
\textup{(i)} For $\M$ a compact hyperbolic surface and $L=0$, in the whole regime $0<\be<\sqrt{2}$;\\
\textup{(ii)} For $\M=\T^2$ with $L=1$, in the regime $a_1<\frac{2}\be$ and $0<\be < \min\big(\sqrt{2},\sqrt{a_1^2+4}-a_1\big)$. In particular, if $a_1 = \be$, this regime reduces to $0<\be <\sqrt{\frac43}$;\\
\textup{(iii)} For $\M=\Sp^2$ with $L=3$, in the regime $a_{\ell_\textup{max}}<\frac2\be$ and $0<\be<\min\big(\sqrt{2},\sqrt{a_{\ell_\textup{max}}^2+4}-a_{\ell_\textup{max}}\big)$.
\end{corollary}
\begin{remark}\label{REM:Garban}\rm
Without surprise, the adjunction of punctures (i.e. of singularities in the nonlinearity in \eqref{heat2}) reduces the range of admissible $\be$'s. This was already observed in \cite{Garban}, where Garban obtained uniform (in $N$) well-posedness of \eqref{heat2} in the regime\footnote{along with convergence of the approximations $u_N$ in the smaller regime $\frac{\be^2}2-2\sqrt{2}\be + \min(0,\frac{\be}{2\sqrt{2}}-a_{\ell_\textup{max}}\be)>-1$.}
\begin{align*}
\frac{\be^2}2-2\sqrt{2}\be + \min(0,\frac{\be}{2\sqrt{2}}-a_{\ell_\textup{max}}\be)>-2,
\end{align*}
for both $\M=\Sp^2$ with $L=3$ and $\M=\T^2$ with $L=1$. In this latter case, and for the particular choice $a_1 = \be$ (related to random planar maps), his result gives the range $0<\be<\frac{\sqrt{2}}2\approx 0.707$, which we modestly improve to $0<\be <\sqrt{\frac43}\approx 1.15$. In particular, this shows that the conjectured threshold $\g_{\mathrm{pos}}$ given in \cite[Theorem 1.11]{Garban} ($\g_{\mathrm{pos}}=2\sqrt{2}-2$ in the case $L=0$) does not correspond to the actual critical threshold for \eqref{heat1}; see also the discussion after Theorem 1.1 in \cite{ORW}. The main difference in our approach comes from the use of the ``sign-definite structure" as in \cite{ORW} (see Section~\ref{SEC:GWP}) and of $L^p$ based spaces for controlling the solution, whereas Garban used (parabolic) H\"older spaces. In view of the regularity of the main stochastic objects in Proposition \ref{PROP:U} below, which is very sensitive to their integrability properties, we see that working with the more flexible scale of $L^p$ based spaces allows to improve the admissible range. We believe that it is even possible to cover the full sub-critical regime $0<\be<2$ and $\max_{\ell} a_\ell <Q$ without requiring a heavy machinery such as regularity structures or higher-order paracontrolled calculus.
\end{remark}

 \subsection{Scheme of the proof}
 The proof of Theorem \ref{THM:LQG} follows along the line of the previous works \cite{DKRV,DRV,GRV}, and is essentially a consequence of the argument in \cite{DKRV} along with the study of the Green's function and its regularizations performed in Section \ref{SEC:background}.
 
 As for Theorem \ref{THM:GWP}, we first precise some notations. From \eqref{covxi}, we see that the action of $\xi_\gm$ can be extended to functions in $L^2(\R;L^2(\M,\gm))$ and if we define for $n\in\N$, $t\ge 0$ the real-valued process
\begin{align}\label{Bn}
B_{n,\gm}(t) \deff \langle \xi_\gm ,\mathbf{1}_{[0,t]}\varphi_n\rangle_{t,\gm}
\end{align}
then we have from \eqref{covxi} that $\{B_{n,\gm}\}_{n\ge 0}$ is a family of independent Brownian motions so that
\begin{align*}
\xi_\gm = \dt W 
\end{align*}
in $\D'(\R_+\times\M)$, where
\begin{align*}
W = W_\gm + \cj W \deff \sum_{n\ge 1}B_{n,\gm}\varphi_n + B_0\varphi_0.
\end{align*}
In particular we can define $\<1>_\gm$ as the solution to the linear stochastic equation
 \begin{align*}
\begin{cases}
 \big(\dt -\frac1{4\pi}\Dlg\big)\<1>_\gm = \dt W_\gm,\\
 \<1>_\gm(0)=X_\gm,
 \end{cases}
 \end{align*}
which we can also write as
 \begin{align}\label{lp}
 \<1>_\gm(t) = e^{\frac{t}{4\pi}\Dlg}X_\gm + \int_0^te^{\frac{t-t'}{4\pi}\Dlg}dW_\gm(t').
 \end{align}
 
 It is well-known that for $X_\gm$ as in \eqref{GFF} and $W_\gm$ as above, $\<1>_\gm$ is a stationary process belonging almost surely to $C(\R_+;H^s_0(\M))$ for any $s<0$; see Lemma \ref{LEM:SC} below. In particular, in view of the roughness of $\<1>_\gm$, we see that $e^{\<1>_\gm}$ does not make sense, which also justifies the need for the renormalization in the dynamics \eqref{heat2}.

As in \cite{Garban,ORW}, we then start by using Da Prato -Debussche trick \cite{DPD} and write\footnote{Strictly speaking, the trick used by Da Prato and Debussche in \cite{DPD} consists in removing the noise in the right-hand side of \eqref{heat2} by doing the change of variable $u_N = \<1>_\gm + \wt v_N$. Here we remove the whole right-hand side of \eqref{heat2} by the decomposition \eqref{DPD}. This is crucial in order to use the ``sign-definite structure" as in \cite{ORW}.}
 \begin{align}\label{DPD}
 u_N = \<1>_\gm + z+ v_N,
 \end{align}
where\footnote{Also note that here $z$ is still random as it depends on $\cj X$ and $B_0$. This is different from \cite{ORW} where $z$ essentially only contained the last two terms in \eqref{z}.}
 \begin{align}\label{z}
 z(t,x,\cj X,\om) = \cj X + \frac{B_0(t,\om)}{\Vg(\M)} -\frac{Q}{8\pi}\int_{\M}\Pg(t,x,y)\Rg_\gm(y)d\Vg(y)+\frac{t}{2\Vg(\M)}\sum_{\ell=1}^La_\ell,
 \end{align}
with $\Pg$ being the heat kernel on $(\M,\gm)$.

The remainder $v_N$ in \eqref{DPD} now solves
 \begin{align}
 \begin{cases}
\dt v_N -\frac1{4\pi}\Dlg v_N + \frac12\nu\be \P_N\big(e^{\be \P_N(z+v_N)}\U_N\big)  = 0\\
v_N|_{t=0}=0.
\end{cases}
\label{vN}
 \end{align}
where the ``punctured" Gaussian multiplicative chaos (GMC) $\U_N$ is defined by
\begin{align}\label{U}
\U_N(t,x)\deff e^{-\pi\be^2C_\P}N^{-\frac{\be^2}2}e^{\be \P_N\<1>_\gm(t,x)+2\pi\be\sum_{\ell=1}^L a_\ell(\P_N\otimes\P_N)\Gg(x_\ell,x)}.
\end{align}
This type of stochastic object has been largely investigated in the probability literature after the seminal work of Kahane \cite{Ka85} (see also \cite{RVrev,RV}). The following proposition gives the convergence properties of the process $\U_N$.
\begin{proposition}\label{PROP:U}
Let $\al\in (0,2)$, $0<\be^2<2\min(1,\al)$ and $a_\ell<\frac2{\be}$ satisfy
\begin{align*}
\be^2 + 2\be a_\ell^+<4
\end{align*}
for any $\ell=1,...,L$, where $a_\ell^+=\max(a_\ell,0)$. Then for any $T>0$, $\{\U_N\}$ is a Cauchy sequence in $L^2(\mu_\gm\otimes\Prob;L^2([0,T];B^{-\al}_{2,2}(\M)))$, thus converging to some limit $\U$ in this class. 

Moreover, for any $1<p\le 2$ or $p=2m$, $m\in\N$, it holds $\U\in L^p(\mu_\gm\otimes\Prob;L^p([0,T];B^{-\al(p)}_{p,p}(\M)))$, where $\al(p)\in (0,2)$ is such that
\begin{align}
\al(p)>\begin{cases}
 \max\Big\{(p-1)\frac{\be^2}{p}, (p-1)\frac{\be^2}{ p} + 2(p-1)\frac{\be a_\ell^+}{p} + \frac2p -2\Big\} & \text{for }1 < p \leq 2, \rule[-4mm]{0pt}{5pt}\\
\max\Big\{\be a_\ell^+ - \frac4p, (p-1)\frac{\be^2}{2}, (p-1)\frac{\be^2}{2} + \be a_\ell^+ - \frac2p\Big\} & \text{for }p=2m > 2,
\end{cases}
\label{al}
\end{align} 
provided that  $0 < \be^2 < 2\min\big(1,(p-1)^{-1}\big) $ and
\begin{align*}
(p-1)\be^2+2\be a_\ell^+<4
\end{align*}
for all $\ell=1,...,L$.  

 Finally, the limit $\U$ is independent of the approximation procedure in the sense of Theorem \ref{THM:LQG} (ii).
\end{proposition}
\begin{remark}\rm
Here convergence of $\U_N$ is only established in probability. However, if $N$ runs over dyadic integers as in \cite{DS11} instead of $\N$, then we can also prove convergence almost surely.
\end{remark}
In Proposition \ref{PROP:U}, the function space $B^{-\al}_{p,p}(\M)$ denotes the Besov space; see \eqref{Besov} below. Then for any positive distribution $\U_N\in L^2([0,T];H^{(-1)+}(\M))$ we can solve locally the Cauchy problem \eqref{vN} by a fixed-point argument, and the continuous dependence of the flow of \eqref{vN} in $\U_N$ gives the convergence $v_N\to v$ where $v$ solves
\begin{align}\label{v}
 \begin{cases}
\dt v -\frac1{4\pi}\Dlg v + \frac12\nu\be \big(e^{\be (z+v)}\U\big)  = 0\\
v|_{t=0}=0.
\end{cases}
\end{align}
with $\U$ given by Proposition \ref{PROP:U}. Global well-posedness of \eqref{vN} and uniform a priori bounds are established by the same argument as in \cite[Theorem 1.4]{ORW} by exploiting the ``sign-definite structure" of the equation. Note that this is in order to preserve this structure that we consider regularization by using $\P_N$ \eqref{PN} in Theorem \ref{THM:GWP}. A compactness argument then provides existence of a limit which satisfies \eqref{v}, and an energy estimate yields the uniqueness property. In particular, the process $u$ in Theorem~\ref{THM:GWP}~(i) is unique in the class
\begin{align*}
\<1>_\gm + z + \XX^0_T,
\end{align*}
where $\XX_T^0$ is the energy space defined in \eqref{energyZ} below. The invariance of the measure then follows from standard arguments. Finally, observe that the condition $\U\in L^2([0,T];H^{(-1)+}(\M))$ implies the conditions 
\begin{align*}
\begin{cases}
\be^2<2,~~a_{\ell_\textup{max}}<\frac2{\be},\\
\be^2 +2\be a_{\ell_\textup{max}}^+<4,
\end{cases}
\end{align*} in view of Proposition \ref{PROP:U} with $\al=1-$ and $p=2$. This yields the condition \eqref{A2}.

\begin{remark}\label{REM:Ricci}\rm
As pointed out above, another way to look at the stochastic quantization of the Liouville action \eqref{SL} is to consider the stochastic Ricci flow \eqref{SRic} as in \cite{DS19}. Indeed, $S_\L$ in \eqref{SL} depends on both the path $u$ and the metric $\gm$, and the stochastic quantization for $u$ that we considered in \eqref{heat0} corresponds to ``freezing'' $\gm$. On the contrary, \eqref{SRic} takes into account the coupling between $u$ and $\gm = e^{u}\gm_0$ when the metric $\gm$ belongs to the conformal class of a fixed reference metric $\gm_0$. Notice that this makes the dynamics \eqref{SRic} much more nonlinear. 

Interestingly, in the case of a surface with genus strictly greater than one, the probabilistic approach to the full Liouville quantum gravity (LQG) amounts to couple the matter field given by LCFT with gravity by also averaging over all possible metrics, replacing \eqref{correlation1} with
\begin{align}\label{LQG2}
\int\int \prod_{\ell=1}^L\VV_{a_\ell}(x_\ell)(u)e^{-S_\L(u,\gm)}Due^{-S_{EH}(\gm)}D\gm
\end{align}
where again $D\gm$ stands for a ``uniform measure'' on the moduli space of $\M$, and $S_{EH}$ is the Einstein-Hilbert action
\begin{align*}
S_{EH}(\gm) = \frac1{2\ld}\int_{\M}\Rg_\gm d\Vg;
\end{align*}
see \cite{GRV}. It would be interesting to study the stochastic quantization of the full measure \eqref{LQG2} as a coupled stochastic dynamics on $(u,\gm)$ without restricting $\gm$ to a particular conformal class. See also Section 5 in \cite{DS19}.
\end{remark}

The rest of the manuscript is organised as follows. In Section \ref{SEC:background}, we recall the necessary tools to perform the construction of the measure and  analyse the stochastic equation; in particular we prove a Schauder estimate as well as pointwise bounds on a large class of regularizations of the Green's function. Section \ref{SEC:Proba} contains the probabilistic part of the analysis, where we establish Proposition \ref{PROP:U} in Subsection \ref{SUBS:GMC} and prove Theorem \ref{THM:LQG} in Subsection \ref{SUBS:Gibbs}. In Section \ref{SEC:GWP} we analyse the stochastic PDE \eqref{heat2} and give the proof of Theorem \ref{THM:GWP}. Finally, in Appendix \ref{SEC:A}, we briefly discuss how the measure construction fails in the case of a negative cosmological constant $\nu<0$, for both the LQG measure \eqref{LQG} and the $\exp(\Phi)_2$ measure 
$``d\rho^{\exp} = \exp\big\{-\|u\|_{H^1}^2-\nu\int_\M \,:\!e^{\be u}\!:\,d\Vg\big\}Du".$

\section{Background material}\label{SEC:background}

\subsection{Basic tools from analysis on manifolds}
Let $(\M,\gm)$ be a two-dimensional closed (compact, boundaryless), connected, orientable smooth Riemannian manifold, where we fix the metric $\gm$ once and for all. In local coordinates, the metric $\gm$ is given by a smooth function $x\mapsto(\gm_{j,k}(x))_{j,k=1,2}$ taking value in the set of positive symmetric definite matrices. In particular $(\gm_{j,k})$ is invertible and its inverse is denoted by 
\begin{align*}
\big(\gm^{j,k}(x) \big) = \big(\gm_{j,k}(x)\big)^{-1}.
\end{align*}
We also write
\begin{align*}
|\gm(x)| = \det \big(\gm_{j,k}(x)\big)>0.
\end{align*} 
The volume form $\Vg$ can then be written locally as
\begin{align*}
d\Vg(x) = |\gm(x)|^{\frac12}dx.
\end{align*}
The Laplace-Beltrami operator is given in local coordinates by
\begin{align*}
\Dlg f = |\gm|^{-\frac12}\dd_j(|\gm|^{\frac12}\gm^{j,k}\dd_k f\big)
\end{align*}
for any smooth function $f\in C^{\infty}(\M)$. Here and in the following we use Einstein's convention for the summation on repeated indices. 

We set $\{\varphi_n\}_{n\ge 0}\subset C^{\infty}(\M)$ to be a basis of $L^2(\M,\gm)$ consisting of eigenfunctions of $\Dlg$ associated with the eigenvalue $-\ld_n^2$, assumed to be arranged in increasing order: $0=\ld_0 <\ld_1\le\ld_2\le...$ In particular $\varphi_0\equiv \Vg(\M)^{-\frac12}$ is constant. As for $\ld_n$, $n\to \infty$, we have the following asymptotic behaviour given by Weyl's law:
\begin{align}\label{Weyl}
\frac{\ld_n^2}{n} \too \frac{\Vg(\M)}{4\pi} \text{ as }n\to\infty.
\end{align}
Indeed, this is a consequence of the following result\footnote{Compared to the notations of \cite[Theorem 17.5.7]{Hormander}, we have in our case that the density of the weighted manifold $\M$ is given by $|\gm|^{\frac12}=\det(\gm^{j,k})^{-\frac12}$, and the $\ld_n$'s are the square-roots of the eigenvalues of $-\Dlg$.} of H\"ormander \cite[Theorem 17.5.7]{Hormander} for the spectral function of $\Dlg$: there exists $C>0$ such that for any $\Ld>0$ and any $x\in\M$, it holds
\begin{align}\label{EF1}
\Big|\sum_{\ld_n\le \Ld}\varphi_n(x)^2 - \frac1{4\pi}\Ld^2\Big|\le C\Ld.
\end{align}
In particular we also get that for any $n\ge 1$,
\begin{align}\label{EF2}
\big\|\varphi_n\big\|_{L^{\infty}(\M)}\les \ld_n^{\frac12}\les \jb{n}^{\frac14}.
\end{align}
On the other hand, the following lemma shows the uniform boundedness in average of the eigenfunctions.
\begin{lemma}\label{LEM:EF}
For any $A\in\R$, there exists $C>0$ such that for any $\Ld>0$ and any $x\in\M$, it holds
\begin{align*}
\sum_{\ld_n\in (\Ld,\Ld+1]}\frac{\varphi_n(x)^2}{\jb{\ld_n}^A} \le C\sum_{\ld_n\in (\Ld,\Ld+1]}\frac1{\jb{\ld_n}^A}.
\end{align*}
\end{lemma}
\begin{proof}
This follows from \eqref{Weyl}-\eqref{EF1}-\eqref{EF2}. See also \cite[Proposition 8.3]{BTT1}.
\end{proof}

Using the above basis of $L^2(\M,\gm)$, we can expand any $u\in\D'(\M)$ as
\begin{align*}
u = \sum_{n\ge 0}\langle u,\varphi_n\rangle_\gm \varphi_n,
\end{align*}
where $\langle\cdot,\cdot\rangle_\gm$ denotes the distributional pairing $\D'(\M)\times\D(\M)\to \R$ which coincides with the usual inner product in $L^2(\M,\gm)$ for distributions which are regular enough. For any $s\in\R$ and $1\le p\le \infty$, we thus define the Sobolev spaces
\begin{align*}
W^{s,p}(\M,\gm) \deff\Big\{u\in\D'(\M),~\|u\|_{W^{s,p}(\M)} <\infty\Big\}
\end{align*}
where
\begin{align*}
\|u\|_{W^{s,p}(\M)}  \deff \big\|(1-\Dlg)^{\frac{s}2}u\big\|_{L^p(\M)} = \Big\|\sum_{n\ge 0}\jb{\ld_n}^s\langle u,\varphi_n\rangle_\gm \varphi_n\Big\|_{L^p(\M)}.
\end{align*}
When $p=2$ we write $H^s(\M)\deff W^{s,2}(\M)$.

\subsection{Schwartz multipliers and Schauder estimate on $\M$}
Next, recall that for any self-adjoint elliptic operator $\AA$ on $L^2(\M,\gm)$ with discrete spectrum $\{\ld_n\}$ and orthonormal basis of eigenfunctions $\{\varphi_n\}$, the functional calculus of $\AA$ is defined for any $\psi\in L^{\infty}(\R)$ by
\begin{align*}
\psi\big(\AA\big)u = \sum_{n\ge 0}\psi(\ld_n)\langle u,\varphi_n\rangle_\gm \varphi_n,
\end{align*}
for all $u\in C^{\infty}(\M)$. This in particular allows us to define the more general class of Besov spaces. First, using the functional calculus, we can define the Littlewood-Paley projectors $\Q_M$ for a dyadic integer $M\in 2^{\Z_{\ge-1}}\deff \{0,1,2,4,...\}$ as 
\begin{align}\label{QM}
\Q_M = \begin{cases}
\psi_0\big(-(2M)^{-2}\Dlg\big)-\psi_0\big(-M^{-2}\Dlg\big),~~M\ge 1,\\
\psi_0\big(-\Dlg\big),~~M=0,
\end{cases}
\end{align}
where $\psi_0\in C^{\infty}_0(\R)$ is non-negative and such that $\supp\psi_0\subset [-1,1]$ and $\psi_0\equiv 1$ on $[-\frac12,\frac12]$. With the inhomogeneous dyadic partition of unity $\{\Q_M\}_{M\in 2^{\Z_{\ge -1}}}$, we can then define the Besov spaces
\begin{align}\label{Besov}
B^s_{p,r}(\M)\deff \Big\{u\in \D'(\M),~~\|u\|_{\B^s_{p,r}(\M)} \deff \Big(\sum_{M\in 2^{\Z_{\ge -1}}}\jb{M}^{rs}\big\|\Q_Mu\big\|_{L^p(\M)}^r\Big)^{\frac1r} <\infty\Big\}
\end{align}
for any $p,r\in [1,\infty]$. These function spaces are the natural generalization of the usual Besov spaces on $\R^d$ to the context of closed manifolds. In particular, we proved in \cite[Proposition 2.5]{ORTz} the following characterization of these spaces.
\begin{lemma}\label{LEM:Beloc}
Let $(U,V,\kk)$ be a coordinate patch and $\chi\in C^{\infty}_0(V)$. For any $s\in\R$ and $1\leq p,r\leq \infty$, there exist $c,C>0$ such that for any $u\in C^{\infty}(\M)$,
\begin{align*}
c\|\chi u\|_{B^s_{p,r}(\M)}\leq  \|\kk^{\star}(\chi u)\|_{B^{s}_{p,r}(\R^2)}\leq C\|u\|_{B^s_{p,r}(\M)}.
\end{align*}
\end{lemma}

In order to describe more precisely the operators given above by the functional calculus, we start by looking at the local description of multipliers with smooth symbol. For $\psi\in C^{\infty}_0(\R)$, we define through the functional calculus the smoothing operators
\begin{align*}
\psi\big(-N^{-2}\Dlg\big) : v\in \D'(\M) \mapsto \sum_{n\ge 0}\psi(N^{-2}\ld_n^2)\langle v,\varphi_n\rangle\varphi_n \in C^{\infty}(\M),
\end{align*} 
for any $N\in\N$. We can see these operators as semi-classical pseudo-differential operators on $\M$. Indeed, for a semi-classical parameter $h\in (0,1]$, we use the quantization rule
\begin{align*}
a(x,hD) : u\in\S(\R^2)\mapsto a(x,hD)u = \frac1{(2\pi)^2}\int_{\R^2}e^{ix\cdot\xi}a(x,h\xi)\ft u(\xi)d\xi,
\end{align*}
with the convention
\begin{align*}
\ft u(\xi) = \int_{\R^2}e^{-ix\cdot\xi}u(x)dx
\end{align*}
for the Fourier transform. We then have the following result from \cite[Proposition 2.1]{BGT}.
\begin{proposition}\label{PROP:pseudo}
Let $\AA$ be an elliptic self-adjoint operator of order $m$ on $L^2(\M,\gm)$. Let $\psi\in C^{\infty}_0(\R)$, $\kk: U\subset \R^2\rightarrow V\subset \M$ be a coordinate patch, and $\chi,\widetilde{\chi}\in C^{\infty}_0(V)$ with $\widetilde{\chi}\equiv 1$ on $\supp\chi$. Then there exists a sequence of symbols $(a_k)_{k\geq 0}$ in $C^{\infty}_0(U\times\R^2)$ with the following properties:\\

\noi\smallskip
\textup{(i)}  for any $v\in C^{\infty}(\M)$ and $K\geq 1$, we have the expansion
\begin{align}\label{dvp-psi}
\Big\|\kk^{\star}\big(\chi\psi\big(h^m\AA\big)v\big) -\sum_{k=0}^{K-1}h^ka_k(x,hD)\kk^{\star}(\wt\chi v)\Big\|_{H^{s+\s}(\R^2)} \les h^{K-\s-|s|}\|v\|_{H^s(\M)},
\end{align} 
uniformly in $h\in (0,1]$;\\
\noi
\textup{(ii)} for any $x\in U$ the principal symbol is given by
\begin{equation*}
a_0(x,\xi)= \chi(\kk(x))\psi\big(p_m(x,\xi)\big),
\end{equation*}
where $p_m$ is the principal symbol of $\AA$ in $\kk$;\\

\noi\smallskip
\textup{(iii)} for all $k\geq 0$, $a_k$ is supported in 
\begin{align}\label{support}
\Big\{(x,\xi)\in U\times\R^2,~\kk(x)\in\supp\chi,~-p_m(x,\xi)\in \supp \psi^{(k)}\Big\}.
\end{align}
\end{proposition}
\begin{remark}\rm\label{REM:PDO}
Actually, one can check that the proof of \cite[Proposition 2.1]{BGT} also works for multipliers $\psi$ in the class $\S^{-\g}(\R)$ for any $\g>0$, yielding a similar decomposition as in \eqref{dvp-psi} but with symbols $a_k\in \S^{-m\g-k}(U\times \R^2)$ instead of \eqref{support} (see also \cite[Theorem 41.9]{Zworski}). This in particular applies to multipliers $\psi\in \S(\R)$ such that $\psi(x)=e^{-x}$ for $x\ge 0$, for which $\psi\big(-N^{-2}\Dlg\big)=\P_N$ defined in \eqref{PN} above.
\end{remark}

To pursue our investigation of multipliers on $\M$ with symbol in $\S(\R)$, using Proposition \ref{PROP:pseudo} with Remark~\ref{REM:PDO}, we have the following classical estimate on their kernel.
\begin{lemma}\label{LEM:PM}
Let $\psi\in\S(\R)$, and for any $h\in (0,1]$ define the kernel
\begin{align}\label{K1}
\K_h(x,y)\deff \sum_{n\ge 0}\psi\big(h^{2}\ld_n^2\big)\varphi_n(x)\varphi_n(y)
\end{align}  
Then for any $A>0$, there exists $C>0$ such that for any $h\in (0,1]$ and $x,y\in\M$ it holds
\begin{align}\label{K2}
\big|\K_h(x,y)\big|\le C h^{-2}\jb{h^{-1}\dg(x,y)}^{-A},
\end{align}
where $\dg$ is the geodesic distance on $\M$. In particular, it holds
\begin{align}\label{K2b}
\Big\|\psi(-h^{2}\Dlg)\Big\|_{L^p(\M)\to L^q(\M)}\les h^{-2(\frac1p-\frac1q)}
\end{align}
for any $h\in (0,1]$ and $1\le p\le q \le\infty$.
\end{lemma}
\begin{proof}
Let $h\in (0,1]$ and $\K_h$ be given by \eqref{K1} for some $\psi\in\S(\R)$. Fix $(x,y)\in\M$.\\

\noi
\textbf{Case 1. If $\dg(x,y)\les h$:} in this case, we use that $\psi\in\S(\R)$ with Cauchy-Schwarz inequality to bound
\begin{align*}
\big|\K_h(x,y)\big| & \les \sum_{n\ge 0}\jb{h\ld_n}^{-3}\big|\varphi_n(x)\varphi_n(y)\big|\\
&\les \Big(\sum_{n\ge 0}\frac{\varphi_n(x)^2}{\jb{h\ld_n}^3}\Big)^{\frac12}\Big(\sum_{n\ge 0}\frac{\varphi_n(y)^2}{\jb{h\ld_n}^3}\Big)^{\frac12}.
\end{align*}
Now, for any $x\in\M$, we can use H\"ormander's bound \eqref{EF1} to estimate
\begin{align*}
\sum_{n\ge 0}\frac{\varphi_n(x)^2}{\jb{h\ld_n}^3}&\sim \sum_{\ld_n\le h^{-1}}\varphi_n(x)^2 + h^{-3}\sum_{\ld_n\ge h^{-1}}\frac{\varphi_n(x)^2}{\ld_n^3}\\
&\les h^{-2}+O(h^{-1}) + h^{-3}\sum_{j= [h^{-1}]}^{\infty}j^{-3}\sum_{\ld_n\in (j,j+1]}\varphi_n(x)^2\\
&\les h^{-2}+O(h^{-1}) + h^{-3}\sum_{j= [h^{-1}]}^{\infty}j^{-2} \les h^{-2}.
\end{align*}
This proves that $\big|\K_h(x,y)\big|\les h^{-2}$ in this case.\\

\noi
\textbf{Case 2. If $h\ll \dg(x,y)\le \frac{\i(\M)}2$,} where $\i(\M)$ is the injectivity radius of $\M$: we let $A,K\in\N$, and we also take geodesic normal coordinates $(U,\kk)$ centred at $x$ and containing $y$. Then we can use Proposition \ref{PROP:pseudo} with Remark~\ref{REM:PDO} to get that
\begin{align*}
\K_h(x,y) &= \chi(x)\chi(y)\sum_{k=0}^{K-1}h^{k}\int_{\R^2}a_k(\kk^{-1}(x),h\xi)e^{i(\kk^{-1}(x)-\kk^{-1}(y))\cdot\xi}d\xi\\
&\qquad + \chi(x)K_{-K,h}(\kk^{-1}(x),y),
\end{align*}
for some symbols $a_k\in \S^{-m-k}(U\times \R^2)$ for any $m>0$, and where $K_{-K,h}$ is the kernel of
\begin{align*}
R_{-K,h} : v\in H^s(\M) \mapsto \kk_1^{\star}\big(\chi_1\psi\big(-h^2\Dlg\big)\chi_2v\big) -\sum_{k=0}^{K-1}h^ka_k(x,hD)\kk^*(\wt\chi v) \in H^{s+\s}(\R^2)
\end{align*}
which satisfies
\begin{align*}
\big\|R_{-K,h}\big\|_{H^s(\M)\to H^{s+\s}(\R^2)}\les h^{K-\s-|s|}.
\end{align*} 
For this last term, the previous bound implies that
\begin{align*}
\big\|\chi(x)K_{-K,h}(\kk(x),y)\big\|_{L^{\infty}(\M\times\M)}&\les \big\|R_{-K,h}\big\|_{H^{-1-\eps}(\M)\to H^{1+\eps}(\R^2)}\\
& \les h^{K-2-2\eps} \les h^{K-2-2\eps}\dg(x,y)^{-A}
\end{align*}
for any $\eps>0$, where in the last step we used the compactness of $\M$. This is enough for \eqref{K2} by taking $K$ large enough.

As for the contribution of the symbols $a_k$, changing variable and using integration by parts, we can bound it by
\begin{align*}
&\Big| h^{k}\int_{\R^2}a_k\big(\kk^{-1}(x),h\xi\big)e^{i\big(\kk^{-1}(x)-\kk^{-1}(y)\big)\cdot\xi}d\xi\Big|\\&= \Big| h^{k-2}\int_{\R^2}a_k\big(\kk^{-1}(x),\eta\big)e^{\frac{i}h\big(\kk^{-1}(x)-\kk^{-1}(y)\big)\cdot\eta}d\eta\Big|\\
&\les h^{k-2}\jb{h^{-1}(\kk^{-1}(x)-\kk^{-1}(y))}^{-A}\int_{\R^2}\big|\jb{D_{\eta}}^A a_k\big(\kk^{-1}(x),\eta\big)\big|d\eta\\
&\les h^{k-2}\jb{h^{-1}(\kk^{-1}(x)-\kk^{-1}(y))}^{-A}\int_{\R^2}\jb{\eta}^{A-K}d\eta \les h^{k-2}\jb{h^{-1}\dg(x,y)}^{-A},
\end{align*}
where we used that $\big|\partial_x^{\al}a_k(x,\eta)\big| \les \jb{\eta}^{|\al|-m-k}$ for any $\al\in\N^2$ and $m\in \R$ since by Remark~\ref{REM:PDO} $a_k \in \cap_{m\in\R} \S^{-m-k}(U\times\R^2)$ since $\psi\in\S(\R)$. Note that the last integral converges by taking $K$ large enough.\\

\noi
\textbf{Case 3: if $\dg(x,y)>\frac{\i(\M)}2$:} then we can repeat the same argument as above, and use that the pseudo-differential operators $\chi\kk_\star a_k(z,D)\kk^\star\wt\chi$ are properly supported (as can be checked on \eqref{dvp-psi}), which implies in this case that their contribution to $\K_h$ vanishes. Indeed, take charts $(U_j,V_j,\kk_j)$, $j=1,2$, then since $\dg(x,y)>\frac{\i(\M)}2$, we can find $\chi_j\in C^{\infty}_0(V_j)$ such that $\supp \chi_1\cap \supp\chi_2 = \emptyset$, then using \eqref{dvp-psi} with $v= \chi_2 v$ we see that all the term in the sum on $k=0,...,K-1$ vanish, so that
\begin{align*}
\big|\K_h(x,y)\big|&= \big|\chi_1(x)\chi_2(y)\K_h(x,y)\big| = \big|\chi_1(x)\chi_2(y)K_{-K,h}(\kk_1^{-1}(x),y)\big|\\
&\les h^{K-2-2\eps} \les h^{K-2-2\eps}\dg(x,y)^{-A}
\end{align*}
for any $\eps>0$. This proves \eqref{K2} by taking $K$ large enough.

The bound \eqref{K2b} then follows from
\begin{align*}
\Big\|\K_h(x,y)\Big\|_{L^{\infty}_xL^r_y} \les \Big\|h^{-2}\jb{h^{-1}\dg(x,y)}^{-A}\Big\|_{L^{\infty}_xL^r_y} \les h^{\frac2r-2}
\end{align*}
for any $r\ge 1$, the symmetry of $\K_h$ and Schur's lemma.
\end{proof}

As a corollary of the estimates above, we can then investigate the behaviour of the solutions to the linear heat equation on $\M$. Indeed, we can use the eigenfunctions expansion to represent the solution of the heat equation
\begin{align*}
\begin{cases}
\dt u -\frac1{4\pi}\Dlg u = 0,\\
u_{|t=0} = u_0\in \D'(\M)
\end{cases}
~~(t,x)\in \R_+\times\M,
\end{align*}
as the distribution
\begin{align*}
u(t) = e^{\frac{t}{4\pi}\Dlg}u_0 = \sum_{n\ge 0}e^{-\frac{t}{4\pi}\ld_n^2} \langle u_0,\varphi_n\rangle\varphi_n.
\end{align*}
The well-known heat kernel is then the kernel of the above propagator, defined as
\begin{align}\label{heatker}
\Pg(t,x,y)\deff \sum_{n\ge 0}e^{-\frac{t}{4\pi}\ld_n^2}\varphi_n(x)\varphi_n(y).
\end{align}
The convergence of the sum above is a priori in the sense of distributions, but we have the following properties of $\Pg$.
\begin{lemma}\label{LEM:heatker}
Let $\Pg$ be the heat kernel defined in \eqref{heatker} above.\\
\textup{(i)} For any $t>0$, $\Pg(t,\cdot,\cdot)$ is a smooth, symmetric, non-negative function on $\M\times\M$, which satisfies the semigroup property 
\begin{align*}
\Pg(t+s,x,y)=\int_\M \Pg(t,x,z)\Pg(s,y,z)d\Vg(z)
\end{align*}
 for any $t,s>0$. Moreover for any $f\in C(\M)$ it holds 
 \begin{align*}
 \int_\M \Pg(t,x,y)f(y)d\Vg(y) \to f(x)
 \end{align*}
  uniformly as $t\to 0^+$.\\
\textup{(ii)} For any $1\le p\le q\le \infty$, any $1\le r\le\infty$ and $\al_1,\al_2\in\R$ with $\al_1\le \al_2$, there exists $C>0$ such that for any $0<t\le 1$ and any $u\in B^{\al_1}_{p,r}(\M)$, we have \textup{Schauder's estimate}
\begin{align}\label{Schauder}
\big\|e^{\frac{t}{4\pi}\Dlg}u\big\|_{B^{\al_2}_{q,r}(\M)}\le C t^{-\frac{\al_2-\al_1}2-(\frac1p-\frac1q)}\big\|u\big\|_{B^{\al_1}_{p,r}(\M)}.
\end{align}
\end{lemma}
Using the aforementioned smoothing properties of the heat kernel, we can then define the smoothing operators $\P_N$, $N\in\N$ using $\Pg$ as in \eqref{PN} above.
\begin{proof}
The properties (i) are well-known, and can be found e.g. in \cite{Grigoryan}.

As for (ii), since we can write $e^{t\Dlg} = \psi(-h^2\Dlg)$ with $\psi\in\S(\R)$ such that $\psi(x)=e^{-x}$ for $x\ge 0$ and with the semiclassical parameter $h=\sqrt{t}$, we then have that the heat kernel $\Pg$ satisfies the assumptions of Lemma \ref{LEM:PM}. In particular, \eqref{K2b} gives (ii) in the case $\al_1=\al_2$: indeed, using the definition of $B^{\al_2}_{q,r}(\M)$ and \eqref{K2b} we have
\begin{align*}
\big\|e^{\frac{t}{4\pi}\Dlg}u\big\|_{B^{\al_2}_{q,r}(\M)} &= \Big(\sum_{M\in 2^{\Z_{\ge -1}}}\jb{M}^{r\al_2} \big\|e^{\frac{t}2\Dlg}\Q_Mu\big\|_{L^q(\M)}^r\Big)^{\frac1r} \\&\les t^{-(\frac1p-\frac1q)} \Big(\sum_{M\in 2^{\Z_{\ge -1}}} \jb{M}^{r\al_2} \big\|\Q_Mu\big\|_{L^p(\M)}^r\Big)^{\frac1r}\\
& \les t^{-(\frac1p-\frac1q)} \|u\|_{B^{\al_2}_{p,r}(\M)}.
\end{align*}

As for the case $\al_1<\al_2$, writing $e^{\frac{t}{4\pi}\Dlg} = e^{\frac{t}{8\pi}\Dlg}e^{\frac{t}{8\pi}\Dlg}$ and using again \eqref{K2b} we have
 \begin{align*}
 \big\|e^{\frac{t}{4\pi}\Dlg}u\big\|_{B^{\al_2}_{q,r}} &=\Big(\sum_{M\in 2^{\Z_{\ge -1}}} \jb{M}^{r\al_2}\big\| e^{\frac{t}{4\pi}\Dlg}\Q_Mu\big\|_{L^q(\M)}^r\Big)^{\frac1r}\\
 & \les t^{-(\frac1p-\frac1q)} \Big(\sum_{M\in 2^{\Z_{\ge -1}}} \jb{M}^{r\al_2}\big\| e^{\frac{t}{8\pi}\Dlg}\Q_Mu\big\|_{L^p(\M)}^r\Big)^{\frac1r},
 \end{align*}
so it suffices to prove
 \begin{align}\label{P1}
 \Big\|e^{t\Dlg}\Q_M\Big\|_{L^p(\M)\to L^p(\M)} \les \big(\sqrt{t}\jb{M}\big)^{-\al}
 \end{align}
 for any $t\in (0,1]$, $p\in [1,\infty]$ and $M\in 2^{\Z_{\ge -1}}$, where $\al = \al_2-\al_1>0$. Note that $e^{t\Dlg}$ and $\Q_M$ both satisfy the assumptions of Lemma \ref{LEM:PM} with respective semiclassical parameter $h=\sqrt{t}$ and $h=M^{-1}$, so that using \eqref{K2b} again we have
 \begin{align*}
  \Big\|e^{t\Dlg}\Q_M\Big\|_{L^p(\M)\to L^p(\M)} \les \Big\|\Q_M\Big\|_{L^p(\M)\to L^p(\M)} \les 1,
 \end{align*}
 which is enough for \eqref{P1} when $M\le t^{-\frac12}$ or $M\les 1$. Thus in the following we assume $M\ge t^{-\frac12}$ and $M\gg 1$.
 
 Let then $\Q_M = \psi(-M^{-2}\Dlg)$ where $\psi$ is the symbol of $\Q_M$ in \eqref{QM}, and let $(U,V,\kk)$ be a chart on $\M$ with $\chi\in C^{\infty}_0(V)$. Then using Proposition \ref{PROP:pseudo} (twice) we can expand locally
\begin{align*}
\kk^\star\big(\chi \Q_Me^{t\Dlg}\big) &= \chi\Big(\sum_{\ell=0}^{L-1}M^{-\ell}a_{\ell}(x,M^{-1}D)\kk^\star\wt\chi +R_{-L,M}\Big)\kk^\star\big(\wt\chi e^{t\Dlg}\big)\\
&=\chi\sum_{\ell=0}^{L-1}M^{-\ell}a_{\ell}(x,M^{-1}D)\Big(\sum_{k=0}^{K-1}t^{\frac{k}2}a_k(x,\sqrt{t}D)\kk^\star\wt\chi+R_{-K,t}\Big)+R_{-L,M}e^{t\Dlg},
\end{align*}
for some $\wt\chi\in C^{\infty}_0(V)$ and some symbols $a_k,a_{\ell}\in \S^{-\infty}$ as in Proposition \ref{PROP:pseudo}, and the remainders satisfy
\begin{align*}
\big\|R_{-K,t}\big\|_{H^{-s_1}(\M)\to H^{s_2}(\R^2)} \les t^{\frac{K-s_1-s_2}2}
\end{align*}
and
\begin{align*}
\big\|R_{-L,M}\big\|_{H^{-s_1}(\M)\to H^{s_2}(\R^2)} \les M^{s_1+s_2-L}
\end{align*}
for any $s_1,s_2\ge 0$ with $s_1+s_2\le \min\{K,L\}$. In particular, taking $s_1,s_2$ large enough so that $H^{s_2}(\M)\subset L^{\infty}(\M)$ and $L^1(\M)\subset H^{-s_1}(\M)$ and using \eqref{K2b} again, we see that the contribution of $R_{-L,M}e^{t\Dlg}$ satisfies \eqref{P1} provided that we take $K,L$ large enough. We also have for any $\ell=0,...,L-1$
\begin{align*}
\big\|a_{\ell}(x,M^{-1}D)R_{-K,t}\big\|_{L^p(\M)\to L^p(\R^2)} &\les \big\|a_{\ell}(x,M^{-1}D)R_{-K,t}\big\|_{H^{-s_1}(\M)\to H^{s_2}(\R^2)} \\&\les M^{-\al}\big\|R_{-K,t}\big\|_{H^{-s_1}(\M)\to H^{s_2+\al}(\R^2)}\les M^{-\al}t^{\frac{K-s_1-s_2-\al}{2}}
\end{align*}
by taking $s_1,s_2$ and then $K$ large enough, and using that the symbols $a_{\ell}$ are supported on an annulus in view of Proposition \ref{PROP:pseudo} with $M\gg1$.

As for the products $a_{\ell}(x,M^{-1}D)a_k(x,\sqrt{t}D)$, similarly as for the usual composition rule for (semiclassical) pseudo-differential operators (see e.g. \cite{Zworski}), their symbol is given by
\begin{align*}
b_{\ell,k,M,t}(x,\xi) &= \frac1{(2\pi)^2}\int_{\R^2}\int_{\R^2}e^{-iz\cdot\eta}a_\ell\big(x,M^{-1}(\xi+\eta)\big)a_k(x+z,\sqrt{t}\xi)dzd\eta\\
&=\sum_{|\g|=0}^{A-1}\frac{M^{-|\g|}}{\g!}\partial_\xi^{\g}a_{\ell}(x,M^{-1}\xi)\partial^{\g}_xa_k(x,\sqrt{t}\xi) + \wt R_{\ell,k,A,M,t}(x,\xi)
\end{align*}
with
\begin{align*}
\wt R_{\ell,k,A,M,t}(x,\xi)&= \frac1{(2\pi)^2}\sum_{|\g|=A}\frac{AM^{-A}}{\g!}\int_{\R^2}\int_{\R^2} e^{-iz\cdot\eta}\partial^{\g}_\eta a_\ell\big(x,M^{-1}(\xi+\eta)\big)\\
&\qquad\qquad\times\int_0^1(1-\theta)^A\partial^{\g}_xa_k(x+\theta z,\sqrt{t}\xi)d\theta dzd\eta.
\end{align*}
Integrating by parts in $\eta$, we find that the corresponding kernel is bounded by
\begin{align*}
\big|\wt R_{\ell,k,A,M,t}(x,y)\big|&\les M^{-A}\sum_{|\g|=A}\int_{\R^2}\int_{\R^2}\int_{\R^2}\jb{z}^{-3}\Big|\jb{D_\eta}^3\partial^{\g}_\eta a_\ell\big(x,M^{-1}(\xi+\eta)\big)\Big|\\
&\qquad\qquad\times\Big|\int_0^1(1-\theta)^A\partial^{\g}_xa_k(x+\theta z,\sqrt{t}\xi)d\theta\Big| dzd\eta d\xi\intertext{Using next the properties of the symbols $a_k,a_\ell$ given by Proposition \ref{PROP:pseudo} and Remark~\ref{REM:PDO}, we can continue with}
&\les M^{-A}\int_{\R^2}\int_{\R^2}\int_{\R^2}\jb{z}^{-3}\mathbf{1}\big(|\xi+\eta|\les M\big)\jb{\sqrt{t}\xi}^{-B} dzd\eta d\xi \intertext{for any $B>0$. This yields the bound}
&\les t^{-1}M^{2-A}
\end{align*}
which is enough for \eqref{P1} by taking $A$ large enough since we are in the case $M\ge t^{-\frac12}$.

At last, the contribution of the leading symbols are then
\begin{align*}
\Big|\K_{k,\ell,\g,M,t}(x,y)\Big|&\sim\Big|\int_{\R^2}e^{i(x-y)\cdot\xi}\partial_\xi^{\g}a_{\ell}(x,M^{-1}\xi)\partial^{\g}_xa_k(x,\sqrt{t}\xi)d\xi\Big|\\&=M^2\Big|\int_{\R^2}e^{iM(x-y)\cdot\xi}\partial_\xi^{\g}a_{\ell}(x,\xi)\partial^{\g}_xa_k(x,\sqrt{t}M\xi)d\xi\Big|\intertext{Then we integrate by parts in $\xi$ to get}
&\les M^2\jb{M(x-y)}^{-B}\Big|\int_{\R^2}e^{iM(x-y)\cdot\xi}\jb{D_\xi}^{B}\Big[\partial_\xi^{\g}a_{\ell}(x,\xi)\partial^{\g}_xa_k(x,\sqrt{t}M\xi)\Big]d\xi\Big|\intertext{for any $B>0$. We can then use again the properties of the symbols $a_\ell,a_k$ in Proposition \ref{PROP:pseudo} to get}
&\les M^2\jb{M(x-y)}^{-B}\int_{|\xi|\sim 1}\jb{\sqrt{t}M}^{B}\jb{\sqrt{t}M\xi}^{-D}d\xi\\
&\les M^2\jb{M(x-y)}^{-B}\jb{\sqrt{t}M}^{B-D}
\end{align*}
for any $D>0$. We thus get
\begin{align*}
\big\|\K_{k,\ell,\g,M,t}(x,y)\big\|_{L^{\infty}_xL^1_y}\les \jb{\sqrt{t}M}^{B-D},
\end{align*}
by taking $B>2$, which along with Schur's lemma shows \eqref{P1} by taking $D$ large enough. This completes the proof of Lemma \ref{LEM:heatker}.
\end{proof}

\subsection{Properties of the Green's function}
We now investigate the properties of the fundamental solution for the Laplace equation. It is defined as
\begin{align}\label{Green}
\Gg(x,y) \deff \sum_{n\ge 1}\frac{\varphi_n(x,\gm)\varphi_n(y,\gm)}{\ld_n(\gm)^2},
\end{align}
where the convergence of the sum holds again in the sense of distributions. Here we wrote $\varphi_n(x,\gm)$ and $\ld_n(\gm)$ to emphasize the dependence of the eigenfunctions and eigenvalues (and hence of $\Gg$) on the metric.

For a distribution $f\in \D'(\M)$, we also abuse notations and define
\begin{align}\label{mg}
\jb{f}_\gm \deff \langle f,\varphi_0\rangle_\gm \varphi_0 = \frac{1}{\Vg(\M)}\int_{\M}f(x)d\Vg(x).
\end{align}
Then we can check on \eqref{Green} that for any $f\in \D'(\M)$, it holds
\begin{align}\label{Green2}
u(x)=\int_{\M}\Gg(x,y)f(y)d\Vg(y) \qquad \Longleftrightarrow \qquad \begin{cases}-\Dlg u = f-\jb{f}_\gm\\
\jb{u}_\gm = 0,
\end{cases}
\end{align}
where the equalities are again in distributional sense.

We now list some further properties of the Green's function.
\begin{lemma}\label{LEM:Green}
Let $\Gg$ be given by \eqref{Green}.\\
\textup{(i)} $\Gg$ is symmetric and smooth on $\M\times\M\setminus\diag$, where $\diag=\{(x,x),~x\in\M\}$ is the diagonal.\\
\textup{(ii)} There exists a continuous function $\wt \Gg$ on $\M\times\M$ such that for any $(x,y)\in\M\times\M\setminus\diag$ it holds
\begin{align*}
\Gg(x,y) = -\frac1{2\pi}\log\big(\dg(x,y)\big)+\wt \Gg(x,y).
\end{align*}
In particular, $\Gg$ is bounded from below.\\
\textup{(iii)} There exists a constant $C>0$ such that for any $(x,y)\in\M\times\M\setminus\diag$ it holds
\begin{align*}
\big|\nabla_x\Gg(x,y)\big|_\gm\le C\dg(x,y)^{-1}.
\end{align*}
\end{lemma}
\begin{proof}
This is classical, and can be found e.g. in \cite[Section 4.2]{Aubin}.
\end{proof}
Next, we investigate the behaviour of the approximations
\begin{align*}
(\P_N\otimes\P_N)\Gg(x,y)\deff \int_{\M}\int_{\M}P_\gm(4\pi N^{-2},x,z)P_\gm(4\pi N^{-2},y,z')\Gg(z,z')d\Vg(z)d\Vg(z')
\end{align*}
 of the Green's function. Recall indeed that the smoothing operator $\P_N$ has been defined in \eqref{PN} above, so that its kernel is given by $P_\gm(4\pi N^{-2},\cdot,\cdot)$. 
\begin{lemma}\label{LEM:GN1}
There exists $C>0$ such that for any $N\in\N$ and $(x,y)\in\M\times\M\setminus\diag$, it holds
\begin{align}\label{GN1}
-C\le (\P_N\otimes\P_N)\Gg(x,y) \le \Gg(x,y)+ \frac1{N^2\Vg(\M)}.
\end{align}
Moreover, for any $N_1\le N_2\in\N$ and $(x,y)\in\M\times\M\setminus\diag$ we have for any $j=1,2$
\begin{align}
&\Big|(\P_{N_j}\otimes\P_{N_j})\Gg(x,y)-(\P_{N_1}\otimes\P_{N_2})\Gg(x,y)\Big|\notag\\
&\qquad \les \min\Big\{1+\log\frac1{\dg(x,y)},N_1^{-1}\dg(x,y)^{-1}\Big\}.
\label{GN2}
\end{align}
\end{lemma}
\begin{proof}
We begin by proving the first claim. For the lower bound, we have from Lemma~\ref{LEM:Green}~(ii) that there exists $C>0$ such that $\Gg(x,y)\ge -C$ for any $(x,y)\in\M\times\M\setminus\diag$. In particular, since the kernel of $\P_N$ is non-negative in view of \eqref{PN} and Lemma~\ref{LEM:heatker}~(i), we get that
\begin{align*}
(\P_N\otimes\P_N)\big[\Gg+C](x,y) \ge 0,
\end{align*}
i.e.
\begin{align*}
(\P_N\otimes\P_N)\Gg(x,y)\ge -(\P_N\otimes\P_N)C = -C.
\end{align*}

For the upper bound, note that since $\Gg$ has its eigenfunctions expansion only on the diagonal and since the heat kernel $P_\gm$ is symmetric, we actually have
\begin{align*}
(\P_N\otimes\P_N)\Gg(x,y) &= \sum_{n\ge 1}e^{-2N^{-2}\ld_n^2}\frac{\varphi_n(x)\varphi_n(y)}{\ld_n^2} = \int_{\M}P_\gm(8\pi N^{-2}x,z)\Gg(z,y)d\Vg(z)\\
&= \big(e^{2 N^{-2}\Dlg^x}\Gg\big)(x,y) = (\P_N^2\otimes\Id)\Gg(x,y),
\end{align*}
where the notation $e^{t\Dlg^x}$ means that the Laplace-Beltrami operator only acts on the $x$ variable.

 Fix then $y\in\M$ and for $t\ge 0$ and $x\in\M$ with $y\neq x$, define
\begin{align*}
u(t,x)\deff \Gg(x,y)+\frac{t}{4\pi\Vg(\M)}-e^{\frac{t}{4\pi}\Dlg^x}\Gg(x,y).
\end{align*}
Then we have $u\in C(\R_+;L^2(\M))$ and it satisfies
\begin{align*}
u(t) = \frac1{4\pi}\int_0^tP_\gm(t-t',x,y)dt'.
\end{align*}
This indeed follows from \eqref{heatker}, \eqref{Green} and a straightforward computation. The first claim then follows from the previous identity with Lemma \ref{LEM:heatker}, which ensures that $u(t)\ge 0$ for any $t\ge 0$, and taking $t= 8\pi N^{-2}$ with the definition of $\P_N$ \eqref{PN}.

As for the convergence property, for $N_1\le N_2$ and $(x,y)\in\M\times\M\setminus\diag$, we use the remark above with the semigroup property of $P_\gm$ to write for $j=1,2$
\begin{align*}
&\Big|(\P_{N_j}\otimes\P_{N_j})\Gg(x,y)-(\P_{N_1}\otimes\P_{N_2})\Gg(x,y)\Big|\\&\qquad=\Big|\int_{\M}P_\gm(4\pi (N_j^{-2}+N_2^{-2}),x,z')\int_{\M}P_\gm( 4\pi( N_1^{-2}-N_2^{-2}),z',z)\\
&\qquad\qquad\qquad\times\Big[\Gg(z,y)-\Gg(z',y)\Big]d\Vg(z)d\Vg(z')\Big|
\end{align*}

We first deal with the inner integral. For $z',z$ in $\M$, let $\g:t\in [0,\dg(z',z)] \mapsto \g(t)\in\M$ be a unit speed geodesic between $z$ and $z'$, then we can use the mean value theorem and Lemma \ref{LEM:PM} for $h=\sqrt{4\pi(N_1^{-2}-N_2^{-2})}$, to bound for any $A>0$
\begin{align*}
&\Big|\int_{\M}P_\gm(h^2,z',z)\Big[\Gg(z,y)-\Gg(z',y)\Big]d\Vg(z)\Big|\\&\les \int_{\M}h^{-2}\jb{h^{-1}\dg(z',z)}^{-A}\dg(z',z)|\g'(t_z)|_\gm\big|\nabla\Gg(\g(t_z),y)\big|_\gm d\Vg(z)\intertext{for some $t_z\in (0,\dg(z',z))$. Using then Lemma \ref{LEM:Green} (iii) we can continue with}
&\les \int_{\M}h^{-2}\jb{h^{-1}\dg(z',z)}^{-A}\dg(z',z)\dg(\g(t_z),y)^{-1}d\Vg(z)\\
&\les \int_{\M}h^{-2}\jb{h^{-1}\dg(z',z)}^{-A}\dg(z',z)\big[\dg(z,y)\wedge\dg(z',y)\big]^{-1}d\Vg(z).
\end{align*}
We then distinguish several cases to estimate the integral above, depending on which side is the smallest in the triangle made by $z,z'$ and $y$.\\

\noi
\textbf{Case 1: If $\dg(z',z)\les \dg(z,y)\sim\dg(z',y)$.} In this case we can bound the integral with
\begin{align*}
\dg(z',y)^{-1}\int_{\M}h^{-2}\jb{h^{-1}\dg(z',z)}^{-A}\dg(z',z)d\Vg(z) \les \dg(z',y)^{-1}h,
\end{align*}
where the last bound comes from integrating separately on the regions $\dg(z',z)\les h$ and $\dg(z',z)\gtrsim h$ and taking $A$ large enough.\\

\noi
\textbf{Case 2: If $\dg(z',y)\les \dg(z',z)\sim\dg(z,y)$.} We can proceed as in the previous case to get the same bound as above.\\

\noi
\textbf{Case 3: If $\dg(z,y)\les \dg(z',z)\sim\dg(z',y)$.} In this case, we get the bound
\begin{align*}
\int_{\M}h^{-2}\jb{h^{-1}\dg(z',z)}^{-A}\dg(z',z)\dg(z,y)^{-1}d\Vg(z).
\end{align*}
In the region $\dg(z',z)\les h$ we can the estimate it with
\begin{align*}
h^{-2}\dg(z',y)\int_{\M}\dg(z,y)^{-1}d\Vg(z) \les h^{-2}\dg(z',y)^2 \les 1 \les \dg(z',y)^{-1}h,
\end{align*}
where we used that in this case the integral runs over $\dg(z,y)\les \dg(z',y) \sim \dg(z',z)\les h$.\\
In the other region $\dg(z',z)\gtrsim h$ we have the bound
\begin{align*}
h^{A-2}\dg(z',y)^{1-A}\int_{\M}\dg(z,y)^{-1}d\Vg(z)\les h^{A-2}\dg(z',y)^{2-A} \les \dg(z',y)^{-1}h
\end{align*}
 by using again that  the integral runs over $\dg(z,y)\les \dg(z',y)$ and by choosing $A=3$.

Plugging this bound in the double integral above and using Lemma \ref{LEM:PM} with Remark~\ref{REM:PDO} again, we finally estimate for any $A>0$
 \begin{align*}
\Big|(\P_{N_j}\otimes\P_{N_j})\Gg(x,y)-(\P_{N_1}\otimes\P_{N_2})\Gg(x,y)\Big|\les\int_{\M}N_0^2\jb{N_0\dg(x,z')}^{-A}\dg(z',y)^{-1}hd\Vg(z'),
 \end{align*}
 where $N_0=(N_j^{-2}+N_2^{-2})^{-\frac12}$.
To bound this last integral we divide again the argument into three cases.\\

\noi
\textbf{Case 1: If $\dg(x,y)\les \dg(x,z')\sim\dg(z',y)$.} In this case we have
\begin{align*}
&\Big|(\P_{N_j}\otimes\P_{N_j})\Gg(x,y)-(\P_{N_1}\otimes\P_{N_2})\Gg(x,y)\Big|\\
&\qquad\les h\int_{\M}N_0^2\jb{N_0\dg(x,z')}^{-A}\dg(x,z')^{-1}d\Vg(z')\\
&\qquad\les  hN_0^2\int_{\dg(x,y)\les \dg(x,z')\les N_0^{-1}}\dg(x,z')^{-1}d\Vg(z')\\
&\qquad\qquad+hN_0^{2-A}\int_{\dg(x,z')\gtrsim N_0^{-1}\vee \dg(x,y)}\dg(x,z')^{-1-A}d\Vg(z')\\
&\qquad\les hN_0\mathbf{1}_{\{\dg(x,y)\les N_0^{-1}\}} + h\big(N_0^{-1}\vee \dg(x,y)\big)^{-1}\\
&\qquad\les \dg(x,y)^{-1}h
\end{align*}
where in the second to last step we chose $A=2$.\\

\noi
\textbf{Case 2: If $\dg(x,z')\les \dg(x,y)\sim \dg(z',y)$.} In this case we have the bound
\begin{align*}
\Big|(\P_{N_j}\otimes\P_{N_j})\Gg(x,y)-(\P_{N_1}\otimes\P_{N_2})\Gg(x,y)\Big|&\les h\dg(x,y)^{-1}\int_{\M}N_0^2\jb{N_0\dg(x,z')}^{-A}d\Vg(z')\\
& \les \dg(x,y)^{-1}h
\end{align*}
by separating again the regions $\dg(x,z')\les N_0^{-1}$ and $\dg(x,z')\gtrsim N_0^{-1}$ and taking for example $A=3$.\\

\noi
\textbf{Case 3: If $\dg(z',y)\les \dg(x,y)\sim\dg(x,z')$.} The same argument as in the previous Case 3 gives the final bound
\begin{align*}
\Big|(\P_{N_j}\otimes\P_{N_j})\Gg(x,y)-(\P_{N_1}\otimes\P_{N_2})\Gg(x,y)\Big|& \les \dg(x,y)^{-1}h.
\end{align*}
The second bound in the right-hand of the claim \eqref{GN2} then follows from the previous bound with the definition of $h = \sqrt{4\pi(N_1^{-2}-N_2^{-2})}$. The first one is a consequence of the triangle inequality with \eqref{GN1} and Lemma~\ref{LEM:Green}~(ii).

This completes the proof of Lemma~\ref{LEM:GN1}.
\end{proof}

The upper bound \eqref{GN1} in Lemma \ref{LEM:GN1} is enough to bound the (punctured) Gaussian multiplicative chaos $\U_N$ in \eqref{U}, uniformly in $N\in\N$. However, in order to prove the convergence of the truncated LQG measure \eqref{LQGN}, we will also need the uniform boundedness of the negative moments of $\U_N$, for which we need a two-sided bound on $(\P_N\otimes\P_N)\Gg$ in order to use Kahane's inequality (Lemma \ref{LEM:Kahane}) and the argument of \cite{Molchan}. Thus we also establish the following lemma, which gives a two-sided bound for a more general regularization of the Green's function.
\begin{lemma}\label{LEM:GN2}
Let $\psi\in\S(\R)$. There exists $C>0$ such that for any $N\in\N$ and $(x,y)\in\M\times\M\setminus \diag$ we have\\
\textup{(i)} if $\dg(x,y)\ge N^{-1}$ then
\begin{align*}
\Big|(\psi\otimes\psi)\big(-N^{-2}\Dlg\big)\Gg(x,y) - \Gg(x,y)\Big| \le C;
\end{align*}
\textup{(ii)} if $\dg(x,y)\le N^{-1}$ then
\begin{align*}
\Big|(\psi\otimes\psi)\big(-N^{-2}\Dlg\big)\Gg(x,y)-(\psi\otimes\psi)\big(-N^{-2}\Dlg\big)\Gg(x,x)\Big| \le C.
\end{align*}
\end{lemma}
\begin{proof}
As in the proof of Lemma \ref{LEM:PM}, note that it is enough to treat the case $\dg(x,y)\le \frac{\i(\M)}2$.\\

\noi
\textbf{Case 1: if $\dg(x,y)\ge N^{-1}$.} In this case, with the same observation as in the proof of Lemma \ref{LEM:GN1}, we have
\begin{align}\label{KG}
(\psi\otimes\psi)(-N^{-2}\Dlg)\Gg(x,y) &= \int_{\M}\K_N(x,z)\Gg(z,y)d\Vg(z).
\end{align}
  where we wrote $\K_N$ for the kernel of $\psi^2(-N^{-2}\Dlg)$. Thus we have
\begin{align*}
&\Big|(\psi\otimes\psi)\big(-N^{-2}\Dlg\big)\Gg(x,y)-\Gg(x,y) \Big|\\
&\qquad= \Big|\int_{\M} \K_N(x,z)\Big[\log\big(\frac{\dg(z,y)}{\dg(x,y)}\big)+\wt\Gg(z,y)-\wt\Gg(x,y)\Big]dV_\gm(z)\Big| \\ &\qquad\les N^2\int_\M \jb{N\dg(x,z)}^{-10}\Big|\log\big(\frac{\dg(z,y)}{\dg(x,y)}\big)\Big|dV_\gm(z) + O(1),
\end{align*}
where in the first step we used Lemma \ref{LEM:Green}, and in the second one we used Lemma \ref{LEM:PM} with $\K_N$ being the kernel of the multiplier with symbol $\psi^2\in\S(\R)$. Here $O(1)$ stands for a term bounded uniformly in both $N\in\N$ and $(x,y)\in\M\times\M$.\\

\noi
\textbf{Subcase 1.1:} If $\dg(x,z)\les \dg(z,y)\sim \dg(x,y)$. Then the log term in the previous integral is bounded above and below and the integral is $O(1)$.\\

\noi
\textbf{Subcase 1.2:} If $\dg(z,y)\ll \dg(x,z)\sim \dg(x,y)$. Then using polar geodesic coordinates around $y$ and integrating by parts, we can estimate the integral above by
\begin{align*}
N^{-8}\dg(x,y)^{-10}\int_{0}^{\dg(x,y)}\log\Big(\frac{\dg(x,y)}{r}\Big)rdr &= N^{-8}\dg(x,y)^{-10}\Big[\frac{r^2}2\log\Big(\frac{\dg(x,y)}{r}\Big)\Big]_0^{\dg(x,y)}\\
&\qquad\qquad+N^{-8}\dg(x,y)^{-10}\int_{0}^{\dg(x,y)}\frac{r}2dr\\
&\les N^{-8}\dg(x,y)^{-8}\les 1.
\end{align*}

\noi
\textbf{Subcase 1.3:} If $N^{-1}\le\dg(x,y)\ll \dg(x,z)\sim \dg(z,y)$. In this case, using again polar coordinates around $y$ and noting $\diam <\infty$ for the diameter of $\M$, we get the bound
\begin{align*}
N^{-8}\int_{\dg(x,y)}^{\diam}r^{-10}\log\Big(\frac{r}{\dg(x,y)}\Big)rdr &\les N^{-8}\Big[r^{-8}\log\Big(\frac{\dg(x,y)}{r}\Big)\Big]_{\dg(x,y)}^{\diam}\\
&\qquad\qquad+N^{-8}\int_{\dg(x,y)}^{\diam}r^{1-10}dr\\
&\les N^{-8}\log\big(\dg(x,y)\big)+N^{-8}\dg(x,y)^{-8}\les 1.
\end{align*} 

\noi
\textbf{Case 2: If $\dg(x,y)\le N^{-1}$.} In this case we have similarly
\begin{align*}
&\Big|(\psi\otimes\psi)\big(-N^{-2}\Dlg\big)\Gg(x,y)-(\psi\otimes\psi)\big(-N^{-2}\Dlg\big)\Gg(x,x) \Big|\\
 &\qquad\les N^2\int_\M \jb{N\dg(x,z)}^{-10}\Big|\log\big(\frac{\dg(z,y)}{\dg(z,x)}\big)\Big|dV_\gm(z) + O(1).
\end{align*}

\noi
\textbf{Subcase 2.1:} If $\dg(x,y)\les \dg(z,x)\sim \dg(z,y)$. In this case the log term is bounded above and below, so the integral above is $O(1)$.\\

\noi
\textbf{Subcase 2.2:} If $\dg(z,x)\les \dg(z,y)\sim\dg(x,y)\le N^{-1}$. In this case, using polar geodesic coordinates around $x$ and integrating by parts we can bound the previous integral with
\begin{align*}
N^2\int_0^{\dg(x,y)}\log\Big(\frac{\dg(x,y)}{r}\Big)rdr &= N^2 \Big[\frac{r^2}{2}\log\big(\frac{\dg(x,y)}{r}\big)\Big]_0^{\dg(x,y)} + \frac{N^2}2\int_0^{\dg(x,y)}rdr\\
& = \frac{N^2}{4}\dg(x,y)^2 \les 1.
\end{align*}
\textbf{Subcase 2.3:} If $\dg(z,y)\les \dg(z,x)\sim\dg(x,y)\le N^{-1}$. This case follows from the same computation as in Subcase 2.2.

This completes the proof of Lemma \ref{LEM:GN2}.
\end{proof} 

 In view of Lemma \ref{LEM:Green}, the bound of Lemma \ref{LEM:GN2} (i) is enough for our purpose. In order to precise the one of Lemma \ref{LEM:GN2} (ii), we prove the following.
\begin{lemma}\label{LEM:GN3}
Let $\psi\in\S(\R)$ such that $\psi(0)=1$. Then there exists a constant $C_\psi\in\R$ such that
\begin{align*}
\Big\|(\psi\otimes\psi)(-N^{-2}\Dlg)\Gg(x,x) - \frac1{2\pi}\log N - \wt \Gg(x,x)-C_\psi\Big\|_{L^{\infty}(\M)}\too 0
\end{align*}
as $N\to\infty$, where $\wt \Gg$ is given in Lemma \ref{LEM:Green}.
\end{lemma}
\begin{proof}
Let $\K_N$ be as in the proof of Lemma \ref{LEM:GN2}. Note that we have again the identity \eqref{KG}. Using Lemma \ref{LEM:Green}, we first decompose
\begin{align*}
(\psi\otimes\psi)(-N^{-2}\Dlg)\Gg(x,x) &= -\frac1{2\pi}\int_{\M}\K_N(x,y)\log\big(\dg(y,x)\big)d\Vg(y)\\
&\qquad +\int_{\M}\K_N(x,z)\wt\Gg(y,x)d\Vg(y).
\end{align*} 
Since $\wt \Gg$ is continuous on $\M\times\M$, the use of Lemma \ref{LEM:PM} shows that the second term converges to $\wt\Gg(x,x)$ uniformly.

It remains to treat the first term. We first take a finite partition of unity $\{\chi_j\}$ such that $\chi_j$ are supported in balls of radii $\ll\i(\M)$, so that the exponential chart centred at some point in $\supp\chi_j$ entirely covers $\supp\chi_j$. Then, using Proposition \ref{PROP:pseudo} with Remark~\ref{REM:PDO}, we have
\begin{align*}
&\K_N(x,y)\log\big(\dg(y,x)\big)\\ &\qquad=\sum_j\Big[\big(a_j(0,N^{-1}D)+N^{-1}\wt a_j(0,N^{-1}D)\big)(\exp_x)_\star\wt\chi_j + R_{j,N}\Big]\log\big(\dg(\cdot,x)\big),
\end{align*}
where $\exp_x : T_x\M\simeq \R^2 \to \M$ is the exponential map at $x\in\M$, $\wt\chi_j\in C^{\infty}_0(\M)$ with $\wt\chi_j\equiv 1$ on $\supp \chi_j$, and $\|R_{j,N}\|_{H^{-s_2}(\M)\to H^{s_1}(\R^2)}\les N^{s_1+s_2-2}$ for any $s_1,s_2\ge 0$ with $s_1+s_2\le 2$. The the principal symbol of $\psi^2\big(-N^{-2}\Dlg\big)$ is
 \begin{align*}
 a_{j}(z,\xi)=\chi_j(\exp_x(z))\psi^2\big(N^{-2}\gm^{j,k}(\exp_x(z))\xi_j\xi_k\big),
 \end{align*}
with $\exp_x(0)=x$ and $\gm^{j,k}(0)=\dl_{j,k}$, and $\wt a_j\in \S^{-\infty}(\R^2\times\R^2)$ and is compactly supported in $z$.
 
This gives the decomposition
\begin{align*}
\K_N(x,y) &= \sum_j\frac{\chi_j(x)\wt\chi_j(y)}{(2\pi)^2}\int_{\R^2}e^{-i\exp_x^{-1}(y)\cdot\xi}\psi^2\big(N^{-2}|\xi|^2\big)d\xi + N^{-1}\wt \K_{j,N}+\K_{j,N}
\end{align*}
where $\wt\K_{j,N}$ is the kernel of $\wt a_j(z,N^{-1}D)$ and $K_{j,N}$ the one to $R_{j,N}$.

From Proposition \ref{PROP:pseudo}, we have
\begin{align*}
\big\|R_{j,N}\big\|_{L^2(\M)\to H^{1+\dl}(\M)} \les N^{\dl -1}.
\end{align*}
In particular this shows that
\begin{align*}
\Big\|\int_{\M}\K_{j,N}(x,y)\log\big(\dg(x,y)\big)d\Vg(y)\Big\|_{L^{\infty}(\M)}&\les \big\|R_{j,N}\log\big(\dg(\cdot,x)\big)(y)\big\|_{L^{\infty}_xH^{1+\dl}_y}\\
&\les N^{\dl-1}\big\|\log\big(\dg(y,x)\big)\big\|_{L^{\infty}_xL^2_y}\les N^{\dl-1},
\end{align*}
for any $0<\dl\ll 1$. 

As for $\wt\K_{j,N}$, we have by integrations by parts
\begin{align*}
\big|\wt \K_{j,N}(x,y)\big| &= \Big|\frac{\chi_j(x)\wt\chi_j(y)}{(2\pi)^2}\int_{\R^2}e^{-i\exp_x^{-1}(y)\cdot\xi}\wt a_j(0,N^{-1}\xi)d\xi\Big|\\
&\les \chi_j(x)\wt\chi_j(y)N^2\big|N\exp_x^{-1}(y)\big|^{-2}\int_{\R^2}\Big|\Dl_\xi\wt a_j(0,\xi)\Big|d\xi\\
&\les \chi_j(x)\wt\chi_j(y)\dg(x,y)^{-2}.
\end{align*}
Interpolating with the trivial bound $\big|\wt\K_{j,N}(x,y)\big|\les \chi_j(x)\wt\chi_j(y)N^2$, we get
\begin{align*}
\big|\wt\K_{j,N}(x,y)\big| \les \chi_j(x)\wt\chi_j(y) N^{\dl}\dg(x,y)^{\dl-2}
\end{align*}
for any $0<\dl\ll 1$. This shows that
\begin{align*}
N^{-1}\Big\|\int_{\M}\wt\K_{j,N}(x,y)\log\big(\dg(y,x)\big)d\Vg(y)\Big\|_{L^{\infty}(\M)} &\les N^{\dl-1} \int_{\M}\dg(x,y)^{\dl-2}\log\big(\dg(x,y)\big)d\Vg(y)\\
&\les N^{\dl-1}.
\end{align*}

Finally, to deal with the leading term, we have
\begin{align*}
&-\frac{\chi_j(x)}{2\pi}\int_{\M}\frac{\wt\chi_j(y)}{(2\pi)^2}\int_{\R^2}e^{-i\exp_x^{-1}(y)\cdot\xi}\psi^2(N^{-2}|\xi|^2)d\xi \log\big(\dg(x,y)\big)d\Vg(y)\\
&=-\frac{\chi_j(x)}{(2\pi)^3}\int_{\R^2}\int_{\R^2}e^{-iz\cdot\xi}\wt\chi_j(\exp_x(z))\psi^2(N^{-2}|\xi|^2)\log(|z|)|\gm(z)|^{\frac12}dzd\xi\intertext{Changing variables in $\xi$ then in $z$, we continue with}
&=-\frac{\chi_j(x)}{(2\pi)^3}\int_{\R^2}\int_{\R^2}e^{-iz\cdot\xi}\wt\chi_j(\exp_x(N^{-1}z))\psi^2(|\xi|^2)\log(N^{-1}|z|)|\gm(N^{-1}z)|^{\frac12}dzd\xi\\
&=-\frac{\chi_j(x)}{(2\pi)^3}\int_{\R^2}\int_{\R^2}e^{-iz\cdot\xi}\wt\chi_j(\exp_x(N^{-1}z))\psi^2(|\xi|^2)\Big[-\log N+\log(|z|)\Big]|\gm(N^{-1}z)|^{\frac12}dzd\xi\\
& = \1_j + \II_j.
\end{align*}

The first term is then given by
\begin{align*}
\1_j&=\frac1{2\pi}\log N \Bigg\{\frac{\chi_j(x)}{(2\pi)^2}\int_{\R^2}\int_{\R^2}e^{-iz\cdot\xi}\wt\chi_j(\exp_x(N^{-1}z))\psi^2(|\xi|^2)|\gm(N^{-1}z)|^{\frac12}d\xi dz\Bigg\}\\
&=\frac1{2\pi}\log N \Bigg\{\frac{\chi_j(x)}{(2\pi)^2}\int_{\R^2}\F\big[\psi^2(|\xi|^2)\big](z)\wt\chi_j(\exp_x(N^{-1}z))|\gm(N^{-1}z)|^{\frac12}dz\Bigg\}.
\end{align*}
Note that this last integral converges to $\wt\chi_j(x)$. Indeed, since $\psi\in\S(\R)$ and $\exp_x,\gm$ are smooth with $D\exp_x$ depending continuously on $x\in\M$, we have by the mean value theorem
\begin{align*}
&\Big\|\int_{\R^2}\F\big[\psi^2(|\xi|^2)\big](z)\Big[\wt\chi_j(\exp_x(N^{-1}z))|\gm|^{\frac12}(N^{-1}z)-\wt\chi_j(\exp_x(0))|\gm(0)|^{\frac12}\Big]dz\Big\|_{L^{\infty}(\M)}\\
&\qquad\les\sup_{x\in\M}\int_{\R^2}\jb{z}^{-10}N^{-1}|z|\Big|\int_0^1 \nabla_z\big[\wt\chi_j\circ\exp_x|\gm|^{\frac12}\big](\theta N^{-1}z)d\theta\Big|dz\\
&\qquad\les N^{-1}\int_{\R^2}|z|\jb{z}^{-10}dz \les N^{-1}.
\end{align*}
Thus, using that $\gm(0) = \Id$ with $\exp_x(0)=x$, that $\{\chi_j\}$ is a partition of unity with $\wt\chi_j\equiv 1$ on $\supp\chi_j$, and that $\psi(0)=1$, we have 
\begin{align*}
\sum_j\1_j &= \sum_j\frac{\chi_j(x)}{2\pi}\log N\Big\{\frac{\wt\chi_j(x)}{(2\pi)^2}\int_{\R^2}\F\big[\psi^2(|\xi|^2)\big](z)dz\Big\} +o(1)\\
&=\sum_j\frac{\chi_j(x)\psi^2(0)}{2\pi}\log N + o(1) = \frac1{2\pi}\log N + o(1),
\end{align*}
where $o(1) \to 0$ as $N\to\infty$, uniformly on $\M$.

Finally, the second term is
\begin{align*}
\II_j &= -\frac{\chi_j(x)}{(2\pi)^3}\int_{\R^2}\int_{\R^2}e^{-iz\cdot\xi}\wt\chi_j(\exp_x(N^{-1}z))\psi^2(|\xi|^2)\log(|z|)|\gm(N^{-1}z)|^{\frac12}dzd\xi,
\end{align*}
and the same argument as above gives
\begin{align*}
\sum_j\II_j = C_\psi + o(1) 
\end{align*}
uniformly on $\M$, with
\begin{align*}
C_\psi \deff -\frac1{(2\pi)^3}\int_{\R^2}\F\big[\psi^2(|\xi|^2)\big](z)\log(|z|)dz <\infty.
\end{align*}
This completes the proof of Lemma \ref{LEM:GN3}.
\end{proof}
\begin{remark}\rm\label{REM:kernels}~\\
(i) The condition $\psi(0)=1$ is necessary to have proper approximations of the identity, i.e. the corresponding family of smooth kernels $\K_N$ satisfy $\K_N(x,y)\to \dl_x(y)$ in the sense of distributions as $N\to\infty$, and $\int_{\M}\K_N(x,y)d\Vg(y) = 1$ for any $x\in\M$. In particular, note that Lemma \ref{LEM:GN3} holds for any kernel $\K_N(x,y) = \chi\big(N^2\dg(x,y)^2\big)$ with $\chi\in C^{\infty}_0\big((0,r)\big)$ and $0<r\ll \i(\M)^2$.\\
(ii) On the other hand, using an approximation by averaging on geodesic circles as in \cite{DKRV,DRV,GRV} formally corresponds to taking $\chi = \mathbf{1}_{\{|z|=1\}}$, which is not covered by our result. Note also that in this case $\F\big[\psi^2(|\xi|^2)\big] = \mathbf{1}_{\{|z|=1\}}$, which explains why $C_\psi = 0$ in these works.
\end{remark}
Collecting the previous results, we finally obtain the following key bound.
\begin{corollary}\label{COR:GN}
Let $\psi\in\S(\R)$ such that $\psi(0)=1$. Then there exists $C>0$ such that for any $N\in\N$ and $(x,y)\in\M\times\M\setminus\diag$ it holds
\begin{align}\label{GreenN}
\Big|(\psi\otimes\psi)\big(-N^{-2}\Dlg\big)\Gg(x,y) +\frac1{2\pi}\log\big(\dg(x,y)+N^{-1}\big)\Big| \le C.
\end{align}
\end{corollary}
\begin{proof}
Since 
\begin{align*}
-\frac1{2\pi}\log\big(\dg(x,y)+N^{-1}\big) = -\frac1{2\pi}\log\big(\max\big(\dg(x,y),N^{-1}\big)\big) + O(1),
\end{align*}
the bound \eqref{GreenN} follows from: (a) Lemma~\ref{LEM:GN2}~(i) with Lemma~\ref{LEM:Green}~(ii) in the case $\dg(x,y)\ge N^{-1}$; and (b) Lemma~\ref{LEM:GN2}~(ii) with Lemma~\ref{LEM:GN3} in the case $\dg(x,y)\le N^{-1}$. 
\end{proof}

At last, we look at the behaviour of different regularizations of $\Gg$.
\begin{lemma}\label{LEM:GN4}
Let $\psi_1,\psi_2\in\S(\R)$ be such that $\psi_1(0)=\psi_2(0)$. Then for any $0<\dl\ll 1$, there exists $C>0$ such that for any $N\ge 1$ and any $(x,y)\in\M\times\M\setminus\diag$ it holds
\begin{align}
&\Big|(\psi_1\otimes\psi_1)\big(-N^{-2}\Dlg\big)\Gg(x,y)-(\psi_2\otimes\psi_2)\big(-N^{-2}\Dlg\big)\Gg(x,y)\Big|\notag\\
&\qquad\qquad\le C\min\Big\{-\log\big(\dg(x,y)+N^{-1}\big)+ 1; N^{\dl-1}\dg(x,y)^{-1}\Big\}.
\label{GN3}
\end{align}
\end{lemma}
\begin{proof}
This is a straightforward adaptation of the proofs of  Lemmas \ref{LEM:GN1} and \ref{LEM:GN3}. First, the use of Corollary \ref{COR:GN} with the triangle inequality gives the first term in the right-hand side of \eqref{GN3}.

Next, defining $\psi(x)=\psi_1^2(x)-\psi_2^2(x)$ we have $\psi\in\S(\R)$ with $\psi(0)=0$, and
\begin{align*}
(\psi_1\otimes\psi_1)\big(-N^{-2}\Dlg\big)\Gg(x,y)-(\psi_2\otimes\psi_2)\big(-N^{-2}\Dlg\big)\Gg(x,y) = \sum_{n\ge 1}\psi(N^{-2}\ld_n^2)\frac{\varphi_n(x)\varphi_n(y)}{\ld_n^2}.
\end{align*} 

First, we replace $\ld_n^{-2}$ by $\jb{\ld_n}^{-2}$, since using \eqref{EF2} and that $\psi(0)=0$ with $\psi\in\S(\R)$, we have
\begin{align*}
&\Big\|\sum_{n\ge 1}\psi(N^{-2}\ld_n^2) \frac{\varphi_n(x)\varphi_n(y)}{\ld_n^2} - \sum_{n\ge 0}\psi(N^{-2}\ld_n^2)\frac{\varphi_n(x)\varphi_n(y)}{\jb{\ld_n}^2}\Big\|_{L^{\infty}(\M\times\M)}\\ &\les \sum_{n\ge 1}\psi(N^{-2}\ld_n^2)\frac1{\ld_n\jb{\ld_n}^2}\les \sum_{n\ge 1}(1\wedge N^{-2}\ld_n^2)\jb{N^{-1}\ld_n}^{-10}\frac1{\ld_n\jb{\ld_n}^2}\\
& \les N^{-1}\sum_{n\ge 1}\jb{N^{-1}\ld_n}^{-10}\frac1{\jb{\ld_n}^2} \sim N^{-1}\log N.
\end{align*}

As for the sum with $\jb{\ld_n}^{-2}$, it can be expressed as the kernel of $(1-\Dlg)^{-1}\psi\big(-N^{-2}\Dlg\big)$. This last operator can be expanded locally as
\begin{align*}
\kk^\star\Big(\chi (1-\Dlg)^{-1}\psi(-N^{-2}\Dlg)\Big) &= \big[a_{-2}(z,D)\kk^\star\wt\chi + R_{-3}\big]\psi(-N^{-2}\Dlg)\\
&=a_{-2}(z,D)\big[a_0(z,N^{-1}D)\kk^\star\wt\chi+R_{-1,N}\big] + R_{-3}\psi(-N^{-2}\Dlg),
\end{align*}
where $(U,V,\kk)$ is some chart on $\M$, $\chi,\wt\chi\in C^{\infty}_0(V)$ with $\wt\chi\equiv 1$ on $\supp \chi$. Moreover 
\begin{align*}
a_{-2}(z,\xi) = \chi(\kk^{-1}(z))\big(1+\gm^{i,j}(z)\xi_i\xi_j\big)^{-1}
\end{align*}
is the principal symbol of $(1-\Dlg)^{-1}$ in $\kk$, and 
\begin{align*}
a_0(z,\xi) = \wt\chi(\kk^{-1}(z))\psi(\gm^{i,j}(z)\xi_i\xi_j).
\end{align*}
The remainders satisfy the bounds
\begin{align*}
\big\|R_{-3}\big\|_{H^{s}(\M)\to H^{s+2}(\R^2)} \les 1
\end{align*}
for any $s\in\R$, and
\begin{align*}
\big\|R_{-1,N}\big\|_{H^{-s_1}(\M)\to H^{s_2}(\R^2)} \les N^{s_1+s_2-1},
\end{align*}
for any $s_1,s_2\ge 0$ with $s_1+s_2\le 1$.

Proceeding as for the classical composition rule for pseudo-differential operators, we can use Taylor's formula at order one to write the symbol of $a_{-2}(\kk^{-1}(x),D)a_0(y,N^{-1}D)$ as
 \begin{align*}
b(\kk^{-1}(x),\xi)&=\frac1{(2\pi)^2}\int_{\R^2}\int_{\R^2}e^{-iy\cdot\eta}a_{-2}(\kk^{-1}(x),\xi+\eta)a_0(\kk^{-1}(x)+y,N^{-1}\xi)d\eta dy\\
&= a_{-2}(\kk^{-1}(x),\xi)a_0(\kk^{-1}(x),N^{-1}\xi) + \wt R_{-3,N}(\kk^{-1}(x), \xi)
\end{align*}
with
\begin{align*}
&\wt R_{-3,N}(\kk^{-1}(x),\xi)\\& =  \frac1{(2\pi)^2}\int_{\R^2}\int_{\R^2}e^{-iz\cdot\eta}a_{-2}(\kk^{-1}(x),\xi+\eta)\int_0^1z\cdot \nabla_x a_0(\kk^{-1}(x)+\theta z,N^{-1}\xi)d\theta d\eta dz\\
&=\frac1{i(2\pi)^2}\int_{\R^2}\int_{\R^2}e^{-iz\cdot\eta}\nabla_{\eta}a_{-2}(\kk^{-1}(x),\xi+\eta)\cdot\int_0^1 \nabla_x a_0(\kk^{-1}(x)+\theta z,N^{-1}\xi)d\theta d\eta dz
\end{align*}
after integrating by parts in $\eta$.

Let us then take some charts $(U_1\times U_2,V_1\times  V_2,\kk_1\otimes\kk_2)$ in $\M\times\M$ containing $(x,y)$, and non negative $\chi_1\chi_2, \wt\chi_1\wt\chi_2\in C^{\infty}_0(V_1\times V_2)$ with $\chi_1(x) = 1 = \chi_2(y)$ and $\wt\chi_j\equiv 1$ on $\supp\chi_j$. We thus have
\begin{align*}
&\chi_1(x)\chi_2(y)\sum_{n\ge 0}\psi(N^{-2}\ld_n^2)\frac{\varphi_n(x)\varphi_n(y)}{\ld_n^2}\\
 &= \frac{\chi_1(x)\chi_2(y)}{(2\pi)^2} \int_{\R^2}e^{i(\kk_1^{-1}(x)-\kk_2^{-1}(y))\cdot\xi}a_{-2}(\kk_1^{-1}(x),\xi)a_0(\kk_1^{-1}(x),N^{-1}\xi)d\xi\\
&\qquad\qquad+K_{-1,N}(x,y) + K_{-3,N}(x,y) + \wt K_{-3,N}(x,y),
\end{align*}
where $K_{-1,N}$ is the kernel of $\chi_1(\kk_1)_\star a_{-2}(z,D)R_{-1,N}\chi_2$, $K_{-3,N}$ the one to $\chi_1(\kk_1)_\star R_{-3}\psi(-N^{-2}\Dlg)\chi_2$ and $\wt K_{-3,N}$ the one to $\chi_1(\kk_1)_\star \wt R_{-3,N}\chi_2$. 

In particular, we have
\begin{align*}
&\wt K_{-3,N}(x,y)\notag\\ &= \chi_1(x)\chi_2(y)\frac1{(2\pi)^2}\int_{\R^2}e^{i(\kk_1^{-1}(x)-\kk_2^{-1}(y))\cdot\xi}\wt R_{-3,N}(\kk^{-1}(x),\xi)d\xi\notag\\
&=\frac1{i(2\pi)^4}\int_{\R^2}\int_{\R^2}\int_{\R^2}e^{i\big[(\kk_1^{-1}(x)-\kk_2^{-1}(y))\cdot\xi-z\cdot\eta\big]}\nabla_\eta a_{-2}(\kk^{-1}(x),\xi+\eta)\notag\\
&\qquad\qquad\cdot \int_0^1\nabla_x a_0(\kk^{-1}(x)+\theta z,N^{-1}\xi)d\theta d\eta dz d\xi.
\end{align*}
We first integrate by parts in $\eta$ to get some decay in $z$: for any $A\in\N$ we then have
\begin{align*}
&\wt K_{-3,N}(x,y)\\
&=\chi_1(x)\chi_2(y)\frac1{i(2\pi)^4}\int_{\R^2}\int_{\R^2}\int_{\R^2}e^{i\big[(\kk_1^{-1}(x)-\kk_2^{-1}(y))\cdot\xi-z\cdot\eta\big]}\jb{z}^{-A}\nabla_\eta \jb{D_\eta}^Aa_{-2}(\kk^{-1}(x),\xi+\eta)\\
&\qquad\cdot \int_0^1\nabla_x a_0(\kk^{-1}(x)+\theta z,N^{-1}\xi)d\theta d\eta dz d\xi.
\end{align*}
Next, we integrate by parts in $z$ to get for any $B\in\N$
\begin{align*}
&\wt K_{-3,N}(x,y)\\
&=\chi_1(x)\chi_2(y)\frac1{i(2\pi)^4}\int_{\R^2}\int_{\R^2}\int_{\R^2}e^{i\big[(\kk_1^{-1}(x)-\kk_2^{-1}(y))\cdot\xi-z\cdot\eta\big]}\jb{\eta}^{-B}\nabla_\eta \jb{D_\eta}^Aa_{-2}(\kk^{-1}(x),\xi+\eta)\\
&\qquad\qquad\cdot \int_0^1\jb{D_z}^B\Big\{\jb{z}^{-A}\nabla_x a_0(\kk^{-1}(x)+\theta z,N^{-1}\xi)\Big\}d\theta d\eta dz d\xi.
\end{align*}
Then, using the definition of $a_{-2}$ and $a_0$ with the smoothness of $\gm$ and the compactness of $\M$, we get that
\begin{align*}
\big|\nabla_\eta \jb{D_\eta}^Aa_{-2}(\kk^{-1}(x),\xi+\eta)\big|\les \jb{\xi+\eta}^{-3}
\end{align*}
and
\begin{align*}
\Big|\jb{D_z}^B\Big\{\jb{z}^{-A}\nabla_x a_0(\kk^{-1}(x)+\theta z,N^{-1}\xi)\Big\}\Big|\les \jb{z}^{-A}N^{-2}|\xi|^2\jb{N^{-1}\xi}^{-C}
\end{align*}
for any $C>0$, where we also used that $\psi(0)=0$. Taking $A>2$ we deduce that
\begin{align*}
\big|\wt K_{-3,N}(x,y)\big| & \les \chi_1(x)\chi_2(y)\int_{\R^2}\int_{\R^2}\jb{\eta}^{-B}\jb{\xi+\eta}^{-3}N^{-2}|\xi|^2\jb{N^{-1}\xi}^{-C}d\eta d\xi\\
 &\les \chi_1(x)\chi_2(y)\int_{\R^2}\jb{\xi}^{-3}N^{-2}|\xi|^2\jb{N^{-1}\xi}^{-C}d\xi\\
&\les \chi_1(x)\chi_2(y)\Big\{N^{-2}\int_{|\xi|\les N}\jb{\xi}^{-1}d\xi + N^{C-2}\int_{|\xi|\gtrsim N}|\xi|^{-1-C}d\xi\Big\}\\
& \les N^{-1}\chi_1(x)\chi_2(y).
\end{align*}

To deal with the other remainder terms, we have from the properties of $a_{-2}$ and $R_{-1,N}$ that for $0<\dl\ll 1$,
\begin{align*}
\big\|a_{-2}(\kk^{-1}(x),D)R_{-1,N}\big\|_{H^{-1-\dl}(\M)\to H^{1+\dl}(\M)} \les \big\|R_{-1,N}\big\|_{H^{1-\dl}(\M)\to H^{1+\dl}(\R^2)} \les N^{2\dl -1},
\end{align*}
which implies that 
\begin{align*}
\big\|\chi_1(x)\chi_2(y)K_{-1,N}\big\|_{L^{\infty}(\M\times\M)}\les N^{2\dl-1}.
\end{align*}

Similarly,
\begin{align*}
\big\|R_{-3}\psi\big(-N^{-2}\Dlg\big)\big\|_{H^{-1-\dl}(\M)\to H^{1+\dl}(\R^2)} \les \big\|\psi\big(-N^{-2}\Dlg\big)\big\|_{H^{2-\dl}(\M)\to H^{1+\dl}(\M)} \les N^{2\dl-1}
\end{align*}
for $0<\dl\ll 1$, where we used that $\psi(0)=0$ in the last step. Along with the previous bounds, this implies that 
\begin{align*}
&\chi_1(x)\chi_2(y)\sum_{n\ge 0}\psi(N^{-2}\ld_n^2)\frac{\varphi_n(x)\varphi_n(y)}{\jb{\ld_n}^2}\\
 &= \frac{\chi_1(x)\chi_2(y)}{(2\pi)^2}\wt\chi_1(y)\int_{\R^2}e^{i(\kk_1^{-1}(x)-\kk_2^{-1}(y))\cdot\xi}a_{-2}(\kk^{-1}(x),\xi)a_0(\kk^{-1}(x),N^{-1}\xi)d\xi+ O(N^{2\dl-1}).
\end{align*}

Finally, to deal with the leading term, we see that its contribution is non trivial only if $V_1\cap V_2\neq \emptyset$, and in this case by taking $V_1\times V_2$ to be sufficiently small around $(x,y)$ we can choose $\kk_1=\kk_2=\exp_x$, so that $\gm^{j,k}(\kk_1^{-1}(x)) = \dl_{j,k}$ and $\kk_1^{-1}(x)=0$. From the expression of $a_{-2}$ and $a_0$ and integrating by parts this gives
\begin{align*}
&\Big|\frac1{(2\pi)^2}\int_{\R^2}e^{-i\exp_x^{-1}(y)\cdot\xi}\frac1{1+|\xi|^2}\psi(N^{-2}|\xi|^2) d\xi\Big|\\&\qquad\les \big|\exp_x^{-1}(y)\big|^{-2}\int_{\R^2}\Big|\Dl_\xi\Big[\frac{\psi(N^{-2}|\xi|^2)}{1+|\xi|^2}\Big]\Big|d\xi\\
& \qquad\les \dg(x,y)^{-2}\int_{\R^2}\Big[\jb{\xi}^{-4}N^{-2}|\xi|^2\jb{N^{-1}\xi}^{-10} + \jb{\xi}^{-2}N^{-2}\jb{N^{-1}\xi}^{-10}\Big]d\xi\\
&\qquad \les \dg(x,y)^{-2}N^{-2}\log(N),
\end{align*}
where we used that $\psi\in\S(\R)$ with $\psi(0)=0$. Interpolating with the trivial bound
\begin{align*}
\Big|\frac1{(2\pi)^2}\int_{\R^2}e^{-i\exp_x^{-1}(y)\cdot\xi}\frac1{1+|\xi|^2}\psi(N^{-2}|\xi|^2) d\xi\Big|\les \int_{\R^2}N^{-2}\jb{N^{-1}\xi}^{-10}d\xi \les 1
\end{align*}
yields \eqref{GN3}. This completes the proof of Lemma~\ref{LEM:GN4}.
\end{proof}

\subsection{Conformal change of metric}\label{SUBS:conf}

Recall that we fixed a smooth metric $\gm$ on $\M$ at the beginning of the section. We now invoke the uniformization theorem (see e.g. \cite[Section 8.8]{Aubin}) to get that there exists a smooth metric $\gm_0$ with constant curvature in the conformal class of $\gm$, i.e. 
\begin{align*}
\gm(x) = e^{f_0(x)}\gm_0(x)
\end{align*}
for some $f_0\in C^{\infty}(\M)$ and all $x\in\M$. Then the Laplace-Beltrami operator $\Dl_0$ and the (constant) Ricci scalar curvature $\Rg_0$ associated with $\gm_0$ satisfy
\begin{align}\label{conform}
\Dl_{\gm}u = e^{-f_0}\Dl_0 u\qquad\text{and}\qquad \Rg_{\gm}=e^{-f_0}\big(\Rg_0 -\Dl_0 f_0\big)
\end{align}
for any $u\in C^{\infty}(\M)$. In particular, the Sobolev space $H^1_0(\M)$ defined in \eqref{H0} is invariant under conformal change of the metric. Indeed, it holds for any $u\in C^{\infty}(\M)$
\begin{align*}
\int_{\M}\big|\nabla_\gm u\big|^2_\gm d\Vg = -\int_\M u\Dlg u d\Vg = -\int u e^{-f_0}\Dl_0ue^{f_0}dV_0 = \int_\M\big|\nabla_0u\big|_0^2dV_0.
\end{align*}

Recall also that for a two-dimensional Riemannian manifold, the Ricci scalar curvature is twice the Gaussian curvature, so that Gauss-Bonnet theorem  reads
\begin{align}\label{GB}
\int_{\M}\Rg_\gm d\Vg = 4\pi\chi(\M)=\int_{\M}\Rg_0dV_0= \Rg_0V_0(\M),
\end{align}
where $\chi(\M)$ is the Euler characteristic of $\M$ and $V_0$ is the volume form associated with $\gm_0$. The last two equalities indeed follow again from $dV_{\gm}=|\gm|^{\frac12}dx = e^{f_0}dV_0$ with \eqref{conform} and that $\Rg_0$ is constant.

We also consider the variation of the constant $\Xi=\Xi(\gm)$ in \eqref{Xi2}. It satisfies the so-called Polyakov formula (see \cite{OPS} or (2.10) in \cite{GRV})
\begin{align}\label{Xi3}
\Xi(\gm) = \exp\Big(\frac1{96\pi}\int_{\M}\big(|\nabla_0f_0|^2+\Rg_0f_0\big)dV_0\Big)\Xi(\gm_0).
\end{align}

Finally, up to replacing $\gm_0$ with $\wt\gm_0 = \frac{\Vg(\M)}{V_0(\M)}\gm_0$, whose curvature is also constant, we can then assume that 
\begin{align}\label{vol}
V_0(\M)=\Vg(\M).
\end{align}

We now investigate the relation between the Green's function $G_0$ associated with $\gm_0$ and $\Gg$.
\begin{lemma}\label{LEM:Gconf}
 For all $(x,y)\in\M\times\M\setminus\diag$,
\begin{align}\label{Gconf}
G_0(x,y) &= \Gg(x,y)-\jb{\Gg(x,\cdot)}_0 - \jb{\Gg(\cdot,y)}_0 + \jb{\jb{\Gg}_0}_0.
\end{align}
In particular, we have the equality in law
\begin{align}\label{Xconf}
X_0 = X_\gm - \jb{X_\gm}_0.
\end{align}
\end{lemma}
\begin{proof}
Take any $\chi\in C^{\infty}(\M)$ and define
\begin{align*}
u(x) = \int_{\M}\big[G_0(x,y)+\jb{\Gg(x,\cdot)}_0\big]\chi(y)d\Vg(y)
\end{align*}
and $v = u -\jb{u}_\gm$. Then using \eqref{mg}-\eqref{Green2}-\eqref{vol} we can compute
\begin{align*}
-\Dlg v &= -e^{-f_0}\Dl_0\int_{\M}G_0(x,y)\chi(y)e^{f_0(y)}dV_0(y)\\
&\qquad -\Dlg\frac1{V_0(\M)}\int_{\M}\Gg(x,z)e^{-f_0(z)}d\Vg(z)\int_{\M}\chi(y)d\Vg(y)\\
&= e^{-f_0}\big[\chi e^{f_0}-\jb{\chi e^{f_0}}_0\big](x) + \big[e^{-f_0} - \jb{e^{-f_0}}_\gm\big](x)\jb{\chi}_\gm\\
&= \chi(x) - \jb{\chi}_\gm.
\end{align*}
Using \eqref{Green2}, the previous computation then shows that
\begin{align*}
v(x) = \int_{\M}\Gg(x,y)\chi(y)d\Vg(y)
\end{align*}
for any $\chi\in C^{\infty}(\M)$, which in turn yields
\begin{align}\label{GG1}
 \Gg(x,y) = G_0(x,y)+\jb{\Gg(x,\cdot)}_0 - \jb{G_0(\cdot,y)}_\gm.
\end{align}
Integrating in $x$ with respect to $dV_0$ the previous identity, we find
\begin{align}
\jb{\Gg(\cdot,y)}_0 &= \jb{\jb{\Gg}_0}_0 - \jb{G_0(\cdot,y)}_\gm.
\label{GG2}
\end{align}
Plugging \eqref{GG2} in \eqref{GG1} gives \eqref{Gconf}.

 As for \eqref{Xconf}, it then follows from \eqref{Gconf} and the fact that the covariance function completely characterizes the Gaussian processes $X_\gm-\jb{X_\gm}_0$ and $X_0$, and for any $(x,y)\in\M\times\M\setminus\diag$ this latter is
\begin{align*}
\E \Big[\big(X_\gm(x) - \jb{X_\gm}_0\big)\big(X_\gm(y) - \jb{X_\gm}_0\big)\Big]
&= \Gg(x,y) - \jb{\Gg(x,\cdot)}_0 - \jb{\Gg(\cdot,y)}_0 + \jb{\jb{\Gg}_0}_0\\
&=G_0(x,y).
\end{align*}
This completes the proof of Lemma \ref{LEM:Gconf}.
\end{proof}
Next, we look at the corresponding relation for the truncated version $(\P_N\otimes\P_N)\Gg$. Indeed, recall that $\P_N = e^{N^{-2}\Dlg}$ implicitly depends on the metric as well. We then have the following result concerning regularizations with the more general class of multipliers $\psi\in\S(\R)$.
\begin{lemma}\label{LEM:GN5}
Let $\psi\in\S(\R)$ with $\psi(0)=1$. Then we have the following.\\
\textup{(i)} We have the uniform convergence
\begin{align}\label{sN2}
\Big\|(\psi\otimes\psi)(-N^{-2}\Dlg)G_0(x,x)-(\psi\otimes\psi)(-N^{-2}\Dl_0)G_0(x,x) - \frac1{4\pi} f_0(x)\Big\|_{L^{\infty}(\M)}\too 0
\end{align}
as $N\to\infty$;\\
\textup{(ii)} There exists $C>0$ such that for any $N\in\N$ and $(x,y)\in\M\times\M\setminus\diag$ ,
\begin{align}\label{GN5}
\Big|(\psi\otimes\psi)\big(-N^{-2}\Dlg\big)G_0(x,y) +\frac1{2\pi}\log\big(\dgg(x,y)+N^{-1}\big)\Big| \le C.
\end{align}
\end{lemma}
\begin{proof}
We only prove the first point, since \eqref{GN5} then follows by combining the arguments for \eqref{sN2} with a straightforward adaptation of the proofs of Lemma \ref{LEM:GN1}.

First, using Lemma \ref{LEM:Green}, we have
\begin{align*}
&(\psi\otimes\psi)\big(-N^{-2}\Dlg)G_0(x,x)-(\psi\otimes\psi)\big(-N^{-2}\Dl_0\big)G_0(x,x)\\ &= -\frac1{2\pi}(\psi\otimes\psi)\big(-N^{-2}\Dlg)\log\big(\dgg\big)(x,x)+\frac1{2\pi}(\psi\otimes\psi)\big(-N^{-2}\Dl_0\big)\log\big(\dgg)(x,x)\\
&\qquad + (\psi\otimes\psi)\big(-N^{-2}\Dlg)\wt G_0(x,x) - (\psi\otimes\psi)\big(-N^{-2}\Dl_0\big)\wt G_0(x,x).
\end{align*} 
Since $\wt G_0$ is continuous on $\M\times\M$, we have that the last two terms above both converge uniformly to $\wt G_0(x,x)$. Hence their contribution cancel each other, and we are left with estimating the contribution of the log terms.

Next, we observe that we can write for any $u\in \D'(\M)$
\begin{align*}
\psi(-N^{-2}\Dlg)u(x) = e^{-\frac12f_0(x)}\psi(-N^{-2}\wt\Dl_0)\big[e^{\frac12f_0}u\big](x).
\end{align*}
 Here $\wt\Dl_0 = e^{\frac12f_0}\Dl_0e^{-\frac12f_0}$. Indeed, the previous identity can be seen from \eqref{conform} by solving the eigenvalue problem for $\wt\Dl_0$:
\begin{align*}
\wt\Dl_0 \varphi = -\ld^2\varphi &\Longleftrightarrow e^{\frac12f_0}\Dlg \big[e^{-\frac12f_0}\varphi\big] = -\ld^2 \varphi \Longleftrightarrow \ld = \ld_n(\gm) \text{ and }\varphi = e^{\frac12f_0}\varphi_n(\gm),
\end{align*}
from which we deduce with the definition of the functional calculus that
\begin{align*}
\psi(-N^{-2}\Dlg)u&=\sum_{n\ge 1}\psi(N^{-2}\ld_n^2)\langle u,\varphi_n(\gm)\rangle_\gm \varphi_n(\gm) = \sum_{n\ge 1}\psi(N^{-2}\ld_n^2)\langle e^{\frac12f_0}u,e^{\frac12f_0}\varphi_n(\gm)\rangle_0 \varphi_n(\gm) \\
&= e^{-\frac12f_0}\sum_{n\ge 1}\psi(N^{-2}\ld_n^2)\langle e^{\frac12f_0}u,e^{\frac12f_0}\varphi_n(\gm)\rangle_0 e^{\frac12f_0}\varphi_n(\gm) \\&= e^{-\frac12f_0}\psi(-N^{-2}\wt\Dl_0)\big[e^{\frac12f_0}u\big].
\end{align*}

Next, using that $-\Dl_0$ and $-\wt\Dl_0$ are both elliptic and self-adjoint on $L^2(\M,\gm_0)$, we can use Proposition \ref{PROP:pseudo} to write the smoothing operators as
\begin{align*}
\psi(-N^{-2}\Dl_0) = \sum_j(\kk_j)_\star\Big\{\Big[\sum_{k=0}^2N^{-k}a_{j,k}(z,N^{-1}D)\Big]\kk_j^\star\wt\chi_j + (\kk_j^\star \chi_j)R_{j,N}\Big\}
\end{align*}
and
\begin{align*}
\psi(-N^{-2}\wt\Dl_0) = \sum_j(\kk_j)_\star\Big\{\Big[\sum_{k=0}^2N^{-k}\wt a_{j,k}(z,N^{-1}D)\Big]\kk_j^\star\wt\chi_j + (\kk_j^\star\chi_j)\wt R_{j,N}\Big\},
\end{align*}
where $\{\chi_j\}$ is a suitable partition of unity adapted to some charts $(U_j,V_j,\kk_j)$.

With the previous remarks we can write
\begin{align*}
&(\psi\otimes\psi)\big(-N^{-2}\Dlg)\log\big(\dgg\big)(x,x) \\&= \int_{\M}\int_{\M}e^{-f_0(x)}\wt\K_N(x,y)\wt\K_N(x,y')e^{\frac12 f_0(y)}e^{\frac12f_0(y')}\log\big(\dgg(y,y')\big)dV_0(y)dV_0(y')\intertext{where $\wt\K_N$ is the kernel of $\psi(-N^{-2}\wt\Dl_0)$. The previous expansion then allows us to decompose this term as}
&=\sum_j\chi_j(x)e^{-f_0(x)}\int_{\M}\int_{\M}\Big[\sum_{k=0}^2N^{-k}\wt\K_{j,k,N}(x,y) + \wt\K_{j,N}(x,y)\Big]\\
&\qquad\times\Big[\sum_{k'=0}^2N^{-k'}\wt\K_{j,k',N}(x,y') + \wt\K_{j,N}(x,y')\Big]e^{\frac12f_0(y)+\frac12f_0(y')}\log\big(\dgg(y,y')\big)dV_0(y)dV_0(y'),
\end{align*}
where $\wt\K_{j,k,N}$ is the kernel of $\chi_j(\kk_j)_\star\wt a_{j,k}(z,N^{-1}D)\kk_j^\star\wt\chi_j$ and $\wt\K_{j,N}$ the one to $\chi_j(\kk_j)_\star\wt R_{j,N}$.

 From the property of the remainders in Proposition \ref{PROP:pseudo}, we have
 \begin{align*}
 \|\wt R_{j,N}\|_{H^{-1+\dl}(\M)\to H^{1+\dl}(\R^2)} \les N^{2\dl-1}.
 \end{align*}
Hence we also have that $\wt\K_{j,N}$ is in $L^{\infty}(\M\times\M)$ with norm $O(N^{(-1)+})$. In particular, using also Lemma \ref{LEM:PM}, we have that
\begin{align*}
&\Big|\int_{\M}\int_{\M}\wt\K_N(x,y)\wt\K_{j,N}(x,y')e^{\frac12f_0(y)+\frac12f_0(y')}\log\big(\dgg(y,y')\big)dV_0(y)dV_0(y')\Big|\\
&\les N^{(-1)+}N^2\int_{\M}\int_{\M}\jb{N\dgg(x,y)}^{-10}\Big|\log\big(\dgg(y,y')\big)\Big|dV_0(y)dV_0(y') \les N^{(-1)+}.
\end{align*}
Thus it is enough to consider the contribution of the terms 
\begin{align*}
\int_{\M}\int_{\M}\wt\K_{j,k,N}(x,y)\wt\K_{j,k',N}(x,y')e^{\frac12f_0(y)+\frac12f_0(y')}\log\big(\dgg(y,y')\big)dV_0(y)dV_0(y').
\end{align*}

Note also that, repeating the argument as in the proof of Lemma \ref{LEM:PM}, these latter kernels also satisfy the bound
\begin{align}\label{Kj}
\big|\wt\K_{j,k,N}(x,y)\big|\les N^2\jb{N\dgg(x,y)}^{-10}.
\end{align}
Thus using the mean value theorem we can control for any $j,k,k'$
\begin{align*}
&\Big|\int_{\M}\int_{\M}\wt\K_{j,k,N}(x,y)\wt\K_{j,k',N}(x,y')\Big[e^{-f_0(x)}e^{\frac12f_0(y)+\frac12f_0(y')}-1\Big]\log\big(\dgg(y,y')\big)dV_0(y)dV_0(y')\Big|\\
&\qquad\les \int_{\M}\int_{\M}N^4\jb{N\dgg(x,y)}^{-10}\jb{N\dgg(x,y')}^{-10}\big[\dgg(x,y)+\dgg(x,y')\big]\\
&\qquad\qquad\times\big[1\vee\log\big(\frac1{\dgg(y,y')}\big)\big]dV_0(y)dV_0(y').
\end{align*}
To compute this last integral, we can assume by symmetry that $\dgg(x,y')\le \dgg(x,y)$. In the case $\dgg(x,y')\les \dgg(y,y')\sim\dgg(x,y)$ we then have
\begin{align*}
N^4\int_{\M}\int_{\M}\jb{N\dgg(x,y)}^{-10}\jb{N\dgg(x,y')}^{-10}\dgg(x,y)^{1-}dV_0(y)dV_0(y')\les N^{(-1)+}.
\end{align*}
In the other case $\dgg(y,y')\ll \dgg(x,y)\sim\dgg(x,y')$, we have
\begin{align*}
&N^4\int_{\M}\int_{\M}\jb{N\dgg(x,y)}^{-10}\dgg(x,y)\dgg(y,y')^{0-}dV_0(y)dV_0(y')\\
&\qquad \les N^4\int_{\M}\jb{N\dgg(x,y)}^{-10}\dgg(x,y)^{3-}dV_0(y) \les N^{(-1)+}.
\end{align*}

Thus we are left with estimating
\begin{align*}
-\frac1{(2\pi)}\sum_j\sum_{k,k'=0}^2\int_{\M}\int_{\M}N^{-k-k'}\wt\K_{j,k,N}(x,y)\wt\K_{j,k',N}(x,y')\log\big(\dgg(y,y')\big)dV_0(y)dV_0(y').
\end{align*}

In the case $k+k'\ge 1$, it is enough to use the rough bound \eqref{Kj} to estimate the integrals above with
\begin{align*}
N^{4-k-k'}\int_{\M}\int_{\M}\jb{N\dgg(x,y)}^{-10}\jb{N\dgg(x,y')}^{-10}\dgg(y,y')^{0-}dV_0(y)dV_0(y') \les N^{(-k-k')+}
\end{align*}
by proceeding similarly as for the previous integrals.

It remains to compute the leading term
\begin{align*}
-\frac1{2\pi}\sum_j\int_{\M}\int_{\M}\wt\K_{j,0,N}(x,y)\wt\K_{j,0,N}(x,y')\log\big(\dgg(y,y')\big)dV_0(y)dV_0(y').
\end{align*}

Up to taking the charts with a sufficiently small diameter, we can use $\kk_j = \exp_x$ on $V_j\ni x$,
so that $\gm_0(\kk_j^{-1}(x))=\Id$ and $\kk_j^{-1}(x)=0$. We also set $\theta = e^{-\frac12f_0(x)}>0$. We can then write the kernel of $(\kk_j)_\star \wt a_0(z,N^{-1}D)\kk_j^\star\wt\chi$ as
\begin{align*}
\wt \K_{j,0,N}(x,y)& = \frac{\wt\chi_j(y)}{(2\pi)^2}\int_{\R^2}e^{-i\exp_x^{-1}(y)\cdot \xi} \wt a_0(0,N^{-1}\xi)d\xi \\
&=\frac{\chi_j(x)\wt\chi_j(y)}{(2\pi)^2}\int_{\R^2}e^{-i\exp_x^{-1}(y)\cdot \xi} \psi\big(N^{-2}\theta^2|\xi|^2\big)d\xi ,
\end{align*}
where we used that the principal symbol of $-\Dl_0$ in $\kk_j$ is 
\begin{align*}
|\xi|_0^2 \deff \gm_0^{i,\ell}(x)\xi_i\xi_\ell
\end{align*}
so that the one of $-\wt\Dl_0$ is 
\begin{align*}
e^{-f_0(x)}|\xi|_0^2 = \theta^2|\xi|^2
\end{align*}
 by definition of $\theta$ and choice of $\kk_j$.
 
This gives
\begin{align*}
&-\frac1{2\pi}\sum_j\int_{\M}\int_{\M}\wt\K_{j,0,N}(x,y)\wt\K_{j,0,N}(x,y')\log\big(\dgg(y,y')\big)dV_0(y)dV_0(y')\\
&=-\sum_j\frac{\chi_j(x)}{(2\pi)^3}\int_{\R^2}\int_{\R^2}\int_{\R^2}\int_{\R^2}e^{-iz\cdot \xi} \psi\big(N^{-2}\theta^2|\xi|^2\big) e^{-iz'\cdot \xi'} \psi\big(N^{-2}\theta^2|\xi'|^2\big)\\
&\qquad\times\log\big(\dgg\big(\exp_x(z),\exp_x(z')\big)\big)\wt\chi_j(\exp_x(z))\wt\chi_j(\exp_x(z'))|\gm_0(z)|^{\frac12}|\gm_0(z')|^{\frac12}d\xi d\xi'dzdz'\\
&=-\sum_j\frac{\chi_j(x)}{(2\pi)^3}\int_{\R^2}\int_{\R^2}\int_{\R^2}\int_{\R^2}e^{-iz\cdot \xi} \psi\big(N^{-2}|\xi|^2\big) e^{-iz'\cdot \xi'} \psi\big(N^{-2}|\xi'|^2\big)\\
&\qquad\times\log\big(\dgg\big(\exp_x(\theta z),\exp_x(\theta z')\big)\big)\wt\chi_j(\exp_x(\theta z))\wt\chi_j(\exp_x(\theta z'))|\gm_0(\theta z)|^{\frac12}|\gm_0(\theta z')|^{\frac12}d\xi d\xi'dzdz'
\end{align*}
 after changing variables in both $\xi,\xi'$ and $z,z'$.
 
 Note that $\dgg^2$ is smooth on the support of $\wt\chi_j(y)\wt\chi_j(y')$, and we have the Taylor expansion (see e.g. \cite[Appendix A]{Nico})
 \begin{align*}
\dgg^2\big(\exp_x(z),\exp_x(z')\big)  = |z-z'|^2 + O\big(|z-z'|^2(|z|+|z'|)^2\big).
 \end{align*}
 In particular, we get that
 \begin{align*}
& \bigg|\int_{\R^2}\int_{\R^2}\int_{\R^2}\int_{\R^2}e^{-iz\cdot \xi} \psi\big(N^{-2}|\xi|^2\big) e^{-iz'\cdot \xi'} \psi\big(N^{-2}|\xi'|^2\big)\wt\chi_j(\exp_x(\theta z))\wt\chi_j(\exp_x(\theta z'))\\
&\qquad\qquad\times\Big[\log\big(\dgg\big(\exp_x(\theta z),\exp_x(\theta z')\big)\big)-\log\big(\theta|z-z'|\big)\Big]|\gm_0(\theta z)|^{\frac12}|\gm_0(\theta z')|^{\frac12}d\xi d\xi'dzdz'\bigg|\\
&\qquad\les \int_{\R^2}\int_{\R^2}N^2\jb{Nz}^{-10}N^2\jb{Nz'}^{-10}\frac{|z-z'|^2(|z|^2+|z'|^2)}{|z-z'|^2}dzdz' \les N^{-2},
 \end{align*}
 where we used the Taylor expansion above after changing variables in $\xi,\xi'$ and doing integration by parts similarly as in the proof of Lemma \ref{LEM:PM}.
 
To summarize, we have proved so far
\begin{align*}
&(\psi\otimes\psi)\big(-N^{-2}\Dlg)\log\big(\dgg\big)(x,x) \\
&= -\sum_j\frac{\chi_j(x)}{(2\pi)^3}\int_{\R^2}\int_{\R^2}\int_{\R^2}\int_{\R^2}e^{-iz\cdot \xi} \psi\big(N^{-2}|\xi|^2\big) e^{-iz'\cdot \xi'} \psi\big(N^{-2}|\xi'|^2\big)\\
&\qquad\times\log\big(\theta|z-z'|\big)\wt\chi_j(\exp_x(\theta z))\wt\chi_j(\exp_x(\theta z'))|\gm_0(\theta z)|^{\frac12}|\gm_0(\theta z')|^{\frac12}d\xi d\xi'dzdz' + O(N^{(-1)+}),
\end{align*}
uniformly on $\M$. We also remark that exactly the same computations give
\begin{align*}
&(\psi\otimes\psi)\big(-N^{-2}\Dl_0)\log\big(\dgg\big)(x,x) \\
&= -\sum_j\frac{\chi_j(x)}{(2\pi)^3}\int_{\R^2}\int_{\R^2}\int_{\R^2}\int_{\R^2}e^{-iz\cdot \xi} \psi\big(N^{-2}|\xi|^2\big) e^{-iz'\cdot \xi'} \psi\big(N^{-2}|\xi'|^2\big)\\
&\qquad\times\log\big(|z-z'|\big)\wt\chi_j(\exp_x(z))\wt\chi_j(\exp_x(z'))|\gm_0(z)|^{\frac12}|\gm_0(z')|^{\frac12}d\xi d\xi'dzdz' + O(N^{(-1)+}).
\end{align*}
Changing variables again, we thus have
\begin{align*}
&(\psi\otimes\psi)\big(-N^{-2}\Dlg)\log\big(\dgg\big)(x,x)-(\psi\otimes\psi)\big(-N^{-2}\Dl_0)\log\big(\dgg\big)(x,x)\\
&=-\log(\theta)\sum_j\frac{\chi_j(x)}{(2\pi)^3}\int_{\R^2}\int_{\R^2}\int_{\R^2}\int_{\R^2}e^{-iz\cdot \xi} \psi\big(|\xi|^2\big) e^{-iz'\cdot \xi'} \psi\big(|\xi'|^2\big)\\
&\qquad\qquad\times\wt\chi_j(\exp_x(N^{-1}\theta z))\wt\chi_j(\exp_x(N^{-1}\theta z'))|\gm_0(N^{-1}\theta z)|^{\frac12}|\gm_0(N^{-1}\theta z')|^{\frac12}d\xi d\xi'dzdz'\\
&\qquad-\sum_j\frac{\chi_j(x)}{(2\pi)^3}\int_{\R^2}\int_{\R^2}\int_{\R^2}\int_{\R^2}e^{-iz\cdot \xi} \psi\big(N^{-2}|\xi|^2\big) e^{-iz'\cdot \xi'} \psi\big(N^{-2}|\xi'|^2\big)\log\big(N^{-1}|z-z'|\big)\\
&\qquad\qquad\times\Big[\wt\chi_j(\exp_x(N^{-1}\theta z))\wt\chi_j(\exp_x(N^{-1}\theta z'))|\gm_0(N^{-1}\theta z)|^{\frac12}|\gm_0(N^{-1}\theta z')|^{\frac12}\\
&\qquad\qquad\qquad-\wt\chi_j(\exp_x(N^{-1} z))\wt\chi_j(\exp_x(N^{-1} z'))|\gm_0(N^{-1} z)|^{\frac12}|\gm_0(N^{-1} z')|^{\frac12}\Big]d\xi d\xi'dzdz'\\
&\qquad + O(N^{(-1)+})\\
&= \1+\II+O(N^{(-1)+}).
\end{align*}
 
 For the first term, the same argument as in the proof of Lemma \ref{LEM:GN3} shows that 
 \begin{align*}
\1\to -\frac{\log(\theta)}{2\pi} = \frac1{4\pi}f_0(x)
\end{align*}
as $N\to\infty$, uniformly on $\M$.

Finally, for the second term, integrating by parts in $\xi,\xi'$ and using the mean value theorem gives the estimate
\begin{align*}
\big|\II\big| &\les \int_{\R^2}\int_{\R^2}N^2\jb{Nz}^{-10}N^2\jb{Nz'}^{-10}N^{0+}|z-z'|^{0-}|\theta-1|N^{-1}\big(|z|+|z'|\big)dzdz'\\
& \les N^{(-2)+}.
\end{align*}
  
  This concludes the proof of Lemma \ref{LEM:GN5}.

\end{proof}

\subsection{Some nonlinear estimates}
To conclude this section, we state some useful estimates needed in the analysis of \eqref{heat2}. We first recall the following embeddings and product estimates in Besov spaces proved in \cite[Corollary 2.7]{ORTz}.
\begin{lemma}\label{LEM:Besov}
Let $B^s_{p,q}(\M)$ be the Besov spaces defined above. Then the following properties hold.\\
\textup{(i)} For any $s\in\R$ we have $B^s_{2,2}(\M)=H^s(\M)$, and more generally for any $2\le p <\infty$ and $\eps>0$ we have
\begin{align*}
\|u\|_{B^s_{p,\infty}(\M)}\les \|u\|_{W^{s,p}(\M)}\les \|u\|_{B^s_{p,2}(\M)} \les \|u\|_{B^{s+\eps}_{p,\infty}(\M)}.
\end{align*}
\textup{(ii)} Let $s\in\R$ and $1\leq p_1\leq p_2\leq \infty$ and $q_1,q_2\in [1,\infty]$. Then for any $f\in B^s_{p_1,q}(\M)$ we have
\begin{align*}
\|f\|_{B^{s-2\big(\frac1p_1-\frac1p_2\big)}_{p_2,q}(\M)}\les \|f\|_{B^s_{p_1,q}(\M)}.
\end{align*}
\textup{(iii)} Let $\alpha,\beta\in \R$ with $\alpha+\beta>0$ and $p_1,p_2,q_1,q_2\in [1,\infty]$ with
\begin{align*}
\frac1p = \frac1p_1+\frac1p_2\qquad\text{ and }\qquad\frac1q=\frac1q_1+\frac1q_2.
\end{align*}
 Then for any $f\in B^{\alpha}_{p_1,q_1}(\M)$ and $g\in B^{\beta}_{p_2,q_2}(\M)$, we have $fg \in B^{\alpha\wedge \beta}_{p,q}(\M)$, and moreover it holds
\begin{align*}
\|fg\|_{B^{\alpha\wedge\beta}_{p,q}(\M)}\les \|f\|_{B^{\alpha}_{p_1,q_1}(\M)}\|g\|_{B^{\beta}_{p_2,q_2}(\M)}.
\end{align*}
\end{lemma}

In order to take advantage of the ``sign-definite structure" used in \cite{ORW}, we also need the following product estimate.
\begin{lemma}\label{LEM:posprod}
Let $s\ge 0$ and $1< p<\infty$. Then it holds
\begin{align}\label{posprod}
\big\|fg\big\|_{W^{-s,p}(\M)}\les \|f\|_{L^\infty(\M)}\|g\|_{W^{-s,p}(\M)}
\end{align}
for any $f\in C(\M)$ and any positive distribution $g\in W^{-s,p}(\M)$.
\end{lemma}
\begin{proof}
We merely repeat the argument of \cite[Lemma 2.14]{ORW}. We consider $s>0$ since the case $s=0$ follows directly from H\"older's inequality. 

  Since $g$ is a positive distribution, it can be identified with a positive Radon measure on $\M$;
  see~\cite{Folland}.
If $f\in C(\M)$, then the product $fg$ is a well-defined function in $L^1(\M)$; in particular from Lemma \ref{LEM:heatker} (i) we have $\P_Nf \to f$ in $C(\M)$ as $N\to\infty$, so that $fg = \lim_{N\to\infty} (\P_Nf)g$ in $L^1(\M)$. Moreover, for any $M\in\N$, $\P_Mg$ is also a smooth positive distribution which converges to $g$ in $W^{-s,p}(\M)$, so that from Lemma \ref{LEM:Besov} (i) and (iii) we have that for each fixed $N\in\N$, $(\P_N f)(\P_Mg)$ converges to $(\P_N f)g$ in $L^1(\M)$. Thus from Fatou's lemma we get
\begin{align}
\|(1-\Dlg)^{-\frac{s}2}(fg)\|_{L^p(\M)}\le \liminf_{N\to\infty}\lim_{M\to\infty}\|(1-\Dlg)^{-\frac{s}2}\big[(\P_Nf)(\P_Mg)\big]\big\|_{L^p(\M)}.
\label{pos1}
\end{align}
It thus remains to prove the product estimate \eqref{posprod} for the right-hand side of \eqref{pos1}. Note that since $g$ is a positive distribution we also have by Lemma \ref{LEM:heatker} (i) and the definition of $\P_M$ \eqref{PN} that $\P_Mg$ is a non-negative function for any $M\in\N$.

Let then $K_s$ be the distributional kernel of $(1-\Dlg)^{-\frac{s}2}$: we can write
\begin{align}
K_s(x,y) &= \sum_{n\ge 0}\frac{\varphi_n(x)\varphi_n(y)}{\jb{\ld_n}^s} = \sum_{n\ge 0}\varphi_n(x)\varphi_n(y)\Gamma\Big(\frac{s}2\Big)^{-1}\int_0^{\infty} t^{\frac{s}2-1}e^{-t\jb{\ld_n}^2}dt\notag\\
&= \Gamma\Big(\frac{s}2\Big)^{-1}\int_0^{\infty} t^{\frac{s}2-1}e^{-t}\Pg(4\pi t,x,y)dt,
\label{fractional}
\end{align}
where $\Gamma$ is the Gamma function, and the equality holds in the sense of distributions on $\M\times\M$. Note that the last integral converges since we assumed $s>0$. From Lemma \ref{LEM:heatker} (i) we then deduce that $K_s$ is also a positive distribution. This implies that
\begin{align*}
\big\|(1-\Dlg)^{-\frac{s}2}\big[(\P_Nf)(\P_Mg)\big]\big\|_{L^p(\M)} &= \Big\|\int_{\M} K_s(x,y)\P_Nf(y)\P_Mg(y)d\Vg(y)\Big\|_{L^p(\M)}\\
&\le\Big\|\|\P_Nf\|_{L^{\infty}(\M)}\int_{\M} K_s(x,y)\P_Mg(y)d\Vg(y)\Big\|_{L^p(\M)}\\
& = \|\P_Nf\|_{L^{\infty}(\M)}\|\P_Mg\|_{W^{-s,p}(\M)}.
\end{align*}
This allows to bound the right-hand side of \eqref{pos1}, and taking the limits $M\to\infty$ then $N\to\infty$ finally proves \eqref{posprod}. This completes the proof of Lemma~\ref{LEM:posprod}.
\end{proof}

We will also need a fractional chain rule for the composition with Lipschitz functions. Note that the fractional chain rule for fractional derivatives defined by the right-hand side of \eqref{fractional} for $-2<s<0$ is known to hold on any space of homogeneous type \cite{Gatto}. 
 
In our case (see Section \ref{SEC:GWP} below) we need it for a Lipschitz nonlinearity acting on functions in Besov spaces. We thus have the following estimates.
\begin{lemma}\label{LEM:FC}
Let $\A : \R\to\R$ with $\A(0)=0$ and $0<s<1$. Then:\\
\textup{(i)}if $\A$ is Lipschitz, it holds
\begin{align*}
\big\|\A(u)\big\|_{H^s(\M)} \les \|\A'\|_{L^{\infty}}\|u\|_{H^s(\M)};
\end{align*}

\noi
\textup{(ii)} if $\A$ is $C^1$ and there exists $a\in L^1([0,1])$, $\AB$ positive such that for almost every $\theta\in [0,1]$ and any $u,v$ it holds
\begin{align*}
\big|\A'\big(\theta u + (1-\theta)v\big)\big| \le a(\tau)\big[\AB(u)+\AB(v)\big],
\end{align*}
then for any $1<p<\infty$ and any  $r\in [1,\infty]$ it holds for any $0<\eps<1-s$
\begin{align*}
\big\|\A(u)\big\|_{B^s_{p,r}(\M)} \les \big\|\AB(u)\big\|_{L^{\infty}(\M)}\|u\|_{B^{s+\eps}_{p,r}(\M)}.
\end{align*}
\end{lemma}
\begin{proof}
Let $(U_j,V_j,\kk_j)_j$ be a finite atlas on $\M$ and $\{\chi_j\}_j$ be an associated partition of unity. Using Lemma \ref{LEM:Beloc}, we thus have
\begin{align*}
\big\|\A(u)\big\|_{H^s(\M)} & \les \max_j\big\|\kk_j^\star\big(\chi_j\A( u)\big)\big\|_{H^s(\R^2)}=\max_j\big\|\kk_j^\star(\chi_j)\A\big(\kk_j^\star\wt\chi_j u\big)\big\|_{H^s(\R^2)}\intertext{where $\wt\chi_j\in C^{\infty}_0(V_j)$, $\wt\chi_j\equiv 1$ on $\supp\chi_j$. Then we can use first the same product estimate as in Lemma \ref{LEM:Besov} (iii), which holds on $\R^2$ \cite{BCD}, and then the fractional chain rule for Lipschitz functions on $\R^2$ (see e.g. \cite[Proposition 4.1]{Taylor}), to continue with}
&\les \max_j \|\kk_j^\star(\chi_j)\|_{B^1_{\infty,2}(\R^2)}\big\|\A\big(\kk_j^\star\wt\chi_j u\big)\big\|_{H^s(\R^2)} \les \max_j \|\A'\|_{L^{\infty}}\big\|\kk_j^\star\wt\chi_j u\big\|_{H^s(\R^2)}\\
& \les \|\A'\|_{L^{\infty}}\|u\|_{H^s(\M)}
\end{align*}
where in the last step we used Lemma \ref{LEM:Beloc} again. This proves (i). The second estimate (ii) follows similarly\footnote{The $\eps$ loss comes from moving from $B^s_{p,r}(\R^2)$ to $W^{s+\eps,p}(\R^2)$ in order to use the fractional chain rule on $\R^2$, and can certainly be avoided, by using for example an argument similar to \cite[Lemma 9.3]{STz2}.} by using \cite[Proposition 5.1]{Taylor}. 
\end{proof}

We finish this section by stating another nonlinear (multilinear) estimate, namely the (geometric) Brascamp-Lieb inequality \cite[Example 1.6]{BCCT}.
\begin{lemma}\label{LEM:BL}
Let $p \in \N$ and $f_{q,r}\in L^1(\M\times\M)$ for $1\le q<r\le 2p$. Then it holds
\begin{align}
\begin{split}
 \int_{\M^{2p}}   
 \prod_{1\le q < r \le 2p}   
&  |f_{q,r}(\pi_{q, r} (\vec{y}))|^{\frac1{2p-1}}
d\Vg(\vec{y}) \\
 & \les   \prod_{1\le q < r \le 2p}   
 \bigg( \int_{\M\times\M}  
  |f_{q,r}(y_q,y_r)|  d\Vg(y_q) d\Vg(y_r) \bigg)^{\frac1{2p-1}},
\end{split}
\label{BL0}
\end{align}

\noi
where $\vec{y}=(y_1,...,y_{2p})\in \M^{2p}$, $d\Vg(\vec{y})$ is the corresponding product measure on $\M^{2p}$, and  $\pi_{q, r}$
denotes the projection defined by 
$\pi_{q, r}(\vec{y}) = \pi_{q, r}(y_1, \dots, y_{2p})
= (y_q, y_r) $.
\end{lemma}
\begin{proof}
The corresponding estimate on $\M = \R^d$ is proved in greater generality in \cite{BCCT}. The estimate \eqref{BL0} then follows from (1) in \cite{BCCT} by using a partition of unity with the compactness of $\M$, and that in each chart $\Vg$ is equivalent to the Lebesgue measure since $\gm$ is smooth.
\end{proof}

\section{GMC theory and the LQG measure}\label{SEC:Proba}
In this section, we deal with the construction of the main stochastic objects, namely the linear stochastic evolution $\<1>_\gm$ in \eqref{lp}, and the ``punctured" Gaussian multiplicative chaos $\U$ in \eqref{U} and the LQG measure $\rho_{\{a_\ell,x_\ell\},\gm}$ in \eqref{LQGN}. We mainly follow the arguments of \cite{ORW,DKRV,DRV,GRV}.
\subsection{On punctured Gaussian multiplicative chaos}\label{SUBS:GMC}
Before moving to the proof of Proposition \ref{PROP:U}, we first express the covariance function of the processes $\P_N\<1>_\gm$ in \eqref{lp} and $\U_N(t)$ \eqref{U}.
\begin{lemma}\label{LEM:covar}
The following identities hold:\\
\textup{(i)} The covariance function of the truncated linear stochastic evolution $\<1>_\gm$ is given by
\begin{align}\label{covar}
\G_{N_1,N_2}(t_1,t_2,x_1,x_2)&\deff\int_{H^s_0(\M,\gm)}\E\big[\P_{N_1}\<1>_\gm(t_1,x_1)\P_{N_2}\<1>_\gm(t_2,x_2)\big]d\mu_\gm(X_\gm)\notag\\
& = 2\pi e^{\frac{t_2-t_1}{4\pi}\Dlg}(\P_{N_1}\otimes\P_{N_2})\Gg(x_1,x_2),
\end{align}
for any $(x_1,x_2)\in\M\times\M\setminus\diag$, $t_1\le t_2$ and $N_1,N_2\in\N$;\\
\textup{(ii)} For any $t\ge 0$ and $p\in\N$, the $2p$ (spatial) covariance function of $\U_N(t)$ is given by
\begin{align*}
\int_{H^s_0(\M,\gm)}\E\Big[\prod_{j=1}^{2p}\U_N(t,y_j)\Big]d\mu_\gm(X_\gm)& = e^{\pi\be^2\sum_{j=1}^{2p}\wt\Gg(y_j,y_j) + o(1)}e^{2\pi\be^2 \sum_{j<k}(\P_N\otimes\P_N)\Gg(y_j,y_k)}\\
&\qquad\times e^{2\pi\be\sum_{j=1}^{2p}\sum_{\ell=1}^La_\ell(\P_N\otimes\P_N)\Gg(x_\ell,y_j)},
\end{align*}
where $o(1)$ is deterministic and uniform on $\M$.
\end{lemma}
\begin{proof}
Writing $X_\gm$ as the random variable \eqref{GFF} with $n_n\sim\NN(0,1)$ independent of $B_{n,\gm}$ in \eqref{Bn}, we first compute
\begin{align*}
&\G_{N_1,N_2}(t_1,t_2,x_1,x_2)\\
&=\int_{H^s_0(\M,\gm)}\E\Big[\P_{N_1}\Big(e^{\frac{t_1}{4\pi}\Dlg}X_\gm(x_1)+\int_0^{t_1}e^{\frac{t_1-t}{4\pi}\Dlg}dW_\gm(t,x_1)\Big)\\
&\qquad\times\P_{N_2}\Big(e^{\frac{t_2}{4\pi}\Dlg}X_\gm(x_2)+\int_0^{t_2}e^{\frac{t_2-t}{4\pi}\Dlg}dW_\gm(t,x_2)\Big)\Big]d\mu_\gm(X_\gm)\\
&=\sum_{n_1,n_2\ge 1}e^{-N_1^{-2}\ld_{n_1}^2}e^{-N_2^{-2}\ld_{n_2}^2}\varphi_{n_1}(x_1)\varphi_{n_2}(x_2)\E\Big[\Big(e^{-\frac{t_1}{4\pi}\ld_{n_1}^2}\frac{\sqrt{2\pi}h_{n_1}}{\ld_{n_1}}+\int_0^{t_1}e^{-\frac{t_1-t}{4\pi}\ld_{n_1}^2}dB_{n_1,\gm}(t)\Big)\\
&\qquad\times\Big(e^{-\frac{t_2}{4\pi}\ld_{n_2}^2}\frac{\sqrt{2\pi}h_{n_2}}{\ld_{n_2}}+\int_0^{t_1}e^{-\frac{t_2-t}{4\pi}\ld_{n_2}^2}dB_{n_2,\gm}(t)\Big)\Big]\\
&=\sum_{n\ge 1}e^{-N_1^{-2}\ld_{n}^2}e^{-N_2^{-2}\ld_{n}^2}\varphi_{n}(x_1)\varphi_{n}(x_2)\Big[\frac{2\pi}{\ld_n^2} e^{-\frac{t_1+t_2}{4\pi}\ld_n^2}+\int_0^{\min(t_1,t_2)}e^{-\frac{t_1+t_2-2t}{4\pi}\ld_n^2}dt\Big]
\end{align*}
where we used that $h_n\sim\NN(0,1)$ are independent, and independent of the standard Brownian motions $B_{n,\gm}$, with Ito's isometry. The last integral can be computed as
\begin{align*}
\int_0^{\min(t_1,t_2)}e^{-\frac{t_1+t_2-2t}{4\pi}\ld_n^2}dt &= \frac{2\pi}{\ld_n^2}\big[e^{\frac{\min(t_1,t_2)-\max(t_1,t_2)}{4\pi}\ld_n^2}-e^{-\frac{t_1+t_2}{4\pi}\ld_n^2}\big],
\end{align*}
and plugging this identity in the previous one, we end up with
\begin{align*}
\G_{N_1,N_2}(t_1,t_2,x_1,x_2)&=\sum_{n\ge 1}e^{-N_1^{-2}\ld_{n}^2}e^{-N_2^{-2}\ld_{n}^2}\varphi_{n}(x_1)\varphi_{n}(x_2)\frac{2\pi}{\ld_n^2}e^{\frac{\min(t_1,t_2)-\max(t_1,t_2)}{4\pi}\ld_n^2}.
\end{align*}
This shows \eqref{covar} in view of \eqref{Green} and \eqref{PN}.

As for Lemma~\ref{LEM:covar}~(ii), note that by Lemma \ref{LEM:GN3} we have
\begin{align*}
e^{-\pi\be^2C_\P}N^{-\frac{\be^2}2} = e^{-\frac{\be^2}2\s_N(x)+\pi\be^2\wt\Gg(x,x)+o(1)}
\end{align*}
for any $x\in\M$, where $\s_N$ is as in \eqref{sN}.

Thus, with the definition \eqref{U} of $\U_N$, we can compute for any $p,N\in\N$, $t\ge 0$ and $y_1,...,y_{2p}\in\M$:
\begin{align*}
&\int_{H^s_0(\M,\gm)}\E\Big[\prod_{j=1}^{2p}\U_N(t,y_j)\Big]d\mu_\gm(X_\gm)\\
& = \int_{H^s_0(\M,\gm)}\E\Big[e^{-2p\pi\be^2C_\P}N^{-2p\frac{\be^2}2}e^{\sum_{j=1}^{2p}\big[\be \P_N\<1>_\gm(t,y_j)+2\pi\be\sum_{\ell=1}^L a_\ell(\P_N\otimes\P_N)\Gg(x_\ell,y_j)\big]}\Big]d\mu_\gm(X_\gm)\\
&=e^{\sum_{j=1}^{2p}\big[-\frac{\be^2}2\s_N(y_j)+\pi\be^2\wt\Gg(y_j,y_j)+o(1)\big]}e^{2\pi\be \sum_{j=1}^{2p}\sum_{\ell=1}^La_\ell(\P_N\otimes\P_N)\Gg(x_\ell,y_j)}e^{\frac{\be^2}2\E\big|\sum_{j=1}^{2p}\P_N\<1>_\gm(t,y_j)\big|^2}\intertext{where we used the prevous identity and that $(\P_N\<1>_\gm(t,y_1),...,\P_N\<1>_\gm(t,y_{2p}))$ is a Gaussian vector. Thus using \eqref{covar} and the definition of $\s_N$ \eqref{sN}, we continue with}
&=e^{\sum_{j=1}^{2p}\pi\be^2\wt\Gg(y_j,y_j)+o(1)}e^{2\pi\be \sum_{j=1}^{2p}\sum_{\ell=1}^La_\ell(\P_N\otimes\P_N)\Gg(x_\ell,y_j)}e^{2\pi\be^2\sum_{j<k}(\P_N\otimes\P_N)\Gg(y_j,y_k)}.
\end{align*}
This finishes the proof of Lemma~\ref{LEM:covar}~(ii).
\end{proof}

For the solution $\<1>_\gm$ of the linear stochastic heat equation, we have the following classical result.
\begin{lemma}\label{LEM:SC}
The process $\<1>_\gm$ defined in \eqref{lp} is stationary and belongs almost surely to $C(\R_+;H^s_0(\M))$ for any $s<0$.
\end{lemma}
\begin{proof}
Since 
\begin{align*}
\P_N\<1>_\gm = e^{\frac{t}{4\pi}\Dlg}(\P_NX_\gm) + \int_0^te^{\frac{t-t'}{4\pi}\Dlg}d(\P_NW_\gm)(t').
\end{align*}
 is Gaussian and stationary in view of Lemma~\ref{LEM:covar}~(i), it is enough to construct $\<1>_\gm$ as the limit of $\P_N\<1>_\gm$. Take any $s<0$, $p\ge 2$, $t,\dl>0$ and $N_1,N_2\in\N$. Write 
\begin{align*}
\<1>_{N_1-N_2} \deff (\P_{N_1}-\P_{N_2})\<1>_\gm,
\end{align*} 
which in view of Lemma~\ref{LEM:covar}~(i) has covariance function
\begin{align}\label{cov2}
\G_{N_1-N_2}(t_1,t_2,x_1,x_2)= 2\pi e^{\frac{t_2-t_1}{2\pi}\Dlg}\big((\P_{N_1}-\P_{N_2})\otimes(\P_{N_1}-\P_{N_2})\big)\Gg(x_1,x_2).
\end{align}
Then using Minkowski's inequality with $p\ge 2$ we can compute
\begin{align*}
&\Big\|\<1>_{N_1-N_2}(t+\dl)-\<1>_{N_1-N_2}(t)\Big\|_{L^{p}(\mu_\gm\otimes\Prob)H^s_0(\M)}\\ &\le \Big\|(1-\Dlg)^{\frac{s}2}\big[\<1>_{N_1-N_2}(t+\dl)-\<1>_{N_1-N_2}(t)\big]\Big\|_{L^2(\M)L^{p}(\mu_\gm\otimes\Prob)}\intertext{Using next that $\<1>_{N_1-N_2}$ is a Gaussian process, we can continue with}
&\les \Big\|(1-\Dlg)^{\frac{s}2}\big[\<1>_{N_1-N_2}(t+\dl)-\<1>_{N_1-N_2}(t)\big]\Big\|_{L^2(\M)L^{2}(\mu_\gm\otimes\Prob)}\\
&= \bigg\{\sum_{n\ge 0}\jb{\ld_n}^{2s}\int_{H^s_0(\M)}\E\Big|\langle\varphi_n,\big[\<1>_{N_1-N_2}(t+\dl)-\<1>_{N_1-N_2}(t)\big]\rangle\Big|^2d\mu_\gm\bigg\}^{\frac12}\\
&= \bigg\{\sum_{n\ge 0}\jb{\ld_n}^{2s}\langle\varphi_n\otimes\varphi_n,\G_{N_1-N_2}(t+\dl,t+\dl)-2\G_{N_1-N_2}(t+\dl,t) + \G_{N_1-N_2}(t,t)\rangle\bigg\}^{\frac12}\intertext{where $\langle\cdot,\cdot\rangle$ denotes the usual inner product in either $L^2(\M)$ or $L^2(\M\times\M)$. Thus in view of \eqref{cov2} and using the mean value theorem, we can continue with}
& =\bigg\{\sum_{n\ge 1}\jb{\ld_n}^{2s}\frac{4\pi}{\ld_n^2} \big(1-e^{-\frac{\dl}{2\pi}\ld_n^2}\big)\big(e^{-N_1^{-2}\ld_n^2}-e^{-N_2^{-2}\ld_n^2}\big)^2\bigg\}^{\frac12}\\
&\les \bigg\{\sum_{n\ge 1}\ld_n^{2s-2}\min\big(1,\dl \ld_n^2\big)\min\big(1,|N_1^{-2}-N_2^{-2}|\ld_n^2\big)\bigg\}^{\frac12}\\&\les \dl^{\frac{\eps}2}\min(N_1,N_2)^{-\eps}\big\{\sum_{n\ge 1}\ld_n^{2s-2+4\eps}\big\} \les \dl^{\frac{\eps}2}\min(N_1,N_2)^{-\eps}
\end{align*}
for $0<\eps\ll 1$ small enough, since $s<0$. In particular, this shows that
\begin{align*}
\int_{H^s_0(\M)}\E\big\|\<1>_{N_1-N_2}(t+\dl)-\<1>_{N_1-N_2}(t)\big\|_{H^s_0(\M)}^p \les \dl^{p\frac{\eps}2}\min(N_1,N_2)^{-p\eps},
\end{align*}
so that we can conclude from Kolmogorov's continuity criterion (see e.g. \cite[Theorem 8.2]{Bass}) that for any $T>0$, there exists $p\ge 1$ large enough such that $\{\P_N\<1>_\gm\}_N$ is a Cauchy sequence in $L^p(\mu_\gm\otimes\Prob;C([0,T];H^s_0(\M)))$ and in particular $\<1>_\gm \in C([0,T];H^s_0(\M))$ almost surely. This proves Lemma \ref{LEM:SC}.
\end{proof}

We now move on to the proof of Proposition \ref{PROP:U}. Similarly to \cite{ORW}, it will follow from the following set of estimates.
\begin{lemma}\label{LEM:UEst}
Assume that $a_\ell\in\R$, $\be >0$ and $Q=\frac{2}{\be}+\frac{\be}2$ satisfy the assumptions \eqref{A1}-\eqref{Seiberg1}-\eqref{Seiberg2b}.
Then, for any $0<\eps\ll1$ and $T>0$, there exists $C>0$ such that for any $t\in [0,T]$, any $x\in\M\setminus\{x_1,...,x_L\}$ and any $N,N_1,N_2\in\N$, $M\in 2^{\Z_{\ge -1}}$, the following statements hold:

\smallskip

\noi
\textup{(i)} We have 
\begin{align*}
\int_{H^s_0(\M,\gm)}\E \big[ | \U_N (t,x)| \big]d\mu_\gm \le C \prod_{\ell=1}^\ell\dg(x_\ell,x)^{-\be a_\ell^+}\in L^1([0,T]\times\M)
\end{align*}
where $a_\ell^+=\max(a_\ell,0)$;

\smallskip

\noi
\textup{(ii)} Let $p$ be an even integer, $0<\al<2$ and $0<(p-1)\be^2<2\min(1,\al)$ with $(p-1)\be^2+2\be a_{\ell_\textup{max}}^+<4$, then
 \begin{align}
 \int_{H^s_0(\M,\gm)}\E \Big[ \left|\Q_M\U_N(t,x) \right|^{p} \Big]d\mu_\gm \le C  M^{p(\al-\eps)}f_{\al-\eps,\{x_\ell\}}(x)^{\frac{p}2}, 
\label{momentpunc}
 \end{align}
  where $f_{\al-\eps,\{x_\ell\}}$ is given by
\begin{align}\label{fj}
f_{\al-\eps,\{x_\ell\}}(x) = &\sum_{\substack{\ell_1,\ell_2=1\\\ell_1\neq \ell_2}}^L\big(1+ \dg(x_{\ell_1},x)^{\al-\eps-\be a_{\ell_1}^+}\big)\big(1+ \dg(x_{\ell_2},x)^{\al-\eps-\be a_{\ell_2}^+}\big) \notag\\
&+ \sum_{\ell=1}^L \dg(x_\ell,x)^{2\al-2\eps-(p-1)\be^2-2\be a_\ell^+}
\end{align}
with $a_{\ell_\textup{max}}=\max_{\ell=1,...,L}a_\ell$;
\smallskip

\noi
\textup{(iii)} Given any $0<\be^2<2\min(1,\al)$ for $0<\al<2$ with $\be^2+2\be a_{\ell_\textup{max}}^+<4$,
 then it holds
\begin{align*}
\int_{H^s_0(\M,\gm)}\E \Big[ \big| \Q_M\big(\U_{N_1}(t,x)-\U_{N_2}(t,x)\big) \big|^{2} \Big]d\mu_\gm\le C 
M^{2(\al-\eps)}\big|N_1^{-2}-N_2^{-2}\big|^{-\wt\eps}f_{\al-\eps-2\wt\eps,\{x_\ell\}}(x)
\end{align*}

\noi
for some small $0<\wt\eps\ll\eps\ll 1$.\\
\textup{(iv)} If $\psi_1,\psi_2\in\S(\R)$ satisfy $\psi_1(0)=\psi_2(0)$, then
\begin{align}
\int_{H^s_0(\M,\gm)}\E \Big[ \big| \Q_M\big(\U_{N,1}(t,x)-\U_{N,2}(t,x)\big) \big|^{2} \Big]d\mu_\gm\le C 
M^{p\al-\eps}N^{-\wt\eps}f_{\al-\eps-2\wt\eps,\{x_\ell\}}(x),
\label{momentdiff2}
\end{align}
where 
\begin{align*}
\U_{N,j} = e^{-\pi\be^2C_{\psi_j}}N^{-\frac{\be^2}2}e^{\be\psi_j(-N^{-2}\Dlg)\<1>_\gm+\sum_{\ell=1}^L\be a_\ell2\pi(\psi_j\otimes\psi_j)(-N^{-2}\Dlg)\Gg(x_\ell,x)}.
\end{align*}
\end{lemma}
\begin{proof}
The proof of Lemma \ref{LEM:UEst} follows from the same argument as for \cite[Proposition 3.2]{ORW}. Indeed, (i) is a straightforward computation: using Lemmas \ref{LEM:covar}, \ref{LEM:Green}, \ref{LEM:GN1} and \ref{LEM:GN4}, we indeed obtain
\begin{align*}
\int_{H^s_0(\M,\gm)}\E \big[ | \U_N (t,x)| \big]d\mu_\gm &\le C e^{2\pi\be\sum_{\ell=1}^L a_\ell(\P_N\otimes\P_N)\Gg(x_\ell,x)}\\&\le C \prod_{\ell=1}^L e^{\be a_\ell^+2\pi\big(\Gg(x_\ell,x)+(N^2\Vg(\M))^{-1}\big)}\\
&\le C\prod_{\ell=1}^L e^{-\be a_\ell^+\log\big(\dg(x_\ell,x)\big)},
\end{align*}
where we used that both $(N^2\Vg(\M))^{-1}$ and $\wt\Gg$ in Lemma \ref{LEM:Green} are bounded (uniformly in $N$). We also used that the metric is smooth and that $\M$ is compact. This shows (i).

 As for (ii), recall that the kernel $\K_M$ of $\Q_M$ has been defined in \eqref{K1} with $\psi$ as in \eqref{QM}, so that with Lemma \ref{LEM:covar} we can compute for any $p=2q\in\N$
\begin{align*}
&\int_{H^s_0(\M,\gm)}\E\Big[\Big|\Q_M\U_N(t,x)\Big|^{2q}\Big]d\mu_\gm\\
&= \int_{\M}\cdots \int_{\M} \int_{H^s_0(\M,\gm)}\E\Big[\prod_{j=1}^{2q}\U_N(t,y_j)\Big]d\mu_\gm\prod_{j=1}^{2q}\K_M(x,y_j)d\Vg(y_j)\\
&=\int_{\M}\cdots \int_{\M} e^{2\pi\be^2 \sum_{j<k}(\P_N\otimes\P_N)\Gg(y_j,y_k)}e^{\pi\be^2\sum_{j=1}^{2p}\wt\Gg(y_j,y_j)+o(1)}\\
&\qquad\times \prod_{j=1}^{2q}e^{2\pi\be\sum_{\ell=1}^La_\ell (\P_N\otimes\P_N)\Gg(x_\ell,y_j)}\K_M(x,y_j)d\Vg(y_j)\\
&\le C\int_{\M^{2q}} \prod_{j<k}\Big[e^{2\pi\be^2(\P_N\otimes\P_N)\Gg(y_j,y_k)} e^{\frac{2\pi\be}{(2q-1)}\sum_{\ell=1}^La_\ell \big((\P_N\otimes\P_N)\Gg(x_\ell,y_j)+(\P_N\otimes\P_N)\Gg(x_\ell,y_k)\big)}\\
&\qquad\times |\K_M(x,y_j)|^{\frac1{2q-1}}|\K_M(x,y_k)|^{\frac1{2q-1}}\Big]d\Vg(\vec{y}),\intertext{where $\vec{y} = (y_1,...,y_{2q}) \in\M^{2q}$ and $\Vg(\vec{y})$ is the corresponding product measure on $\M^{2q}$. We can then use the Brascamp-Lieb inequality given by Lemma \ref{LEM:BL} with the smoothness of $\gm$ and the compactness of $\M$ to continue with}
&\les \prod_{j<k}\Big(\int_\M\int_\M e^{(2q-1)2\pi\be^2(\P_N\otimes\P_N)\Gg(y_j,y_k)} e^{2\pi\be\sum_{\ell=1}^La_\ell \big((\P_N\otimes\P_N)\Gg(x_\ell,y_j)+(\P_N\otimes\P_N)\Gg(x_\ell,y_k)\big)}\\
&\qquad\qquad\times |\K_M(x,y_j)\K_M(x,y_k)|d\Vg(y_j)d\Vg(y_k)\Big)^{\frac1{2q-1}}\intertext{Then, using Lemmas \ref{LEM:Green}, \ref{LEM:GN1} and \ref{LEM:PM}, we can bound by symmetry this last term by}
&\les \Big(\int_\M\int_\M \dg(y,z)^{-(2q-1)\be^2}\prod_{\ell=1}^L\dg(x_\ell,y)^{-\be a_\ell^+}\dg(x_\ell,z)^{-\be a_\ell^+}\\
&\qquad\qquad\times M^2\jb{M\dg(x,y)}^{-A}M^2\jb{M\dg(x,z)}^{-A}d\Vg(y)d\Vg(z)\Big)^{q}
 \end{align*}
 for any $A>0$. In particular, using that $p=2q$ and taking $A=2-\al+\eps>0$ for some $0<\eps\ll 1$ and using that $M^2\jb{M\dg(x,y)}^{-A}\les M^{\al-\eps}\dg(x,y)^{\al-2-\eps}$,  (ii) will be established once we show that the double integrals
 \begin{align*}
& \int_\M\int_\M \dg(x,y)^{\al-2-\eps}\dg(x,z)^{\al-2-\eps} \dg(y,z)^{-(p-1)\be^2}\\
&\qquad\qquad\times\prod_{\ell=1}^L\dg(x_\ell,y)^{-\be a_\ell^+}\dg(x_\ell,z)^{-\be a_\ell^+}d\Vg(y)d\Vg(z)
 \end{align*}
 are bounded by $f_{\al-\eps,\{x_\ell\}}(x)$ as in \eqref{fj}.
 
To bound this last double integral, first note that we only need to consider one of the singularities $\dg(x_\ell,y)^{-\be a_\ell^+}$ and similarly for $z$. Indeed, if $0<r\ll\min_{\ell\neq k}\dg(x_\ell,x_k)$ and $B_\ell$ is the ball of radius $r$ around $x_\ell$, we can bound the previous integrals with
\begin{align}
&\sum_{\ell,k=1}^Lr^{-\sum_{\ell'\neq \ell}\be a_{\ell'}-\sum_{k'\neq k}\be a_{k'}}\int_{B_\ell}\int_{B_k}\dg(x,y)^{\al-2-\eps}\dg(x,z)^{\al-2-\eps}\notag\\
&\qquad\times\dg(x_\ell,y)^{-\be a_\ell^+}\dg(x_k,z)^{-\be a_k^+}\dg(y,z)^{-(p-1)\be^2}d\Vg(y)d\Vg(z)\notag\\
&+2\sum_{\ell=1}^Lr^{-\sum_{\ell'\neq \ell}\be a_{\ell'}-\sum_k \be a_k^+}\int_{B_\ell}\int_{\M\setminus(\cup_k B_k)}\dg(x,y)^{\al-2-\eps}\dg(x,z)^{\al-2-\eps}\notag\\
&\qquad\qquad\times\dg(x_\ell,y)^{-\be a_\ell^+}\dg(y,z)^{-(p-1)\be^2}d\Vg(y)d\Vg(z)\notag\\
&+r^{-2\sum_{\ell=1}^L\be a_\ell^+}\int_{\M\setminus(\cup_\ell B_\ell)}\int_{\M\setminus(\cup_\ell B_\ell)}\dg(x,y)^{\al-2-\eps}\dg(x,z)^{\al-2-\eps}\dg(y,z)^{-(p-1)\be^2}d\Vg(y)d\Vg(z)\notag\\
& =\1 + \II + \III.
\label{U1}
\end{align}

We first deal with the integrals of $\III$ in \eqref{U1}. In the case $\dg(x,y)\ll \dg(x,z)\sim\dg(y,z)$ we bound them with
\begin{align*}
&\int_{\M}\dg(x,z)^{\al-2-\eps-(p-1)\be^2}\int_{\dg(x,y)\ll\dg(x,z)}\dg(x,y)^{\al-2-\eps}d\Vg(y)d\Vg(z)\\
&\les\int_{\M}\dg(x,z)^{2\al-2-\eps-(p-1)\be^2}d\Vg(z)\les 1
\end{align*}
uniformly in $x\in\M$, where in the first step we used that $\al>0$ with $0<\eps\ll 1$ and in the last one that $2\al>(p-1)\be^2$. The case $\dg(x,z)\ll\dg(x,y)\sim\dg(y,z)$ is handled similarly by symmetry, and for the case $\dg(y,z)\les \dg(x,y)\sim\dg(x,z)$ we have the bound
\begin{align*}
&\int_{\M}\dg(x,y)^{2\al-4-2\eps}\int_{\dg(y,z)\les\dg(x,y)}\dg(y,z)^{-(p-1)\be^2}d\Vg(z)d\Vg(y)\\
&\les\int_{\M}\dg(x,y)^{2\al-2-2\eps-(p-1)\be^2}\les 1,
\end{align*}
where this time we used in the first step that $(p-1)\be^2<2$.

We now turn to the terms of $\1$ in \eqref{U1}, and for each $\ell,k=1,...,L$ we estimate
\begin{align*}
\int_{B_\ell}\int_{B_k}\dg(x,y)^{\al-2-\eps}\dg(x,z)^{\al-2-\eps}\dg(x_\ell,y)^{-\be a_\ell^+}\dg(x_k,z)^{-\be a_k^+}\dg(y,z)^{-(p-1)\be^2}d\Vg(y)d\Vg(z).
\end{align*}
\textbf{Case 1: if $\ell\neq k$.} In this case the integrals can be bounded by
\begin{align}\label{RHS1}
&r^{-(p-1)\be^2}\Big(\int_{B_\ell}\dg(x,y)^{\al-2-\eps}\dg(x_\ell,y)^{-\be a_\ell^+}d\Vg(y)\Big)\Big(\int_{B_k}\dg(x,z)^{\al-2-\eps}\dg(x_k,z)^{-\be a_k^+}d\Vg(z)\Big).
\end{align}
From \cite[Proposition 4.12]{Aubin}, we have that\begin{align*}
\int_{B_\ell}\dg(x,y)^{\al-2-\eps}\dg(x_\ell,y)^{-\be a_\ell^+}d\Vg(y) \les \begin{cases}
1\text{ if }\al> \be a_\ell^++\eps\\
1+\big|\log\big(\dg(x,x_\ell)\big)\big| \text{ if }\al=\be a_\ell^++\eps,\\
\dg(x,x_\ell)^{\al-\be a_\ell^+-\eps} \text{ if }\al < \be a_\ell^++\eps,
\end{cases}
\end{align*}
and similarly for the integral in $z$. In all three cases, the terms \eqref{RHS1} are then bounded by the first terms in the right-hand side of \eqref{fj}. \\
\textbf{Case 2: if $\ell=k$.} We now need to bound
\begin{align}\label{integral}
\int_{B_\ell}\int_{B_\ell}\dg(x,y)^{\al-2-\eps}\dg(x,z)^{\al-2-\eps}\dg(x_\ell,y)^{-\be a_\ell^+}\dg(x_\ell,z)^{-\be a_\ell^+}\dg(y,z)^{-(p-1)\be^2}d\Vg(y)d\Vg(z).
\end{align}
To estimate these integrals we look at the following regions:
\begin{align*}
&\Rg_1 \deff \big\{y\in B_\ell,~\dg(x,y)\ll \dg(x_\ell,y)\sim\dg(x,x_\ell)\big\},\\
&\Rg_2 \deff \big\{y\in B_\ell,~\dg(x_\ell,y)\ll \dg(x,y)\sim\dg(x,x_\ell)\big\},\\
&\Rg_3 \deff \big\{y\in B_\ell,~\dg(x,x_\ell)\les \dg(x,y)\sim\dg(x_\ell,y)\big\},
\end{align*}
and for $i_1,i_2\in\{1,2,3\}$ we define the subregion of $B_\ell\times B_\ell$ as
\begin{align*}
\Rg_{i_1,i_2}\deff \Rg_{i_1}\times \Rg_{i_2}.
\end{align*}
In particular note that $B_\ell\times B_\ell = \cup_{i_1,i_2=1}^3 \Rg_{i_1,i_2}$.\\

\noi
\textbf{Contribution of $\Rg_{1,1}$:} In this region we can bound \eqref{integral} with
\begin{align}\label{integral11}
\dg(x,x_\ell)^{-2\be a_\ell^+}\int_{\Rg_1}\int_{\Rg_1}\dg(x,y)^{\al-2-\eps}\dg(x,z)^{\al-2-\eps}\dg(y,z)^{-(p-1)\be^2}d\Vg(z)d\Vg(y).
\end{align}
Proceeding then as for the second line of \eqref{U1}, in the case $\dg(x,y)\les \dg(x,z)\sim \dg(y,z)$ we get the bound
\begin{align*}
&\dg(x,x_\ell)^{-2\be a_\ell^+}\int_{\Rg_1}\dg(x,z)^{\al-2-\eps-(p-1)\be^2}\int_{\dg(x,y)\les \dg(x,z)}\dg(x,y)^{\al-2-\eps}d\Vg(y)d\Vg(z)\\
&\les \dg(x,x_\ell)^{-2\be a_\ell^+}\int_{\Rg_1}\dg(x,z)^{2\al-2-2\eps-(p-1)\be^2}d\Vg(z)\\
&\les \dg(x,x_\ell)^{-2\be a_\ell^++2\al-2\eps-(p-1)\be^2}
\end{align*}
where in the last step we used the condition $2\al>(p-1)\be^2$ along with the definition of $\Rg_1$. The case $\dg(x,z)\les \dg(x,y)\sim \dg(y,z)$ is treated similarly by exchanging the roles of $y$ and $z$, and in the case $\dg(y,z)\les \dg(x,y)\sim \dg(x,z)$ we can bound \eqref{integral11} with
\begin{align*}
&\dg(x,x_\ell)^{-2\be a_\ell^+}\int_{\Rg_1}\dg(x,z)^{2\al-4-2\eps}\int_{\dg(y,z)\les \dg(x,z)}\dg(y,z)^{-(p-1)\be^2}d\Vg(y)d\Vg(z)\\
&\les\dg(x,x_\ell)^{-2\be a_\ell^+}\int_{\Rg_1}\dg(x,z)^{2\al-2-2\eps-(p-1)\be^2}d\Vg(z)\\
&\les \dg(x,x_\ell)^{-2\be a_\ell^++2\al-2\eps-(p-1)\be^2},
\end{align*} 
where in the first step we used the condition $(p-1)\be^2<2$ and in the second step we used that $2\al>(p-1)\be^2$.\\

\noi
\textbf{Contribution of $\Rg_{1,2}$:} In this region, note that we have
\begin{align*}
\dg(y,z)\ge \dg(x,x_\ell)-\dg(x,y)-\dg(x_\ell,z)\gtrsim \dg(x,x_\ell).
\end{align*}
Thus we estimate \eqref{integral} with
\begin{align*}
&\dg(x,x_\ell)^{\al-2-\eps-\be a_\ell^+-(p-1)\be^2}\int_{\Rg_1}\int_{\Rg_2}\dg(x,y)^{\al-2-\eps}\dg(x_\ell,z)^{-\be a_\ell^+}d\Vg(z)d\Vg(y)\\
&\les \dg(x,x_\ell)^{2\al-2\eps-2\be a_\ell^+-(p-1)\be^2}
\end{align*}
by using that $\al-2-\eps>-2$ and $\be a_\ell^+<2$.\\

\noi
\textbf{Contribution of $\Rg_{1,3}$:} In this region we have in particular $\dg(x,y)\ll \dg(x,z)$ so that $\dg(y,z)\sim\dg(x,z)$. Thus the contribution of this region in \eqref{integral} is bounded by
\begin{align*}
\dg(x,x_\ell)^{-2\be a_\ell^+}\int_{\Rg_1}\int_{\Rg_3}\dg(x,y)^{\al-2-\eps}\dg(x,z)^{\al-2-\eps-(p-1)\be^2}d\Vg(z)d\Vg(y).
\end{align*}
Since for any fixed $z$ it holds
\begin{align*}
\int_{\dg(x,y)\ll \dg(x,z)}\dg(x,y)^{\al-2-\eps}d\Vg(y)\les \dg(x,z)^{\al-\eps},
\end{align*}
we get the final bound
\begin{align*}
\dg(x,x_\ell)^{-2\be a_\ell^+}\int_{\dg(x,x_\ell)\les\dg(x,z)}\dg(x,z)^{2\al-2-2\eps-(p-1)\be^2}d\Vg(z) \les \dg(x,x_\ell)^{2\al-2\eps-2\be a_\ell^+-(p-1)\be^2}
\end{align*}
since $2\al>(p-1)\be^2$.\\

\noi
\textbf{Contribution of $\Rg_{2,2}$:} By symmetry, the contribution of this region can be estimated exactly as for $\Rg_{1,1}$ up to exchanging the roles of $x$ and $x_\ell$, and $\al-2-\eps$ with $-\be a_\ell^+$. Indeed, we can bound \eqref{integral} with
\begin{align*}
\dg(x,x_\ell)^{2\al-4-2\eps}\int_{\Rg_2}\int_{\Rg_2}\dg(x_\ell,y)^{-\be a_\ell^+}\dg(x_\ell,z)^{-\be a_\ell^+}\dg(y,z)^{-(p-1)\be^2}d\Vg(z)d\Vg(y).
\end{align*}
As above, in the case $\dg(x_\ell,y)\les \dg(x_\ell,z)\sim \dg(y,z)$ we get the bound
\begin{align*}
&\dg(x,x_\ell)^{2\al-4-2\eps}\int_{\Rg_2}\dg(x_\ell,z)^{-\be a_\ell^+-(p-1)\be^2}\int_{\dg(x_\ell,y)\les \dg(x_\ell,z)}\dg(x_\ell,y)^{-\be a_\ell^+}d\Vg(y)d\Vg(z)\\
&\les \dg(x,x_\ell)^{2\al-4-2\eps}\int_{\Rg_2}\dg(x_\ell,z)^{2-2\be a_\ell^+-(p-1)\be^2}d\Vg(z)\\
&\les \dg(x,x_\ell)^{-2\be a_\ell^++2\al-2\eps-(p-1)\be^2}
\end{align*}
where we used the conditions $\be a_\ell^+<2$ and $2\be a_\ell^++(p-1)\be^2<4$. The case $\dg(x_\ell,z)\les \dg(x_\ell,y)\sim \dg(y,z)$ is treated similarly by exchanging the roles of $y$ and $z$, and in the case $\dg(y,z)\les \dg(x_\ell,y)\sim \dg(x_\ell,z)$ we can bound the integral by
\begin{align*}
&\dg(x,x_\ell)^{2\al-4-2\eps}\int_{\Rg_2}\dg(x_\ell,z)^{-2\be a_\ell^+}\int_{\dg(y,z)\les \dg(x_\ell,z)}\dg(y,z)^{-(p-1)\be^2}d\Vg(y)d\Vg(z)\\
&\les\dg(x,x_\ell)^{2\al-4-2\eps}\int_{\Rg_2}\dg(x_\ell,z)^{2-2\be a_\ell^+-(p-1)\be^2}d\Vg(z)\\
&\les \dg(x,x_\ell)^{-2\be a_\ell^++2\al-2\eps-(p-1)\be^2},
\end{align*} 
where in the first step we used the condition $(p-1)\be^2<2$ and in the second step we used that $2\be a_\ell^++(p-1)\be^2<4$.\\

\noi
\textbf{Contribution of $\Rg_{2,3}$: } This is similar to the contribution of $\Rg_{1,3}$ above, up to exchanging $x$ with $x_\ell$ and $\al-2-\eps$ with $-\be a_\ell^+$ as for $\Rg_{2,2}$.\\

\noi
\textbf{Contribution of $\Rg_{3,3}$: } In this last case, we can bound \eqref{integral} with
\begin{align*}
\int_{\Rg_3}\int_{\Rg_3}\dg(x,y)^{\al-2-\eps-\be a_\ell^+}\dg(x,z)^{\al-2-\eps-\be a_\ell^+}\dg(y,z)^{-(p-1)\be^2}d\Vg(y)d\Vg(z).
\end{align*}
In the case $\dg(x,y)\les \dg(x,z)\sim \dg(y,z)$, these integrals reduce to
\begin{align*}
&\int_{\Rg_3}\int_{\Rg_3}\mathbf{1}_{\{\dg(x,y)\les \dg(x,z)\}}\dg(x,z)^{\al-2-\eps-\be a_\ell^+-(p-1)\be^2}\dg(x,y)^{\al-2-\eps-\be a_\ell^+}d\Vg(y)d\Vg(z).
\end{align*}
If $\al - \be a_\ell - (p-1)\be^2<0$, we integrate in $z$ first to get the bound
\begin{align*}
\int_{\Rg_3}\dg(x,y)^{2\al-2-2\eps -2\be a_\ell^+-(p-1)\be^2}d\Vg(y) \les 1+ \dg(x,x_\ell)^{2\al-2\eps-2\be a_\ell^+-(p-1)\be^2}
\end{align*}
where in the last step we used the definition of $\Rg_3$. On the other hand, if $\al-\be a_\ell^++ - (p-1)\be^2\ge 0$ then in particular $\al-\be a_\ell^+\ge 0$ and thus we can integrate in $y$ first to obtain the bound
\begin{align*}
\int_{\Rg_3}\dg(x,z)^{2\al-2-2\eps -2\be a_\ell^+-(p-1)\be^2}d\Vg(z) \les 1+ \dg(x,x_\ell)^{2\al-2\eps-2\be a_\ell^+-(p-1)\be^2}.
\end{align*}
In the last case $\dg(y,z)\les \dg(x,y)\sim\dg(x,z)$, the integrals reduce to
\begin{align*}
\int_{\Rg_3}\int_{\Rg_3}\dg(x,y)^{2\al-4-2\eps-2\be a_\ell^+}\dg(y,z)^{-(p-1)\be^2}d\Vg(y)d\Vg(z)
\end{align*}
and integrating in $z$ first using that $(p-1)\be^2<2$ we finally get the bound
\begin{align*}
\int_{\Rg_3}\dg(x,y)^{2\al-2-2\eps-2\be a_\ell^+-(p-1)\be^2}d\Vg(y)\les 1+ \dg(x,x_\ell)^{2\al-2\eps-2\be a_\ell^+-(p-1)\be^2}.
\end{align*}
By symmetry of \eqref{integral} between $y$ and $z$, the contribution of the remaining regions is estimated similarly. 

At last, it remains to deal with the integrals of $\II$ in \eqref{U1}. But we can estimate
\begin{align*}
&\int_{B_\ell}\int_{\M\setminus(\cup_k B_k)}\dg(x,y)^{\al-2-\eps}\dg(x,z)^{\al-2-\eps}\dg(x_\ell,y)^{-\be a_\ell^+}\dg(y,z)^{-(p-1)\be^2}d\Vg(y)d\Vg(z)\\
&\les \dg(x,x_\ell)^{2\al-2-2\eps-\be a_\ell^+-(p-1)\be^2}
\end{align*}
by similar computations as above. Note that this requires the constraint $\be a_\ell^++(p-1)\be^2<4$, which is weaker than the constraint used to bound the integrals of $\III$ in \eqref{U1}. All in all, this leads to \eqref{fj}.

Finally, Lemma \ref{LEM:UEst} (iii) follows from similar computations as in the proof of \cite[Proposition 1.1]{ORSW}: indeed, we have
\begin{align}
&\int_{H^s_0(\M,\gm)}\E \Big[ \big|\Q_M\big(\U_{N_1}(t,x)-\U_{N_2}(t,x)\big) \big|^{2} \Big]d\mu_\gm\notag\\ &= \int_{\M}\int_{\M} \K_M(x,y_1)\K_M(x,y_2)\Big\{H_{N_1}(y_1)H_{N_1}(y_2)e^{2\pi\be^2(\P_{N_1}\otimes\P_{N_1})\Gg(y_1,y_2)}\notag\\
&\qquad\qquad-2H_{N_1}(y_1)H_{N_2}(y_2)e^{2\pi\be^2(\P_{N_1}\otimes\P_{N_2})\Gg(y_1,y_2)}\label{U2}\\
&\qquad\qquad+H_{N_2}(y_1)H_{N_2}(y_2)e^{2\pi\be^2(\P_{N_2}\otimes\P_{N_2})\Gg(y_1,y_2)}\Big\}d\Vg(y_1)d\Vg(y_2),\notag
\end{align}
where we wrote
\begin{align}
H_N(y)=e^{2\pi\be\sum_{\ell=1}^L a_\ell(\P_N\otimes\P_N)\Gg(x_\ell,y)}e^{\pi\be^2\wt\Gg(y,y)+o(1)}.
\label{HN}
\end{align}
We deal with
\begin{align*}
&H_{N_1}(y_1)H_{N_1}(y_2)e^{2\pi\be^2(\P_{N_1}\otimes\P_{N_1})\Gg(y_1,y_2)}-H_{N_1}(y_1)H_{N_2}(y_2)e^{2\pi\be^2(\P_{N_1}\otimes\P_{N_2})\Gg(y_1,y_2)}\\
&=\Big(H_{N_1}(y_2)-H_{N_2}(y_2)\Big)H_{N_1}(y_1)e^{2\pi\be^2(\P_{N_1}\otimes\P_{N_1})\Gg(y_1,y_2)}\\
&\qquad\qquad + H_{N_1}(y_1)H_{N_2}(y_2)\Big(e^{2\pi\be^2(\P_{N_1}\otimes\P_{N_1})\Gg(y_1,y_2)}-e^{2\pi\be^2(\P_{N_1}\otimes\P_{N_2})\Gg(y_1,y_2)}\Big)\\
&=\1+\II.
\end{align*}
Using the mean value theorem, Lemma \ref{LEM:GN1} by interpolating between the two bounds in \eqref{GN2}, and Lemma \ref{LEM:Green} with the definition of $H_N$ \eqref{HN}, we have
\begin{align*}
\big|\1\big|&\les \sum_{\ell=1}^L\Big|(\P_{N_1}\otimes\P_{N_1})\Gg(x_\ell,y_2)-(\P_{N_2}\otimes\P_{N_2})\Gg(x_\ell,y_2)\Big|\\
&\qquad\times\Big(\prod_{k=1}^L\dg(x_k,y_2)^{-\be a_k^+}\Big)H_{N_1}(y_1)e^{2\pi\be^2(\P_{N_1}\otimes\P_{N_1})\Gg(y_1,y_2)}\\
&\les \sum_{\ell=1}^LN_1^{-\frac{\wt\eps}4}\dg(x_\ell,y_2)^{-\wt\eps}\Big(\prod_{k=1}^L\dg(x_k,y_2)^{-\be a_k^+}\dg(x_k,y_1)^{-\be a_k^+}\Big)\dg(y_1,y_2)^{-\be^2}\\
&\les N_1^{-\frac{\wt\eps}4}\Big(\prod_{\ell=1}^L\dg(x_\ell,y_2)^{-\wt\eps-\be a_\ell^+}\dg(x_\ell,y_1)^{-\be a_\ell^+}\Big)\dg(y_1,y_2)^{-\be^2},
\end{align*}
for any $0<\wt\eps\ll\eps\ll 1$ and $N_1\le N_2$. Similarly, we use again Lemmas \ref{LEM:Green} and \ref{LEM:GN1} to bound
\begin{align*}
\big|\II\big|&\les \Big|(\P_{N_1}\otimes\P_{N_1})\Gg(y_1,y_2)-(\P_{N_1}\otimes\P_{N_2})\Gg(y_1,y_2)\Big|\\
&\qquad\times \Big(\prod_{\ell=1}^L\dg(x_k,y_2)^{-\be a_\ell^+}\dg(x_k,y_1)^{-\be a_\ell^+}\Big)\dg(y_1,y_2)^{-\be^2}\\
&\les \Big(\prod_{\ell=1}^L\dg(x_k,y_2)^{-\be a_\ell^+}\dg(x_k,y_1)^{-\be a_\ell^+}\Big)N_1^{-\frac{\wt\eps}4}\dg(y_1,y_2)^{-\wt\eps-\be^2}\\
&=N_1^{-\frac{\wt\eps}4}\Big(\prod_{\ell=1}^L\dg(x_k,y_2)^{-\be a_\ell^+}\dg(x_k,y_1)^{-\be a_\ell^+}\Big)\dg(y_1,y_2)^{-\wt\eps-\be^2}
\end{align*}
for some $0<\wt\eps\ll\eps\ll1$.

 Plugging these bounds into \eqref{U2} and proceeding as for (ii), we get
 \begin{align*}
 &\int_{H^s_0(\M,\gm)}\E \Big[ \big| \Q_M\big(\U_{N_1}(t,x)-\U_{N_2}(t,x)\big) \big|^{2} \Big]d\mu_\gm\\
 &\les M^{2(\al - \eps)}N_1^{-\frac{\wt\eps}4}\int_{\M}\int_{\M}\dg(x,y_1)^{\al-2-\eps}\dg(x,y_2)^{\al-2-\eps}\\
 &\qquad\qquad\times\Big(\prod_{\ell=1}^L\dg(x_k,y_2)^{-\wt\eps-\be a_\ell^+}\dg(x_k,y_1)^{-\wt\eps-\be a_\ell^+}\Big)\dg(y_1,y_2)^{-\wt\eps-\be^2}d\Vg(y_1)d\Vg(y_2)\\
 &\les M^{2(\al - \eps)}N_1^{-\frac{\wt\eps}4}f_{\al-\eps-2\wt\eps,\{x_\ell\}}(x).
 \end{align*}
  This finally establishes Lemma \ref{LEM:UEst} (iii). The estimate \eqref{momentdiff2} of Lemma \ref{LEM:UEst} (iv) then follows from the same computations as above up to using Lemma \ref{LEM:GN4} in place of Lemma \ref{LEM:GN1}.
\end{proof}
\begin{proof}[Proof of Proposition \ref{PROP:U}]
Using Lemma \ref{LEM:UEst}, we can finally prove Proposition \ref{PROP:U}. Let $\be,a_\ell,Q$ satisfy \eqref{A1}-\eqref{Seiberg1}-\eqref{Seiberg2b} and $\be^2+2\be a_{\ell_\textup{max}}^+<4$.

We begin by treating the case $1<p\le 2$. Let $\al(p)\in [0,2)$ satisfy the assumptions of Proposition~\ref{PROP:U}. Let also $0<\wt\eps\ll\eps\ll 1$ and $0<\dl\ll 1$, and $\theta\in [0,1]$ be such that $\frac1p = 1-\theta+\frac{\theta}2$, i.e. $\theta = 2\frac{p-1}{p}$. We define $\al=\frac{\al(p)-(1-\theta)\dl}{\theta}$. Then we see that, with the assumptions on $\al(p)$ and taking $\dl$ sufficiently small, we have that $\al\in (0,2)$ satisfies
\begin{align*}
\al>\max\big\{\frac{\be^2}2,\frac{\be^2}2+\be a_{\ell_\textup{max}}^+-1\big\}.
\end{align*} 
Thus we can use Fubini's theorem and Lemma~\ref{LEM:UEst}~(iii) to bound for any $T>0$ and $N_1,N_2\in\N$:
\begin{align*}
&\big\|\U_{N_1}-\U_{N_2}\big\|_{L^2(\mu_\gm\otimes\Prob)L^2_TB^{-\al}_{2,2}(\M)}\\
& = \bigg\{\sum_{M\in 2^{\Z_{\ge -1}}}M^{-2\al}\int_0^T\int_{\M}\int_{H^s_0(\M,\gm)} \E\Big[\big|\Q_M\big(\U_{N_1}-\U_{N_2}\big)(t,x)\big|^2\Big]d\mu_\gm d\Vg(x)dt\bigg\}^{\frac12}\\
&\le C\bigg\{\sum_{M\in 2^{\Z_{\ge -1}}}M^{-2\al+2(\al-\eps)}\min(N_1,N_2)^{-2\wt\eps}\int_0^T\int_{\M}f_{\al-\eps-2\wt\eps,\{x_\ell\}}(x) d\Vg(x)dt\bigg\}^{\frac12}\intertext{where $f_{\al-\eps-2\wt\eps,\{x_\ell\}}$ is as in \eqref{fj}. With the properties of $\al$, we have in particular $f_{\al-\eps-2\wt\eps,\{x_\ell\}}\in L^1(\M)$. Thus we can sum on $M$ to continue with}
&\le CT^{\frac12}\min(N_1,N_2)^{-\wt\eps}\big\|f_{\al-\eps-2\wt\eps,\{x_\ell\}}\big\|_{L^1(\M)}^{\frac12}.
\end{align*}
This shows that $\{\U_N\}$ is a Cauchy sequence in $L^2(\mu_\gm\otimes\Prob;L^2([0,T];B^{-\al}_{2,2}(\M)))$, thus converging to $\U$ in this space. In particular, there exists a subsequence $\{N_k\}$ such that $M^{-\al}\Q_M\U_{N_k}$ converges to $M^{-\al}\Q_M\U$ for almost every $(X_\gm,\om,t,x,M)\in H^s_0(\M)\times\O\times[0,T]\times \M\times 2^{\Z_{\ge -1}}$. Then, using Fatou's lemma and Fubini's theorem with Lemma~\ref{LEM:PM} and Lemma~\ref{LEM:UEst}~(i), we get for any $0<\dl\ll 1$:
\begin{align*}
\big\|\U\big\|_{L^1(\mu_\gm\otimes\Prob)L^1_TB^{-\dl}_{1,1}(\M)} &\le \liminf_k \sum_{M\in 2^{\Z_{\ge -1}}}M^{-\dl}\int_0^T\int_{\M}\int_{H^s_0(\M,\gm)}\E|\Q_M\U_{N_k}(t,x)|d\mu_\gm d\Vg(x)dt\\
&\les \sum_{M\in 2^{\Z_{\ge -1}}}M^{-\dl}\int_0^T\int_{\M}\int_{\M}M^2\jb{M\dg(x,y)}^{-A}\wt f_{\{x_\ell\}}(y)d\Vg(y) d\Vg(x)dt
\end{align*}
for any $A>0$, where $\wt f_{\{x_\ell\}}$ denotes the right-hand side of the estimate in Lemma~\ref{LEM:UEst}~(i). In particular $\wt f_{\{x_\ell\}}\in L^1(\M)$, so that the last term above is finite. This shows that we also have $\U\in L^1(\mu_\gm\otimes\Prob;L^1([0,T];B^{-\dl}_{1,1}(\M)))$ for any $0<\dl\ll 1$. Interpolating with $\U\in L^2(\mu_\gm\otimes\Prob;L^2([0,T];B^{-\al}_{2,2}(\M)))$, we finally get that 
\begin{align*}
\U\in L^p(\mu_\gm\otimes\Prob;L^p([0,T];B^{-(1-\theta)\dl-\theta\al}_{p,p}(\M))))
\end{align*}
which concludes the proof of Proposition~\ref{PROP:U} in the case $1<p\le 2$ by definition of $\al$ and $\theta$.

At last, we discuss the case $p=2m> 2$ an even integer. Assume that $\be^2<(2m-1)^{-1}$ and $(2m-1)\be^2+2\be a_\ell^+<4$. Note that under \eqref{al} we have now
\begin{align*}
\al(2m)>\max\big\{\be a_\ell^+ -\frac2{m},(2m-1)\frac{\be^2}2,(2m-1)\frac{\be^2}2+\be a_{\ell_\textup{max}}^+ -\frac1{m}\big\}.
\end{align*}
Thus in particular $f_{\al(2m)-\eps,\{x_\ell\}} \in L^{m}(\M)$ in view of \eqref{fj}, and the same computations as above, using now Lemma~\ref{LEM:UEst}~(ii), show that $\{\U_N\}$ is uniformly bounded in $L^{2m}(\mu_\gm\otimes\Prob;L^{2m}([0,T];B^{-\al(2m)}_{2m,2m}(\M)))$ for any $T>0$, from which we conclude that $\U$ also belongs to this class. This proves Proposition~\ref{PROP:U}.
%
\end{proof}
\begin{remark}\label{REM:U}\rm~\\
\textup{(i)} As mentioned in the introduction, the condition \eqref{Seiberg2b} is more restrictive than the usual second Seiberg bound \eqref{Seiberg2} usually assumed for the construction of the LQG measure. However in our situation, the bound \eqref{Seiberg2b} is even required for the weakest bound $\sup_N\int\E\|\U_N\|_{L^1([0,T]\times\M)}d\mu_\gm<\infty$. Moreover our argument for the proof of Theorem \ref{THM:GWP} requires $\U_N$ to be bounded in $L^2([0,T];H^{(-1)+}(\M))$, and the computation of the second moment in Lemma \ref{LEM:UEst} (ii) explicitly requires the constraint \eqref{Seiberg2b}.\\
\textup{(ii)} Finally, note that the construction of $\U$ in Proposition \ref{PROP:U} follows from the estimates in Lemma \ref{LEM:UEst} above, and in particular Lemma \ref{LEM:UEst} (iv) shows that $\U$ is  independent of the choice of the approximation by a multiplier $\psi\in \S(\R)$ with $\psi(0)=1$. Note that in \cite{DKRV,GRV}, only regularization by circle averaging is considered. This regularization procedure is however not covered by our results.
\end{remark}

\subsection{Construction of the LQG measure}\label{SUBS:Gibbs}

We now turn to the convergence properties of the truncated measure $\rho_{N,\gm}$ in \eqref{LQGN}. In order to prove Theorem \ref{THM:LQG}, we need several technical lemmas. These are mainly a unified adaptation of \cite{DKRV,DRV,GRV}, though our regularization is different from those works.

In the remaining of this section, we change of viewpoint and fix a realisation of $X_\gm$ on $(\O,\Prob)$ as in \eqref{GFF}, independent of $\xi_\gm$. A first observation is that $\U_N$ and its limit $\U$ given in Proposition \ref{PROP:U} are random \emph{positive} distributions, and thus $(\U_N)_{|t=0}$ and $\U_{|t=0}$ can be identified with random Radon measures on $\M$. Thus let us define the truncated {\it Liouville measure} as the random measure given by
\begin{align}
\Y_N(B) &\deff \int_{B}e^{-\pi\be^2C_\P}N^{-\frac{\be^2}2}e^{\be \P_NX_\gm(x)}H_N(x)d\Vg(x)\notag\\
& = \int_B\U_N(0,x)d\Vg(x)
\label{YN}
\end{align}
for any Borel set $B\subset \M$, where we redefine the function
\begin{align}\label{HN2}
H_N(x)\deff e^{2\pi\be\sum_{\ell=1}^L a_\ell (\P_N\otimes\P_N)\Gg(x_\ell,x)}.
\end{align}
We also define the closely related random measure
\begin{align}\label{XN}
\X_N(B) \deff \int_{B}e^{\be \P_NX_\gm(x)-\frac{\be^2}2\s_N(x,\gm)}d\Vg(x),
\end{align}
with $\s_N(x,\gm)$ as in \eqref{sN}.

We first recall some basic facts about Gaussian processes, namely Kahane's convexity inequality.
\begin{lemma}\label{LEM:Kahane}
Let $\{X_j\}_{j=1,...,n}$ and $\{Y_j\}_{j=1,...,n}$ be two centred Gaussian vectors such that
\begin{align*}
\E\big[X_jX_k\big]\le \E\big[Y_jY_k\big]
\end{align*}
for any $j,k=1,...,n$. Then for all positive numbers $p_j$ and any convex function $F:\R\to\R$ with at most polynomial growth, it holds
\begin{align*}
\E\Big[F\Big(\sum_{j=1}^np_je^{X_j-\frac12\E[X_j^2]}\Big)\Big] \le \E\Big[F\Big(\sum_{j=1}^np_je^{Y_j-\frac12\E[Y_j^2]}\Big)\Big]
\end{align*}
\end{lemma}
\begin{proof}
See for example \cite[Corollary A.2]{RV}.
\end{proof}

Next we state the existence of negative moments for $\X_N$. Note that taking $a_\ell=0$ for any $\ell$ and repeating the arguments of the proof of Proposition~\ref{PROP:U}, we have that 
\begin{align*}
\!:\,e^{\be X_\gm(x)}\!:\, \deff \lim_{N\to\infty} e^{\be \P_NX_\gm - \frac{\be^2}2\s_N}
\end{align*}
 is well-defined, where the convergence holds in $L^p(\O;B^{-\al(p)}_{p,p}(\M))$ for any $p\ge 1$, with $\al(p)$ as in \eqref{al} (with $a_\ell=0$).
\begin{lemma}\label{LEM:negmoment}
Let $0<\be^2<2$, and let $\X_N$ be defined as in \eqref{XN}. Then for any $a>0$, there exists $C>0$ such that for any $y_0\in\M$, any $0<r\ll\i(\M)$ and any $N\in\N$, it holds
\begin{align*}
\E\Big[\X_N\big(B(y_0,r)\big)^{-a}\Big]\le C.
\end{align*}
Moreover, we have the convergence 
\begin{align*}
\E\Big[\X_N\big(B(y_0,r)\big)^{-a}\Big]\too \E\Big[\X\big(B(y_0,r)\big)^{-a}\Big]
\end{align*}
 as $N\to\infty$, where
\begin{align*}
\X\big(B(y_0,r)\big) \deff \int_{B(y_0,r)}\!:\,e^{\be X_\gm(x)}\!:\,d\Vg(x).
\end{align*}
\end{lemma}
\begin{proof}
First, note that for any $B\subset\M$, $\X(B)$ as in Lemma \ref{LEM:negmoment} is a well-defined random variable. Indeed, thanks to the positivity of $e^{\be \P_NX_\gm(x)-\frac{\be^2}2\s_N(x)}$ it holds $0<\X_N(B)\le\X_N(\M)$ and this last term is in $L^2(\mu_\gm)$, uniformly in $N\in\N$. This follows by taking $a_\ell=0$ for any $\ell$, so that we have
\begin{align*}
\E\Big[\X_N(\M)^{2}\Big]&= \int_{\M}\int_{\M}e^{2\pi\be^2(\P_N\otimes\P_N)\Gg(x,y)}d\Vg(x)d\Vg(y) \\
&\les \int_{\M}\int_{\M}\dg(x,y)^{-\be^2}d\Vg(x)d\Vg(y) <\infty
\end{align*}
uniformly in $N\in\N$ since $0<\be^2<2$.

 As for the convergence of $\X_N(B)$, we have
\begin{align*}
\E\big|\X_{N_1}(B)-\X_{N_2}(B)\big|^2 &= \int_{B}\int_{B}\Big[e^{2\pi\be^2(\P_{N_1}\otimes\P_{N_1})\Gg(x,y)}-2e^{2\pi\be^2(\P_{N_1}\otimes\P_{N_2})\Gg(x,y)}\\
&\qquad\qquad+e^{2\pi\be^2(\P_{N_2}\otimes\P_{N_2})\Gg(x,y)}\Big]d\Vg(x)d\Vg(y)
\end{align*}
which converges to 0 as $N_1,N_2\to \infty$ by similar (simpler) computations as for Lemma \ref{LEM:UEst} (iii).

To prove Lemma \ref{LEM:negmoment}, we then follow the argument in \cite[Proposition 4]{Molchan}. It is of a general nature: provided that one has an inequality in law of the type
\begin{align*}
\X_N(B) \ge \sum_{k=1}^{K^2} \om_k\X_{N,k}(B)
\end{align*}
for some independent copies $\X_{N,k}(B)$ of $\X_N(B)$ and some independent random variables $\om_k$ admitting negative moments, then this argument implies that $\X_N(B)$ has in turn negative moments. Here, as in \cite{RV}, we will rely on Kahane's inequality in Lemma \ref{LEM:Kahane} and the main computations from the proof of Proposition 4(a) in \cite{Molchan} will then apply. 

Let us then take geodesic normal coordinates $\kk$ centred at $y_0\in\M$, and $0<r\ll \i(\M)$. We define $\Cb\subset\R^2$ to be the cube of side-length $r$ centred around 0; in particular we have $\Cb \subset B(0,r)=\kk^{-1}\Big(B(y_0,r)\Big)\subset 2\Cb$. Let us take an integer $K\in\N$ and denote by $\Cb_k$, $k=1,...,K^2$ the partition of $\Cb$ in cubes $\Cb_k$ of side-length $K^{-1}r$ centred around some $z_k\in\Cb$. We also define the smaller cubes $\Cb_k'$ centred around $z_k$ with side-length $(4K)^{-1}r$. Since $e^{\be \P_NX_\gm - \frac{\be^2}2\s_N}$ is positive, and since $\Vg$ is equivalent to the Lebesgue measure in $\kk$, we have
\begin{align}
\X_N\big(B(y_0,r)\big)&\ge \X_N\big(\kk^{-1}(\Cb)\big)\ge C\int_{\Cb}e^{\kk_\star\big[\be \P_NX_\gm - \frac{\be^2}2\s_N\big](z)}dz\notag\\
& =\sum_{k=1}^{K^2}\int_{\Cb_k}e^{\kk_\star\big[\be \P_NX_\gm - \frac{\be^2}2\s_N\big](z)+c}dz \ge \sum_{k=1}^{K^2}\int_{\Cb_k'}e^{\kk_\star\big[\be \P_NX_\gm - \frac{\be^2}2\s_N\big](z)+c}dz\notag\\
& = \sum_{k=1}^{K^2}(8K)^{-2}\int_{2\Cb}e^{\kk_\star\big[\be \P_NX_\gm - \frac{\be^2}2\s_N\big](z_k+(8K)^{-1}z)+c}dz.
\label{B1}
\end{align}

Next, we discretize the integral on $2\Cb$ in order to use Lemma \ref{LEM:Kahane}. Thus for any $J\in\N$ we subdivide again $2\Cb = \sqcup_{j=1}^{J^2}\Cb_{j}$ into cubes of side-length $J^{-1}r$ centred at some $\wt z_j\in 2\Cb$. Since for any fixed $N\in\N$, $\kk_\star\big[\be\P_NX_\gm-\frac{\be^2}2\s_N\big]$ has almost surely continuous paths, for any $k=1,...,K^2$ the following convergence holds almost surely:
\begin{align}\label{B1b}
\int_{2\Cb}e^{\kk_\star\big[\be \P_NX_\gm - \frac{\be^2}2\s_N\big](z_k+(8K)^{-1}z)+c}dz = \lim_{J\to\infty} \sum_{j=1}^{J^2}J^{-2}e^{\kk_\star\big[\be \P_NX_\gm - \frac{\be^2}2\s_N\big](z_k+(8K)^{-1}\wt z_j)+c}.
\end{align}

Now from the same computation as in Lemma \ref{LEM:covar} it holds for any $k_1,k_2=1,...,K^2$ and $j_1,j_2=1,...,J^2$:
\begin{align*}
&\E\Big[\big(\kk_\star\P_NX_\gm(z_{k_1}+(8K)^{-1}\wt z_{j_1})+c\big)\big(\kk_\star\P_NX_\gm(z_{k_2}+(8K)^{-1}\wt z_{j_2})+c\big)\Big]\\
&= 2\pi(\P_N\otimes\P_N)\Gg\Big(\kk\big(z_{k_1}+(8K)^{-1}\wt z_{j_1}\big),\kk\big(z_{k_2}+(8K)^{-1}\wt z_{j_2}\big)\Big)+c.
\end{align*}
From the two-sided bound of Corollary \ref{COR:GN}, we can bound this term with
\begin{align*}
-\log\big(\big|(z_{k_1}-z_{k_2})+(8K)^{-1}(\wt z_{j_1}-\wt z_{j_2})\big|+N^{-1}\big) + C_K
\end{align*}
for some constant $C_K>0$ independent of $J$ and $N$. 

In the case $k_1=k_2$, we have then (using Corollary \ref{COR:GN} again)
\begin{align*}
&\E\big[\kk_\star\P_NX_\gm(z_{k_1}+(8K)^{-1}\wt z_{j_1})\kk_\star\P_NX_\gm(z_{k_1}+(8K)^{-1}\wt z_{j_2})\big]\\
 &\qquad\le -\log\big(\big|\wt z_{j_1}-\wt z_{j_2}\big|+8KN^{-1}\big) + \log K + C_K\\
 &\qquad\le -\log\big(\big|\wt z_{j_1}-\wt z_{j_2}\big|+N^{-1}\big) + C_K\\
&\qquad\le \E\big[\kk_\star\P_NX_{\gm,k_1}(\wt z_{j_1})\kk_\star\P_NX_{\gm,k_1}(\wt z_{j_2})\big]+C_K
\end{align*}
where $X_{\gm,k}$, $k=1,...,K^{2}$ are independent copies of $X_\gm$, and $C_K>0$ is some constant independent of $N,J\in\N$.

On the other hand, in the case $k_1\neq k_2$, we have by choice of $z_k$ and $\wt z_j$ that $|z_{k_1}-z_{k_2}|\ge K^{-1}r$ and $|\wt z_{j_1}-\wt z_{j_2}|\le 2\sqrt{2}r$, which in turn implies
\begin{align*}
&-\log\big(\big|(z_{k_1}-z_{k_2})+(8K)^{-1}(\wt z_{j_1}-\wt z_{j_2})\big|+N^{-1}\big)\\
&\qquad\le -\log\big(\big||z_{k_1}-z_{k_2}|-(8K)^{-1}|\wt z_{j_1}-\wt z_{j_2}|\big|+N^{-1}\big)\\
&\qquad\le -\log\big(cK^{-1}+N^{-1}\big) \le \log K + C.
\end{align*}

Hence we arrive at
\begin{align}\label{B2}
&\E\Big[\big(\kk_\star\P_NX_\gm(z_{k_1}+(8K)^{-1}\wt z_{j_1})+c\big)\big(\kk_\star\P_NX_\gm(z_{k_2}+(8K)^{-1}\wt z_{j_2})+c\big)\Big]\notag\\
&\qquad\le \E\big[\kk_\star\P_NX_{\gm,k_1}(\wt z_{j_1})\kk_\star\P_NX_{\gm,k_2}(\wt z_{j_2})\big] + C_K
\end{align}
for some constant $C_K>0$ independent of $N,J\in\N$. If we then take some independent $h_k\sim\mathcal{N}(0,C_K)$ (and independent of $X_\gm$ and $(X_{\gm,k})_{k=1,...,K^2}$), we can then use \eqref{B1}-\eqref{B1b}-\eqref{B2} with Kahane's inequality in Lemma \ref{LEM:Kahane} to get for any decreasing and convex function $F : \R_+\to\R$,
\begin{align}\label{B3}
\E\Big[F\Big(\X_N\big(B(y_0,r)\big)\Big)\Big] &\le \E\Big[F\Big(\sum_{k=1}^{K^2}(8K)^{-2}\int_{\Cb}e^{\kk_\star\big[\be \P_NX_\gm - \frac{\be^2}2\s_N\big](z_k+(8K)^{-1}z)+c}dz\Big)\Big]\notag\\
&= \lim_{J\to\infty}\E\Big[F\Big(\sum_{k=1}^{K^2}\sum_{j=1}^{J^2}(8KJ)^{-2}e^{\kk_\star\big[\be \P_NX_\gm - \frac{\be^2}2\s_N\big](z_k+(8K)^{-1}\wt z_j)+c}\Big)\Big]\notag\\
&\le \lim_{J\to\infty}\E\Big[F\Big(\sum_{k=1}^{K^2}\sum_{j=1}^{J^2}(8KJ)^{-2}e^{\kk_\star\big[\be \P_NX_{\gm,k} - \frac{\be^2}2\s_N\big](\wt z_j)+\be h_k-\frac{\be^2}2C_K}\Big)\Big]\notag\\
&=\E\Big[F\Big(\sum_{k=1}^{K^2}(8K)^{-2}\int_{2\Cb}e^{\kk_\star\big[\be \P_NX_{\gm,k} - \frac{\be^2}2\s_N\big]+\be h_k-\frac{\be^2}2C_K}dz\Big)\Big]\notag\\
&\le \E\Big[F\Big(c\sum_{k=1}^{K^2}(8K)^{-2}\int_{B(y_0,r)}e^{\be \P_NX_{\gm,k} - \frac{\be^2}2\s_N+\be h_k-\frac{\be^2}2C_K}d\Vg\Big)\Big].
\end{align}

We can now prove the existence of negative moments for $\X_N\big(B(y_0,r)\big)$, uniformly in $N\in\N$. First, we will prove that $\E\Big[\X_N\big(B(y_0,r)\big)^{-\eps}\Big]<\infty$ uniformly in $N\in\N$ for some $0<\eps\ll 1$. To do so, let us introduce the moment-generating function
\begin{align*}
\varphi(t) \deff \E\Big[e^{-t\X_N\big(B(y_0,r)\big)}\Big],~~t\ge 0.
 \end{align*}
Note that by Fubini's theorem
\begin{align*}
\int_0^{\infty}t^{\eps -1}\varphi(t)dt = \G(\eps)\E\Big[\X_N\big(B(y_0,r)\big)^{-\eps}\Big],
\end{align*}
where $\G$ is the Gamma function. Since $\X_N\big(B(y_0,r)\big)$ is almost surely positive, we always have $\varphi(t)\le 1$ and that $\varphi$ is decreasing with $\varphi(t)\to 0$ as $t\to+\infty$. Hence it is enough to prove that $\varphi(t) \le ct^{-\eps'}$ for some $c>0$, $\eps'>\eps$ and $t\ge t_0$ large enough.

From \eqref{B3} with the decreasing convex function $x\mapsto e^{-tx}$ for any $t>0$ and the independence of $h_k$ and $X_k$ with Jensen's inequality, the inequality of arithmetic and geometric means, and Jensen's inequality again, we then have
\begin{align}\label{B4}
\varphi(t) &\le \prod_{k=1}^{K^2}\E\Big[\exp\Big(-ct(8K)^{-2}e^{\be h_k-\frac{\be^2}2C_K}\int_{B(y_0,r)}e^{\be \P_NX_{\gm,k} - \frac{\be^2}2\s_N}d\Vg\Big)\Big]\notag\\
&\le \sum_{k=1}^{K^2}K^{-2}\bigg(\E\Big\{\E\Big[\exp\Big(-ct(8K)^{-2}e^{\be h_k-\frac{\be^2}2C_K}\int_{B(y_0,r)}e^{\be \P_NX_{\gm,k} - \frac{\be^2}2\s_N}d\Vg\Big)\Big|h_k\Big]\Big\}\bigg)^{K^2}\notag\\
&\le \sum_{k=1}^{K^2}K^{-2}\E\Big[\varphi(ct(8K)^{-2}e^{\be h_k-\frac{\be^2}2C_K})^{K^2}\Big] = \E\Big[\varphi\big(ct(8K)^{-2}e^{\be h_1-\frac{\be^2}2C_K}\big)^{K^2}\Big].
\end{align}
In particular, using that $\varphi(t)\le 1$ for any $t>0$, we deduce from \eqref{B4} that
\begin{align*}
\varphi(t^2)&\le \E\Big[\mathbf{1}\big(\{ct(8K)^{-2}e^{\be h_1-\frac{\be^2}2C_K}<1\}\big)\varphi(ct^2(8K)^{-2}e^{\be h_1-\frac{\be^2}2C_K})^{K^2}\Big]\\
&\qquad + \E\Big[\mathbf{1}(ct(8K)^{-2}e^{\be h_1-\frac{\be^2}2C_K}\ge 1)\varphi(ct^2(8K)^{-2}e^{\be h_1-\frac{\be^2}2C_K})^{K^2}\Big]\\
&< \Prob\Big(e^{\be h_1-\frac{\be^2}2C_K}<\frac{(8K)^2}{ct}\Big) + \varphi(t)^{K^2}.
\end{align*}
Using then Chebychev's inequality, we can bound for any $p\ge 1$
\begin{align*}
\Prob\Big(e^{\be h_1-\frac{\be^2}2C_K}<\frac{(8K)^2}{ct}\Big) &\le \frac{(8K)^{4p}}{(ct)^{2p}}\E\big[e^{-2p(\be h_1-\frac{\be^2}2C_K)}\big] = \frac{(8K)^{4p}}{(ct)^{2p}}e^{\frac{\be^2}2C_K2p(2p+1)}\\
&\le C(K,p) t^{-2p}.
\end{align*}
All in all, provided that $K\ge 2$ we arrive at
\begin{align*}
\varphi(t^2) < q(t)\big[t^{-p}+\varphi(t)^2\big]
\end{align*}
where $q(t) = C(K,p)(t^{-p}+\varphi(t)^{K^2-2})\to 0$ as $t\to+\infty$. This is precisely the estimate at the bottom of p. 687 in \cite{Molchan}, and from there the same computations apply and give
\begin{align*}
\varphi(t)\le ct^{-\eps'}
\end{align*}
for some $0<\eps'\ll 1$ and all $t\ge t_0$ large enough. This shows that
\begin{align}\label{B5}
\sup_{N\in\N}\E\big[\X_N(B(y_0,r))^{-\eps}\big]<\infty
\end{align}
for some $0<\eps<\eps'\ll 1$. For some arbitrary $a>0$, we then use \eqref{B3} again with the decreasing convex function $x\mapsto x^{-a}$ and Young's inequality with the independence of $h_k$ and $X_{\gm,k}$ to get
\begin{align*}
\E\big[\X_N(B(y_0,r))^{-a}\big]&\le \E\Big[\Big(\sum_{k=1}^{K^2}(8K)^{-2}e^{\be h_k-\frac{\be^2}2C_K}\int_{B(y_0,r)}e^{\be\P_NX_{\gm,k}-\frac{\be^2}2\s_N}d\Vg\Big)^{-a}\Big]\\
&\les \prod_{k=1}^{K^2}\E\big[\big(e^{\be h_k-\frac{\be^2}2C_K}\big)^{-aK^{-2}}\big]\prod_{k=1}^{K^2}\E\Big[\Big(\int_{B(y_0,r)}e^{\be\P_NX_{\gm,k}-\frac{\be^2}2\s_N}d\Vg\Big)^{-aK^{-2}}\Big]\\
&\les\Big(\E\Big[\X_N(B(y_0,r))^{-aK^{-2}}\Big]\Big)^{K^2}.
\end{align*}
We can then iterate the inequality above $n$ times so that $aK^{-2n}<\eps$ and conclude from \eqref{B5}. This shows that all the negative moments are bounded uniformly in $N\in\N$.

The convergence of the negative moments then follows from their boundedness and the convergence of $\X_N(B(y_0,r))$ in $L^2(\mu_\gm)$: indeed using the mean value theorem and Cauchy-Schwarz inequality, we have
\begin{align*}
&\Big|\E\big[\X_{N_1}(B(y_0,r))^{-a}\big]-\E\big[\X_{N_2}(B(y_0,r))^{-a}\big]\Big|\\
&\les \E\Big[\big|\X_{N_1}(B(y_0,r))-\X_{N_2}(B(y_0,r))\big|\big(\X_{N_1}(B(y_0,r))^{-a-1}+\X_{N_2}(B(y_0,r))^{-a-1}\Big]\\
&\les \E\Big|\X_{N_1}(B(y_0,r))-\X_{N_2}(B(y_0,r))\Big|^2\Big(\E\big[\X_{N_1}(B(y_0,r))^{-2a-2}\big]+\E\big[\X_{N_2}(B(y_0,r))^{-2a-2}\big]\Big)\\
&\les \E\Big|\X_{N_1}(B(y_0,r))-\X_{N_2}(B(y_0,r))\Big|^2\too 0
\end{align*}
as $N_1,N_2\to\infty$. This proves Lemma \ref{LEM:negmoment}.
\end{proof}

We also show the following estimates on the random variables $\Y_N(B)$ defined in \eqref{YN} above.

\begin{lemma}\label{LEM:ZNM}
Suppose that the assumption \eqref{A1} and \eqref{Seiberg2b} hold. Then for any $a>0$ it holds
\begin{align*}
\lim_{N\to\infty}\E\Big[\Y_N(\M)^{-a}\Big] = \E\Big[\Y(\M)^{-a}\Big]<\infty,
\end{align*}
where 
\begin{align*}
\Y(\M) \deff \int_{\M} \U(0,x)d\Vg(x)
\end{align*}
is independent of the choice of $\psi\in\S(\R)$ with $\psi(0)=1$ used to define $\P_N=\psi(-N^{-2}\Dlg)$.
\end{lemma}
\begin{proof}
This is similar to \cite[Lemma 3.3]{DKRV}. Indeed, since the density is almost surely positive, we have that for any subset $B\subset \M$, $\Y_N(\M)^{-a}\le \Y_N(B)^{-a}$. In particular if we take $B$ to be any ball such that $x_\ell\notin B$ for any $\ell=1,...,L$, we have that $H_N$ in \eqref{HN2} is bounded from below on $B$ by some positive constant (depending on $B$), so we can estimate
\begin{align*}
\Y_N(B)^{-a} \les_B \X_N(B)^{-a}.
\end{align*}
The finiteness of the negative moments of the Gaussian multiplicative chaos given by Lemma \ref{LEM:negmoment} then ensures that $\E\Big[\Y_N(B)^{-a}\Big]$ is uniformly bounded.

Finally, the convergence of $\E\Big[\Y_N(\M)^{-a}\Big]$ follows from the previous step along with the mean value theorem, Cauchy-Schwarz inequality, and the convergence in $L^2(\O)$ of $\Y_N(\M)$ as in the proof of Lemma \ref{LEM:negmoment}. The independence on the choice of multiplier follows similarly. This proves Lemma~\ref{LEM:ZNM}.
\end{proof}

With  Lemma \ref{LEM:ZNM} at hand, we can finally give the proof of our first main result.

\begin{proof}[Proof of Theorem \ref{THM:LQG}]
The proof of Theorem \ref{THM:LQG} is a straightforward adaptation of the proofs of \cite[Theorem 4.3]{DRV} and \cite[Theorem 3.2]{DKRV}. In order to compare with our approach in Proposition \ref{PROP:U}, we detail the argument nonetheless.

\textbf{Case 1: if $\gm=\gm_0$.} We begin by treating the case where the metric has constant curvature. Recall from \eqref{LQGN} that $d\rho_{N,\gm_0} (X_0,\cj X)=R_N(X_0+\cj X)d\mu_0(X_0)d\cj X$ where, as in \eqref{RN}, the density $R_N$ is given by
\begin{align*}
&R_N (X_0+\cj X)\\ 
&\qquad= \Xi\exp\bigg\{\sum_{\ell=1}^L\Big(a_\ell\P_N\big(X_0+\cj X\big)(x_\ell) -\frac{a_\ell^2}{2}\big(\log N+2\pi C_\P\big)\Big)\notag\\
&\qquad\qquad -\frac{Q}{4\pi}\int_\M \Rg_0 (X_0+\cj X)dV_0 - \nu \int_{\M}e^{-\pi\be^2C_\P}N^{-\frac{\be^2}2}e^{\be \P_N(X_0+\cj X)}dV_0\bigg\}\\
&\qquad= \Xi\exp\bigg\{\sum_{\ell=1}^L\Big(a_\ell\P_NX_0(x_\ell)+a_\ell \cj X -\frac{a_\ell^2}{2}\big(\s_N(x_\ell)-2\pi\wt G_0(x_\ell,x_\ell)+o(1)\big)\Big)\\
&\qquad\qquad-\chi(\M)Q\cj X - \nu \int_{\M}e^{-\pi\be^2C_\P}N^{-\frac{\be^2}2}e^{\be \P_N(X_0+\cj X)}dV_0\bigg\}.
\end{align*} 
with $\s_N$ as in \eqref{sN} and where $o(1)$ is deterministic and uniform on $\M$. This follows by using Lemma \ref{LEM:GN3} as well as Gauss-Bonnet theorem \eqref{GB} with the fact that $X_0$ has mean zero.

We first show the convergence of the partition function $\ZZ_N(\gm_0)$ in \eqref{LQGN}. Let us define
\begin{align*}
\GG_{N,0}&\deff \int_{H^s_0(\M,\gm_0)} \bigg(\prod_{\ell=1}^Le^{a_\ell \P_NX_0(x_\ell)-\frac{a_\ell^2}{2} \big(\s_N(x_\ell)-2\pi\wt G_0(x_\ell,x_\ell)+o(1)\big)}\bigg)d\mu_0(X_0) \\
 &= \Big(\prod_{\ell=1}^Le^{2\pi\frac{a_\ell^2}2\wt G_0(x_\ell,x_\ell)+o(1)}\Big)e^{2\pi\sum_{\ell< k}a_\ell a_kG_0(x_\ell,x_k)+o(1)}>0.
\end{align*}
Note that for any $N\in\N$, the shift $2\pi \sum_{\ell=1}^L a_\ell(\P_N\otimes\Id)G_0(x_\ell,x)$ is smooth and with mean zero, and so in particular it belongs to the Cameron-Martin space $H^1_0(\M,\gm_0)$ of $\mu_0$. Thus by using Cameron-Martin theorem (see e.g. \cite[Proposition 2.26]{DPZ}), the Gaussian process 
\begin{align}
\wt X_0 = X_0 - 2\pi\sum_{\ell=1}^L a_\ell(\P_N\otimes\Id)G_0(x_\ell,x)
 \label{Girsanov}
\end{align}
has the same law as $X_0$ under the new probability measure
\begin{align*}
e^{-2\pi\sum_{\ell<k}a_\ell a_k(\P_N\otimes\P_N)G_0(x_\ell,x_k)}\exp\Big(\sum_{\ell=1}^La_\ell\P_NX_0(x_\ell) -\frac{a_\ell^2}{2}\s_N(x_\ell)\Big)d\mu_0.
\end{align*}
Indeed, in view of \eqref{GFF}, we can compute\footnote{Note that, in view of the white noise expansion \eqref{GFF} of the GFF, we indeed use $\big\{\frac{\varphi_n}{\sqrt{2\pi}\ld_n}\big\}_{n\ge 1}$ as an orthonormal basis of the Cameron-Martin space. This explains the coefficient $\frac1{2\pi}$ in the $H^1_0(\M,\gm_0)$ inner product.}
\begin{align*}
&\Big\langle X_0,2\pi \sum_{\ell=1}^L a_\ell(\P_N\otimes\Id)G_0(x_\ell,x)\Big\rangle_{H^1_0} - \frac12\big\|2\pi \sum_{\ell=1}^L a_\ell(\P_N\otimes\Id)G_0(x_\ell,x)\big\|_{H^1_0}^2\\
&= 2\pi\sum_{\ell=1}^La_\ell \sum_{n\ge 1}\frac{\ld_n^2}{2\pi}\frac{\sqrt{\ld_n}h_n(\om)}{\ld_n} e^{-N^{-2}\ld_n^2}\frac{\varphi_n(x_\ell)}{\ld_n^2}\\
&\qquad\qquad - \frac{(2\pi)^2}2\sum_{\ell,k=1}^La_\ell a_k\sum_{n\ge 1}\frac{\ld_n^2}{2\pi}e^{-2N^{-2}\ld_n^2}\frac{\varphi_n(x_\ell)}{\ld_n^2}\frac{\varphi_n(x_k)}{\ld_n^2}\\
&=\sum_{\ell=1}^La_\ell\P_NX_0(x_\ell)-2\pi\sum_{\ell,k=1}^L\frac{a_\ell a_k}2(\P_N\otimes\P_N)G_0(x_\ell,x_k).
\end{align*}
In particular for any continuous and bounded $F:\H^s_0\to\R$ the Cameron-Martin theorem implies\footnote{see also \eqref{CM} below.}
\begin{align*}
&\int_{H^s_0(\M,\gm_0)}F(X_0)e^{\sum_{\ell=1}^La_\ell\P_NX_0(x_\ell)-2\pi\sum_{\ell,k=1}^L\frac{a_\ell a_k}2(\P_N\otimes\P_N)G_0(x_\ell,x_k)}d\mu_0(X_0)\\
& = \int_{H^s_0(\M,\gm_0)}F\big(X_0+2\pi\sum_{\ell=1}^La_\ell(\P_N\otimes\Id)G_0(x_\ell,x)\big)d\mu_0(X_0).
\end{align*}
Therefore the partition function of $\rho_{N,\gm_0}$ can be expressed as
\begin{align*}
\ZZ_N(\gm_0)&=\int_{H^s_0(\M,\gm_0)}\int_{\R} R_N (X_0+\cj X)d\mu_0(X_0)d\cj X\notag\\
 &=\GG_{N,0}\Xi\int_{H^s_0(\M,\gm_0)}\int_{\R}\exp\bigg\{\Big[\sum_{\ell=1}^La_\ell-\chi(\M)Q\Big]\cj X-\nu e^{\be \cj X}\Y_N(\M)\bigg\}d\mu_0(X_0)d\cj X
 \end{align*}
 with $\Y_N(\M)$ as in \eqref{YN}. With the change of variable 
 \begin{align*}
 \cj X\mapsto \tau &\deff \nu e^{\be \cj X}\Y_N(\M),
  \end{align*}
   we continue with
 \begin{align}\label{bd-RN(X,m)}
\ZZ_N(\gm_0) &= \GG_{N,0}\Xi\be^{-1} \nu^{\be^{-1}\big[\chi(\M)Q)-\cj a\big]}\Big[\int_0^{+\infty}\tau^{\be^{-1}\big[\cj a -\chi(\M)Q\big]-1}e^{-\tau}d\tau\Big]\notag\\
&\qquad\times\Big[\int_{H^s_0(\M,\gm_0)} \Y_N(\M)^{\be^{-1}\big[\chi(\M)Q-\cj a\big]} d\mu_0\Big],
\end{align}
where we set $\cj a = \sum_{\ell=1}^La_\ell$.

The integral in $\tau$ in \eqref{bd-RN(X,m)} is $\G\big(\be^{-1}[\cj a -\chi(\M)Q]\big)$ which is finite under the assumption that the first Seiberg bound \eqref{Seiberg1} holds, and the integral with respect to $\mu_0$ converges by Lemma \ref{LEM:ZNM}. Since the $x_\ell$'s are all different, the constant $\GG_{N,0}$ also converges to the constant
\begin{align*}
\GG_0\deff \Big(\prod_{\ell=1}^Le^{2\pi\frac{a_\ell^2}2\wt G_0(x_\ell,x_\ell)}\Big)e^{2\pi\sum_{\ell< k}a_\ell a_kG_0(x_\ell,x_k)}<\infty,
\end{align*}
which shows that the whole partition function $\ZZ_{N}(\gm_0)$ converges to a non trivial real number.

Moreover, for any $F\in C_b(H^s_0(\M)\oplus\R)$, we can use the same argument as in the proof of \cite[Theorem 3.2]{DKRV}: we have
\begin{align*}
&\int_{H^s_0(\M,\gm_0)}\int_{\R} F(X_0+\cj X)R_N (X_0,\cj X)d\mu_0(X_0)d\cj X\\
&=\GG_{N,0}\Xi\be^{-1}\nu^{\be^{-1}\big[ \chi(\M)Q-\cj a\big]}\int_0^{+\infty}\tau^{\be^{-1}\big[\cj a- \chi(\M)Q\big]-1}e^{-\tau}\\
&\qquad\times\int_{H^s_0(\M,\gm_0)}F\Big(X_0+\be^{-1}\ln\frac{\tau}{\Y_N(\M)}+2\pi\sum_{\ell=1}^La_\ell(\P_N\otimes\P_N)G_0(x_\ell,x)\Big)\\
&\qquad\times\Y_N(\M)^{\be^{-1}\big[\chi(\M)Q-\cj a\big]} d\mu_0(X_0)d\tau.
\end{align*}
Then we note that $(\P_N\otimes\P_N)G_0(x_\ell,\cdot)$ converges to $G_0(x_\ell,\cdot)$ in $H^s_0(\M)$ and that $0<\Y_N(\M)<\infty$ almost surely and $\Y_N(\M)\to \Y(\M)$ in probability from the proof of Lemma \ref{LEM:ZNM}. Since the term with $F$ is then almost surely uniformly bounded and the integral too in view of the previous step, we can use dominated convergence to conclude that the last term above converges.

\textbf{Case 2: general metric.} In the case of a general metric $\gm = e^{f_0}\gm_0$ as in \eqref{conform}, we proceed as in \cite[Subsection 3.5]{DKRV}: we first make the change of variable $\cj X' = \cj X + \jb{X_\gm}_0$ and then use Lemma \ref{LEM:Gconf} to get
\begin{align*}
&\int_{H^s_0(\M,\gm)}\int_{\R} F\big(X_\gm+\cj X\big)R_N (X_\gm,\cj X)d\mu_\gm(X_\gm)d\cj X\\
&=\Xi(\gm)\int_{H^s_0(\M,\gm)}\int_{\R} F\big(X_\gm-\jb{X_\gm}_0+\cj X)\exp\bigg\{- \frac{Q}{4\pi}\int_{\M}\Rg_\gm\big(X_\gm-\jb{X_\gm}_0+\cj X\big)d\Vg\\
&\qquad +\sum_{\ell=1}^L\Big(a_\ell\P_{N,\gm}\big(X_\gm- \jb{X_\gm}_0+\cj X\big)(x_\ell) - \frac{a_\ell^2}2\big(\log N+2\pi C_\P)\Big)\\
&\qquad\qquad - \nu\int_{\M}e^{-\pi\be^2C_\P}N^{-\frac{\be^2}2}e^{\be \P_{N,\gm}\big(X_\gm-\jb{X_\gm}_0 +\cj X\big)}d\Vg\bigg\} d\mu_\gm(X_\gm)d\cj X\\
&=\Xi(\gm)\int_{H^s_0(\M,\gm)}\int_{\R} F\big(X_0+\cj X)\exp\Big\{ - \frac{Q}{4\pi}\int_{\M}\Rg_\gm\big(X_0+\cj X\big)d\Vg\\
&\qquad\times+\sum_{\ell=1}^L\Big(a_\ell\P_{N,\gm}\big(X_0+\cj X\big)(x_\ell) - \frac{a_\ell^2}2\big(\log N +2\pi C_\P\big)\Big) \\
&\qquad\qquad- \nu\int_{\M}e^{-\pi\be^2C_\P}N^{-\frac{\be^2}2}e^{\be \P_{N,\gm}\big(X_0 +\cj X\big)}d\Vg\Big\} d\mu_0(X_0)d\cj X
\end{align*}
where we write $\P_{N,\gm} = e^{N^{-2}\Dlg}$ to emphasize the dependence in the metric.

Then by Lemmas \ref{LEM:GN3} and \ref{LEM:GN5} we have
\begin{align*}
&\int_{H^s_0(\M,\gm)}\int_{\R} F\big(X_\gm+\cj X\big)R_N (X_\gm,\cj X)d\mu_\gm(X_\gm)d\cj X\\
&=\Xi(\gm)\int_{H^s_0(\M,\gm)}\int_{\R} F\big(X_0+\cj X)\exp\bigg\{-\frac{Q}{4\pi}\int_{\M}\Rg_\gm X_0d\Vg\\
&\qquad+\big[\cj a-\chi(\M)Q\big]\cj X+\sum_{\ell=1}^L\Big[a_\ell\P_{N,\gm}X_0(x_\ell) \\
&\qquad- \frac{a_\ell^2}2\big(2\pi(\P_{N,\gm}\otimes\P_{N,\gm})G_0(x_\ell,x_\ell)-\frac12 f_0(x_\ell)-2\pi\wt G_0(x_\ell,x_\ell)+o(1)\big)\Big]\\
&\qquad - \nu e^{\be\cj X}\int_{\M}e^{-\pi\be^2C_\P}N^{-\frac{\be^2}2}e^{\be \P_{N,\gm}X_0+f_0}dV_0\bigg\} d\mu_0(X_0)d\cj X
\end{align*}
where we used again that $\cj X$ is constant and Gauss-Bonnet's formula \eqref{GB}.

To deal with the curvature term $\int \Rg_\gm X_0d\Vg$, we apply again Cameron-Martin's theorem so that 
\begin{align*}
X_0(x) + \frac{Q}{2}\int_{\M}\Rg_\gm(y)G_0(x,y)d\Vg(y)
\end{align*}
is still a mass-less GFF under
\begin{align*}
\exp\Big\{-\frac{Q}{4\pi}\int_{\M}\Rg_\gm X_0d\Vg-\frac12\int \Big|-\frac{Q}{4\pi}\int_{\M}\Rg_\gm X_0d\Vg\Big|^2d\mu_0(X_0)\Big\}d\mu_0.
\end{align*}
In view of \eqref{conform}, the shift in $X_0$ can be computed as
\begin{align*}
\frac{Q}{2}\int_{\M}\Rg_\gm(y)G_0(x,y)d\Vg(y)&=\frac{Q}{2}\int_{\M}G_0(x,y)\big(\Rg_0-\Dl_0f_0(y))dV_0(y)\\
& = \frac{Q}2\big(f_0(x)-\jb{f_0}_0\big)
\end{align*}
by using that $\Rg_0$ is constant, that $G_0(x,\cdot)$ has mean zero, and \eqref{Green2}. Similarly, we also have
\begin{align*}
\int_{H^s_0(\M,\gm)} \Big|\int_{\M}\Rg_\gm X_0d\Vg\Big|^2d\mu_0(X_0)&=2\pi\int_{\M}\int_{\M}\Rg_\gm(x)\Rg_\gm(y)G_0(x,y)d\Vg(x)d\Vg(y)\\
&=2\pi\int_{\M}\big(\Rg_0-\Dl_0f_0(x)\big)\big(f_0(x)-\jb{f_0}_0\big)dV_0(x)\\
&=2\pi\int_{\M}\big|\nabla_0f_0\big|^2dV_0 .
\end{align*}
With the change of variable $\cj X\mapsto \cj X + \frac{Q}2\jb{f_0}_0$, this yields
\begin{align*}
&\int_{H^s_0(\M,\gm)}\int_{\R} F\big(X_\gm+\cj X\big)R_N (X_\gm,\cj X)d\mu_\gm(X_\gm)d\cj X\\
&=\Xi(\gm)\int_{H^s_0(\M,\gm)}\int_{\R} F\Big(X_0-\frac{Q}2f_0+\cj X\Big)e^{\frac{Q^2}{16\pi}\int_{\M}|\nabla_0f_0|^2dV_0+\frac{Q^2}2\chi(\M)\jb{f_0}_0-\sum_{\ell=0}^L\big(\frac{Qa_\ell}2-\frac{a_\ell^2}4\big)f_0(x_\ell)+o(1)}\\
&\qquad\times\exp\bigg\{\sum_{\ell=1}^L\Big[a_\ell\P_{N,0}X_0(x_\ell) - \frac{a_\ell^2}2\big(2\pi(\P_{N,\gm}\otimes\P_{N,\gm})G_0(x_\ell,x_\ell)-2\pi\wt G_0(x_\ell,x_\ell)+o(1)\big)\Big]\\
&\qquad +\big[\cj a -  \chi(\M)Q\big]\cj X - \nu e^{\be\cj X}\int_{\M}e^{-\pi\be^2C_\P}N^{-\frac{\be^2}2}e^{\be \P_{N,\gm}X_0 + (1-\be\frac{Q}2)f_0+o(1) }dV_0\bigg\} d\mu_0(X_0)d\cj X.
\end{align*}
Proceeding then as in the previous case and replacing $X_0$ by the shifted field
\begin{align*}
\wt X_0 = X_0 - 2\pi\sum_{\ell=1}^La_\ell(\P_{N,\gm}\otimes\Id)G_0(x_\ell,x)
\end{align*}
which, by virtue of the Cameron-Martin theorem, has the same law as $X_0$ under the new measure
\begin{align*}
e^{-2\pi\sum_{\ell<k}a_\ell a_k(\P_{N_\gm}\otimes\P_{N,\gm})G_0(x_\ell,x_k)}\exp\Big(\sum_{\ell=1}^La_\ell\P_{N,\gm}X_0(x_\ell) -\frac{a_\ell^2}{2}2\pi(\P_{N,\gm}\otimes\P_{N,\gm})G_0(x_\ell,x_\ell)\Big)d\mu_0,
\end{align*}
we get
\begin{align*}
&\int_{H^s_0(\M,\gm)}\int_{\R} F\big(X_\gm+\cj X\big)R_N (X_\gm,\cj X)d\mu_\gm(X_\gm)d\cj X\\
&=\Xi(\gm)\GG_{N,0}(1+o(1))\int_{H^s_0(\M,\gm)}\int_{\R} F\Big(X_0-\frac{Q}2f_0+\cj X+2\pi\sum_{\ell=1}^La_\ell(\P_{N,\gm}\otimes\P_{N,\gm})G_0(x_\ell,x)\Big)\\
&\qquad\times e^{\frac{Q^2}{16\pi}\int_{\M}|\nabla_0f_0|^2dV_0+\frac{Q^2}2\chi(\M)\jb{f_0}_0-\sum_{\ell=0}^L\big(\frac{Qa_\ell}2-\frac{a_\ell^2}4\big)f_0(x_\ell)+o(1)}\exp\bigg\{\big[\cj a -  \chi(\M)Q\big]\cj X\\
&\qquad - \nu e^{\be\cj X}\int_{\M}e^{-\pi\be^2C_\P}N^{-\frac{\be^2}2}e^{\be \P_{N,\gm}X_0 +2\pi\sum_{\ell=1}^La_\ell(\P_{N,\gm}\otimes\P_{N,\gm})G_0(x_\ell,x)}\\
&\qquad\qquad\times e^{(1-\be\frac{Q}2)f_0+o(1) }dV_0\bigg\} d\mu_0(X_0)d\cj X,
\end{align*}
with again $o(1)$ being deterministic.

Using the fact that the estimate of Lemma \ref{LEM:UEst} (iii) also holds for
\begin{align*}
\U_{N,\gm,0}\deff e^{-\pi\be^2C_\P}N^{-\frac{\be^2}2}e^{\be\P_{N,\gm}\<1>_0 +2\pi\sum_{\ell=1}^La_\ell(\P_{N,\gm}\otimes\P_{N,\gm})G_0(x_\ell,x)}
\end{align*} 
and using also Lemma \ref{LEM:GN5}, we find that
\begin{align*}
\U_{N,\gm,0} \too e^{\frac{\be^2}4f_0}\U
\end{align*}
as $N\to\infty$ in the same topology as in Proposition \ref{PROP:U}. With the identity $\frac{\be^2}4+1-\be\frac{Q}2 = 0$ (by definition of $Q$), and the same argument as in the proof of Lemma \ref{LEM:ZNM}, we deduce that
\begin{align*}
\int_{\M}\U_{N,\gm,0}(0)e^{(1-\be\frac{Q}2)f_0+o(1) }dV_0 \too \Y(\M)
\end{align*}
as $N\to\infty$ in $L^2(\mu_0)$. This shows that
\begin{align*}
\int_{H^s_0(\M,\gm)}\int_{\R} F\big(X_\gm+\cj X\big)R_N (X_\gm,\cj X)d\mu_\gm(X_\gm)d\cj X
\end{align*}
converges to
\begin{align*}
&\Xi(\gm)e^{\frac{Q^2}{16\pi}\int_{\M}|\nabla_0f_0|^2dV_0+\frac{Q^2}2\chi(\M)\jb{f_0}_0-\sum_{\ell=0}^L\big(\frac{Qa_\ell}2-\frac{a_\ell^2}4\big)f_0(x_\ell)}\\
&\qquad\qquad\times\Xi(\gm_0)^{-1}\int_{H^s_0(\M,\gm_0)}\int_{\R} F\big(X_0+\cj X-\frac{Q}2f_0\big)d\rho_{\{a_\ell,x_\ell\},\gm_0}(X_0,\cj X).
\end{align*}

Finally, the identity
\begin{align*}
\frac{Q^2}2\chi(\M)\jb{f_0}_0 = \frac{Q^2}{16\pi}\int_{\M}\Rg_0f_0dV_0
\end{align*}
given by Gauss-Bonnet \eqref{GB}, as well as the use of \eqref{Xi3} to simplify $\Xi(\gm)\Xi(\gm_0)^{-1}$ shows \eqref{anomaly} and concludes the proof of Theorem \ref{THM:LQG}.
\end{proof}

\section{Proof of Theorem \ref{THM:GWP}}\label{SEC:GWP}
We know turn to the construction of the dynamics preserving the correlations \eqref{correlation2}.

\subsection{Construction of the dynamics}
Recall that we look for a solution of \eqref{heat2} under the form \eqref{DPD} with $z$ as in \eqref{z} and where $v_N$ solves \eqref{vN}. In particular, note that $z\in C(\R_+;C^{\infty}(\M))$, $\Prob$-almost surely (and for any $\cj X\in\R$). In this subsection, we thus focus on the equation:
\begin{align}
\begin{cases}
\dt v_N -\frac1{4\pi}\Dlg v_N + \frac12\nu\be \P_N\Big[e^{\be \P_Nz} e^{\be \P_Nv_N} \U_N\Big]  = 0\\
v|_{t=0}=0,
\end{cases}
\label{v1}
\end{align}

\noi
where  $z $ is a given {\it deterministic} function in $C(\R_+;C^{\infty}(\M))$, $\U_N$ is a given {\it deterministic} positive space-time distribution which converges to $\U$ in $L^2([0,T];H^{-1+\eps}(\M))$ for any $T>0$ and $0<\eps\ll1$, $\P_N$ is as in \eqref{PN},
and $\nu > 0$. In this case, recall from \cite{ORW} that the ``sign-definite structure" allows us to rewrite the equation \eqref{v1} as
\begin{align}
\begin{cases}
\dt v_N -\frac1{4\pi}\Dlg v_N + \frac12\nu\be \P_N\Big[e^{\be \P_Nz} \NN(\be \P_Nv_N)\U_N\Big]  = 0\\
v|_{t=0}=0,
\end{cases} 
\label{v2}
\end{align}

\noi
where $\NN$ is a suitable smooth, bounded and Lipschitz nonlinearity. 
Indeed, this follows by writing~\eqref{v1} in the Duhamel formulation
\begin{align}
v_N(t,x)=-\frac12\nu\be\int_0^t\int_{\M} P_\gm(t-t',x,y) \P_N\Big[e^{\be \P_Nz} e^{\be \P_Nv_N}\U_N \Big](t',y)d\Vg(y)dt',
\label{v3}
\end{align}
and by using that the heat kernel $P_\gm$ is positive, which also implies that $\P_N$ \eqref{PN} is positivity preserving, and that $\U_N$ is a positive distribution with $\nu>0$, from which we infer on \eqref{v3} that $\be v_N\le 0$. In particular if $\NN\in\S(\R)$ satisfies $\NN(x)\equiv e^x$ for $x\le 0$ we see that \eqref{v3} is then equivalent to \eqref{v4}. We also consider the limit equation
 \begin{align}
\label{v4}
\begin{cases}
\dt v -\frac1{4\pi}\Dlg v +  \frac 12 \nu \be
\Big[e^{\be z }  \NN(\be v) \U\Big]=0\\
v_N|_{t = 0} = 0
\end{cases}
\end{align}
which is again equivalent to \eqref{v} due to the positivity of $\U$ and $\nu$.
\begin{proposition}\label{PROP:flow}
Let  $T>0$, $0<\eps\ll 1$, $z\in C([0,T]; C^2(\M))$
and  $\U\in L^2([0,T];H^{-1+\eps}(\M))$
be a positive distribution.
Suppose that a sequence $\{\U_N\}_{N \in \N}$ 
of smooth non-negative functions
converges to  
$\U$ in $L^2([0,T];H^{-1+\eps}(\M))$.
Then, for $0<\dl\ll \eps$, the Cauchy problem \eqref{v2} is well-posed in $\XX_T^\dl$ for all $N\in\N$, and the corresponding solution $v_N$ converges to 
a limit $v$ in $\XX_T^\dl$. Furthermore, 
 the limit $v$ is the unique solution to~\eqref{v4}
 in the energy class $\XX_T^0$.
\end{proposition}
Here the spaces $\XX_T^s$ are defined for any $s\in\R$  by
\begin{align}\label{energyZ}
\XX_T^s \deff C([0,T];H^s (\M))\cap L^2([0,T];H^{1+s}(\M)).
\end{align}

Using Proposition \ref{PROP:U}, we se that the condition $\U\in L^2([0,T];H^{-1+\eps}(\M))$ for some (small) $\eps>0$ gives the condition \eqref{A2} in Theorem \ref{THM:GWP}, and the convergence of $\U_N$ to $\U$ in $L^2\big(\mu_\gm\otimes\Prob;L^2([0,T];H^{-1+\eps}(\M))\big)$ implies the convergence in measure of $v_N$ to $v$ in \eqref{DPD}. Thus Theorem \ref{THM:GWP} (i) will be established once we prove Proposition \ref{PROP:flow}.

\begin{proof}[Proof of Proposition \ref{PROP:flow}]
Let us define the nonlinear operator
\begin{align*}
\Phi_N = \Phi_{\P_Nz,\U_N}\deff v_N\mapsto -\frac12\nu\be\int_0^te^{\frac{t-t'}2\Dlg}\P_N\big[e^{\be \P_Nz}\NN(\be \P_Nv)\U_N\big](t')dt'
\end{align*}
and similarly for $\Phi = \Phi_{z,\U}$. In the following we fix $T>0$, $z\in C([0,T];C^{\infty}(\M))$ and smooth positive functions $\U_N$ satisfying $\U_N\to \U$ in $L^2([0,T];H^{-1+\eps}(\M))$ for some $0<\eps\ll 1$ and some positive distribution $\U\in L^2([0,T];H^{-1+\eps}(\M))$. We mainly follow the argument in \cite[Section 5]{ORW}.

\textbf{Step 1: global well-posedness of \eqref{v2}.}
For $0<\tau\le T$, we estimate for any $N\in\N$ and any $v_N\in C([0,\tau]\times\M)$:
\begin{align}
\big\|\Phi_N(v_N)\big\|_{C_{\tau,x}} &\les\bigg\|\int_0^t(t-t')^{-\frac12}\Big\|\P_N\big[e^{\be \P_Nz}\NN(\be \P_Nv_N)\U_N\big](t')\Big\|_{L^2 }dt'\bigg\|_{L^{\infty}_\tau}\notag\\
&\les \tau^{\frac12-}\Big\|e^{\be \P_Nz}\NN(\be \P_Nv_N)\U_N\Big\|_{L^{\infty}_\tau L^2 }\notag\\
&\les \tau^{\frac12-}e^{C \|z\|_{L^{\infty}_{T,x}}}\|\U_N\|_{L^{\infty}_T L^2 }
\label{R1}
\end{align}
where we used Schauder's estimate \eqref{Schauder} and Young's inequality with the boundedness of $\NN$ and the uniform boundedness of $\P_N : L^p(\M)\to L^p(\M)$ for any $p$. Note that for fixed $N\in\N$, $\U_N$ is smooth, so that the right-hand side of \eqref{R1} is indeed finite.

We can also estimate similarly for $v_N,w_N\in C([0,\tau]\times\M)$
\begin{align}
&\|\Phi_N(v_N)-\Phi_N(w_N)\|_{C_{\tau,x}}\notag \\
&\les \bigg\|\int_0^t (t-t')^{-\frac12}
\Big\|\P_N\Big[e^{\be \P_Nz}\big[\NN(\be \P_Nv_N)-\NN(\be \P_Nw_N)\big]\U_N \Big)(t')\Big\|_{L^2 }dt'\bigg\|_{L^{\infty}_\tau}\notag\\
&\les \tau^{\frac12-}e^{C \|z\|_{L^{\infty}_{T,x}}}\big\|\NN(\be \P_Nv_N)-\NN(\be \P_Nw_N)\big\|_{L^{\infty}_{\tau,x}}\|\U_N\|_{L^{\infty}_T L^2 }\notag\\
&\les\tau^{\frac12-}e^{C \|z\|_{L^{\infty}_{T,x}}}\|v_N-w_N\|_{L^{\infty}_{\tau,x}}\|\U_N\|_{L^{\infty}_T L^2 }
 \label{R1b}
\end{align}
where in the last step we used the mean value theorem and wrote 
\begin{align}
\NN(\be \P_Nv_N)-\NN(\be  \P_Nw_N) = \be \P_N(v_N -  w_N) \int_0^1\NN'\big(\theta \be \P_Nv_N+(1-\theta)\be\P_N w_N\big)d\theta
\label{F2}
\end{align}
with $\NN'$ bounded. From the estimates \eqref{R1} and \eqref{R1b} we deduce that for $\tau_N=\tau_N\big(\|z\|_{L^{\infty}_{T,x}},\|\U_N\|_{L^{\infty}_TL^2 }\big)>0$, $\Phi_N$ is a contraction on a ball of $C([0,\tau_N]\times\M)$ and thus admits a unique fixed point $v_N$ in this ball, which is the unique solution of \eqref{v2} on $[0,\tau_N^*)$ where $0<\tau_N^*\le T$ is the maximal time of existence of $v_N$. Moreover the estimate \eqref{R1} shows that $\|v_N\|_{C_{\tau,x}}$ stays bounded as $\tau\nearrow \tau_N^\star$, so that we can iterate the fixed-point argument to obtain $\tau_N^\star = T$.

\textbf{Step 2: convergence of $v_N$.} Let $0<\dl \ll \eps$. 
Then, proceeding as above and using the Schauder estimate (Lemma~\ref{LEM:heatker}), Lemma~\ref{LEM:posprod}, 
and Young's inequality with the boundedness of $\NN$, we have
\begin{align}
\begin{split}
\|v_N\|_{C_TH^{2\dl} } 
&\les \bigg\|\int_0^t (t-t')^{-\frac{1+2\dl-\eps}{2}}
\big\|\P_N\big(e^{\be \P_Nz}\NN(\be \P_Nv_N)\U_N \big)(t')\big\|_{H^{-1+\eps} }dt'\bigg\|_{L^{\infty}_T}\\
&\les \big\|e^{\be \P_Nz}\NN(\be \P_Nv_N)\big\|_{L^{\infty}_{T,x}}
\bigg\|\int_0^t (t-t')^{-\frac{2+2\dl-\eps}{2}}\|\U_N(t')\|_{H^{-1+\eps} }dt'\bigg\|_{L^{\infty}_T}\\
 & \les e^{C\|z\|_{L^{\infty}_{T,x}}} \|\U_N\|_{L^2_T H^{-1+\eps} },
\end{split}
\label{R2}
\end{align}

\noi
uniformly in $N \in \N$. We can bound similarly

\begin{align}
\begin{split}
\|v_N\|_{L^2_TH^{1+2\dl} } 
&\les \bigg\|\int_0^t (t-t')^{-\frac{2+2\dl-\eps}{2}}
\big\|\P_N\big(e^{\be \P_Nz}\NN(\be \P_Nv_N)\U_N \big)(t')\big\|_{H^{-1+\eps} }dt'\bigg\|_{L^{2}_T}\\
&\les \big\|e^{\be \P_Nz}\NN(\be \P_Nv_N)\big\|_{L^{\infty}_{T,x}}
\bigg\|\int_0^t (t-t')^{-\frac{2+2\dl-\eps}{2}}\|\U_N(t')\|_{H^{-1+\eps} }dt'\bigg\|_{L^2_T}\\
 & \les e^{C\|z\|_{L^{\infty}_{T,x}}} \|\U_N\|_{L^2_T H^{-1+\eps} },
\end{split}
\label{R2b}
\end{align}

and 

\begin{align}
\begin{split}
\|\dt v_N \|_{L^2_TH^{-1+2\dl} } &= \Big\|\tfrac12\Dlg v_N 
-\tfrac{1}{2}\nu \be  \P_N\big[e^{\be\P_Nz}\NN(\be \P_Nv_N)\U_N\big] \Big\|_{L^2_TH^{-1+2\dl} }\\
&\les \|v_N\|_{L^2_TH^{1+2\dl} }+\big\| e^{\be \P_Nz}\NN(\be \P_Nv_N)\U_N\big\|_{L^2_TH^{-1+\eps} }\\
& \les 
e^{C\| z\|_{L^{\infty}_{T,x}}} \big\| \U_N\big\|_{L^2_TH^{-1+\eps} }, 
\end{split}
\label{R3}
\end{align}

\noi
uniformly in $N \in \N$.

For any $s \in \R$, we define
  $\wt\XX^{s}_T$ by 
\begin{align*}
\wt\XX^{s}_T
& = 
 \big\{v \in 
 \XX^s_T : \dt v\in L^2([0,T];H^{-1+s}(\M))\big\}.
\end{align*} 
Then we deduce from \eqref{R2}, \eqref{R2b}, and \eqref{R3}
along with the convergence of $\U_N$
to $\U$ in $L^2([0,T];H^{-1+\eps}(\M))$, that $\{v_N \}_{N \in \N}$
is bounded in $\wt\XX^{2\dl}_T$. Moreover, it follows from  Rellich's lemma and the Aubin-Lions lemma (see e.g. \cite[Corollary 4 on p.\,85]{Simon})
that 
the embedding of $\wt\XX^{2\dl}_T\subset \XX^\dl_T$ is compact.
Hence, 
there exists a subsequence $\{v_{N_k}\}_{k \in \N}$
converging to some limit $ v$
in $\XX_T^\dl$.

Next, we show that the limit $ v$ satisfies \eqref{v4}. Since $v_N$ satisfies \eqref{v2}, it is enough to prove the convergence of 
 $\Phi_{N_k}(v_{N_k})$ to $\Phi( v)$ in $\D'([0,T]\times\M)$. We thus estimate
\begin{align}
&\|\Phi_{N_k}(v_{N_k})-\Phi( v)\|_{L^1_TB^{-1 + \eps}_{1,1}}\notag\\
&\qquad \les \bigg\|
\int_0^t e^{\frac{t-t'}2\Dlg}\big[\P_{N_k}-\Id\big]\big[e^{\be \P_{N_k}z}\NN(\be \P_{N_k}v_{N_k})\U_{N_k}\big](t')dt'\bigg\|_{L^1_{T}H^{-1+ \eps}}\notag\\
&\qquad\qquad +\bigg\|
\int_0^t e^{\frac{t-t'}2\Dlg}\big[\big(e^{\be \P_{N_k}z}-e^{\be z}\big)\NN(\be \P_{N_k}v_{N_k})\U_{N_k}\big](t')dt'\bigg\|_{L^1_{T}H^{-1+ \eps}}
\notag\\
&\qquad\qquad+\bigg\|
\int_0^t e^{\frac{t-t'}2\Dlg}\big[e^{\be z}\NN(\be \P_{N_k}v_{N_k})(\U_{N_k}-\U)\big](t')dt'\bigg\|_{L^1_{T}B^{-1+ \eps}_{1,1}}
\notag\\
&\qquad\qquad +\bigg\|\int_0^te^{\frac{t-t'}2\Dlg}\big[e^{\be z}\big(\P_{N_k}v_{N_k}-v\big)\cj\NN(\P_{N_k}v_{N_k},v)
\U \big](t') dt'\bigg\|_{L^1_{T}H^{-1 + \eps}}\notag\\
&\qquad=: \1+\II+\III+\IV,
\label{Dv0}
\end{align}
where we used \eqref{F2} and wrote
\begin{align*}
\cj \NN(\P_{N_k}v_{N_k},v)\deff \int_0^1\NN'\big(\theta\be \P_{N_k}v_{N_k}+(1-\theta)\be v\big)d\theta.
\end{align*}

For the first term, note that we have for any $f\in C^{\infty}(\M)$ and any $N\in\N$
\begin{align*}
\Big\|\big[\P_N-\Id\big]f\Big\|_{H^{-1}}^2 & \les \sum_{n\ge 0}\big|e^{-N^{-2}\ld_n^2}-1\big|^2\jb{\ld_n}^{-2}\langle f,\varphi_n\rangle_\gm^2 \les N^{-2\eps}\sum_{n\ge 0}\jb{\ld_n}^{-2+2\eps}\langle f,\varphi_n\rangle_\gm^2\\
&\les N^{-2\eps}\|f\|_{H^{-1+\eps}}^2.
\end{align*}
Therefore the previous remark with Schauder estimate \eqref{Schauder}, Young's inequality and Lemma \ref{LEM:posprod} yield
\begin{align}
\1 &\les N_k^{-\eps}\bigg\|\int_0^t(t-t')^{-\frac{\eps}2}\Big\|\big[e^{\be \P_{N_k}z}\NN(\be \P_{N_k}v_{N_k})\U_{N_k}\big](t')\Big\|_{H^{-1+\eps} }dt'\bigg\|_{L^1_T}\notag\\
&\les N_k^{-\eps}e^{C\|z\|_{L^{\infty}_{T,x}}}\big\|\U_{N_k}\big\|_{L^2_TH^{-1+\eps} }.
\label{Dv1}
\end{align}

\noi
Similarly, we bound
\begin{align}
\II &\les \Big\|\big(e^{\be \P_{N_k}z}-e^{\be z}\big)\NN(\be \P_{N_k}v_{N_k})\U_{N_k}\Big\|_{L^1_TH^{-1+\eps} }\notag\\
&\les \big\|e^{\be \P_{N_k}z}-e^{\be z}\big\|_{L^{\infty}_{T,x}}\big\|\U_{N_k}\big\|_{L^2_TH^{-1+\eps} }\notag\\
&\les N_k^{-\eps}\|z\|_{L^{\infty}_TH^{1+2\eps} }e^{C\|z\|_{L^{\infty}_{T,x}}}\big\|\U_{N_k}\big\|_{L^2_TH^{-1+\eps} }
\label{Dv2}
\end{align}
where we used the mean value theorem, Sobolev inequality and the same argument as above to bound
\begin{align*}
\big\|e^{\be \P_{N_k}z}-e^{\be z}\big\|_{L^{\infty}_{T,x}} & \les \big\|\big[\P_{N_k}-\Id\big]z\big\|_{L^{\infty}_{T,x}}e^{C\|z\|_{L^{\infty}_{T,x}}}\\
&\les \big\|\big[\P_{N_k}-\Id\big]z\big\|_{L^{\infty}_TH^{1+\eps} }e^{C\|z\|_{L^{\infty}_{T,x}}}\\
&\les N_k^{-\eps}\|z\|_{L^{\infty}_TH^{1+2\eps} }e^{C\|z\|_{L^{\infty}_{T,x}}}.
\end{align*}

As for $\III$, we use again the Schauder estimate \eqref{Schauder} with Young's inequality, but due to the lack of positivity of $\U_{N_k}-\U$ we use the product estimate of Lemma \ref{LEM:Besov} (iii) in place of Lemma \ref{LEM:posprod} to get
 \begin{align}
 \begin{split}
\III &\les 
\big\|e^{\be z}\NN(\be \P_{N_k}v_{N_k})(\U_{N_k}-\U)\big\|_{L^1_T B^{-1+\eps}_{1,1} }\\
& \les \big\|e^{\be z}\NN(\be \P_{N_k}v_{N_k})\big\|_{L^2_TH^{1-\frac{\eps}2} }\|\U_{N_k}-\U\|_{L^2_TH^{-1+\eps} }\\
&\les \big\|e^{\be z}\|_{L^{\infty}_T B^{1-\frac{\eps}2}_{\infty,2}}\big\|\NN(\be \P_{N_k}v_{N_k})\big\|_{L^2_TH^{1-\frac{\eps}2} }\|\U_{N_k}-\U\|_{L^2_TH^{-1+\eps} }.
\end{split}
\label{R4}
\end{align}

\noi
Using the fractional chain rule of Lemma \ref{LEM:FC} (i) for $\A(u)=\NN(u)-1$, we bound
\begin{align}
\begin{split}
\big\|\NN(\be \P_{N_k}v_{N_k})\big\|_{L^2_TH^{1-\eps} }
&\les T^{\frac12}+\big\| \A(\be \P_{N_k}v_{N_k})\big\|_{L^2_{T}H^{1-\eps} }\les T^{\frac12}+\big\| v_{N_k}\big\|_{L^2_{T}H^{1-\eps} }\\
&\le C(T) \big(1+\|v_{N_k}\|_{\XX_T^\dl}\big).
\end{split}
\label{R5}
\end{align}

\noi
We also use the fractional chain rule of Lemma \ref{LEM:FC} (ii) to estimate for some large but finite $p_\eps$
\begin{align}
\begin{split}
\big\|e^{\be z}\|_{L^{\infty}_T B^{1-\frac{\eps}2}_{\infty,2}} &\les \big\|e^{\be z}\|_{L^{\infty}_T B^{1-\frac{\eps}4}_{p_\eps,2}} \les 1+ \big\|e^{2\be z}\big\|_{L^{\infty}_{T,x}}\|z\|_{L^{\infty}_TB^{1-\frac{\eps}8}_{\infty,2}}.
\end{split}
\label{R6}
\end{align}

Combining \eqref{R4}, \eqref{R5}, and \eqref{R6} we thus obtain
\begin{align}
\III \les e^{C \|z\|_{L^{\infty}_{T,x}}}\big(1+\|z\|_{L^{\infty}_TC^1 }\big)\big(1 +\|v_{N_k}\|_{\XX_T^\dl}\big)\|\U_{N_k}-\U\|_{L^2_TH^{-1+\eps} }.
\label{Dv3}
\end{align}

For the last term $\IV$ in \eqref{Dv0}, 
we use the Schauder estimate of Lemma \ref{LEM:heatker} and the product estimate\footnote{Note that $v(t)$ is indeed continuous in $x$ for almost every $t\in [0,T]$ by the Sobolev embedding $\XX_T^\dl \subset L^2([0,T];C(\M))$.} of Lemma \ref{LEM:posprod} with H\"older's inequality  and the boundedness of $\cj\NN$ to bound
\begin{align}
\begin{split}
\IV 
&\les \big\| e^{\be z}(\P_{N_k}v_{N_k}- v)\cj\NN(\P_{N_k}v_{N_k}, v)\U\big\|_{L^1_TH^{-1+\eps} }
\\
&\les \big\|e^{\be z}(\P_{N_k}v_{N_k}- v)\cj\NN(\P_{N_k}v_{N_k}, v)\big\|_{L^2_TL^{\infty} }
\|\U\|_{L^2_TH^{-1+\eps} }
\\
& \les
 e^{C\|z\|_{L^{\infty}_{T,x}}}\big(\|v_{N_k}- v\|_{L^2_TL^{\infty} }+\big\|\big[\P_{N_k}-\Id\big]v_{N_k}\|_{L^2_TH^{1+\frac\dl2} }\big)\|\cj\NN(v_{N_k}, v)\|_{L^{\infty}_{T,x}}
\|\U\|_{L^2_TH^{-1+\eps} }\\
&\les e^{C\|z\|_{L^{\infty}_{T,x}}}\big(\|v_{N_k}- v\|_{\XX_T^\dl}+N_k^{-\frac\dl2}\|v_{N_k}\|_{\XX_T^\dl}\big)\|\U\|_{L^2_TH^{-1+\eps} }.
\end{split}
\label{Dv4}
\end{align}

Owing to the convergence of $v_{N_k}$ to $ v$ in $\XX_T^\dl$
and $\U_{N_k}$ to $\U$
in $L^2([0,T];H^{-1+\eps}(\M))$, we see from \eqref{Dv0}, \eqref{Dv1}, \eqref{Dv2}, \eqref{Dv3} and \eqref{Dv4} that 
$\Phi_{N_k}(v_{N_k})$ converges to $\Phi( v)$ in $L^1([0,T];H^{-1 + \eps}(\M))$, and a fortiori in $\D'([0,T]\times\M)$.
Since $v_{N_k}  = \Phi_N(v_{N_k})$, 
we can conclude that 
\[  v = \lim_{k \to \infty}v_{N_k}
=\lim_{k \to \infty} \Phi_{N_k}(v_{N_k}) = \Phi( v) ,
\]
as equalities between functions in $\XX_T^\dl$.
This proves existence of a solution $v$
to \eqref{v3}
in $\XX_T^\dl \subset \XX_T^0$.

For the uniqueness in $\XX_T^0$, we proceed via an energy estimate: if $v_1, v_2 \in \XX_T^0$ are two solutions to \eqref{v3} which are limits in $\XX_T^0$ of solutions $v_{N,1},v_{N,2}$ to \eqref{v2}, noting $w=v_1-v_2$ we see that $w$ satisfies
\begin{align*}
\dt w -\frac1{4\pi}\Dlg w + \frac12\nu\be  e^{\be z}\big(\NN(\be v_1)-\NN(\be v_2)\big) \U = 0.
\end{align*}

\noi
We consider the energy of $w$:
\[\En(t) \deff \frac12\|w(t)\|_{L^2 }^2 + \frac1{4\pi}\int_0^t\|\nabla_\gm w(t')\|_{L^2 }^2dt'\geq 0,\]
and similarly for the energy $\En_N(t)$ of $w_N = v_{N,1}-v_{N,2}$. In particular, since $v_{N,j}\to v_j$ in $\XX_T^0$, we have that $\En_N\to \En$ in $C([0,T])$ for any $T>0$. 

\noi
Since $v_{N,j}$,$ j=1,2$, solve \eqref{v2} and that $\U_N$ and $z$ are smooth, we see that we have for fixed $N\in\N$ that $v_{N,j}\in C^1(\R_+,C^{\infty}(\M))$. Thus the energy functional $\En_N(t)$
is a well-defined differentiable function.
Moreover, since $\be v_{N,j}\leq 0$, $j = 1, 2$, we have again $\NN(\be v_{N,j}) = e^{\be v_{N,j}}$, so that
\[\cj\NN(v_{N,1},v_{N,2}) = \int_0^1\exp\big(\theta \be v_{N,1}+(1-\theta)\be v_{N,2}\big)d\theta\geq 0.\] 
This implies that
\begin{align*}
\frac{d}{dt}\En_N(t)& = \int_{\M}w_N(t)
\big(\dt w_N(t)-\frac1{4\pi}\Dlg w_N(t)\big)d\Vg\\
&= -\frac12\nu\be^2\int_{\M}w_N(t)^2  e^{\be{z}}\cj\NN(v_{N,1},v_{N,2})\U_N(t)d\Vg\\& \leq 0.
\end{align*}
Since $w_N(0) = 0$, 
we conclude that $\En_N(t) = 0$ for any $t \geq 0$. From the uniform convergence $\En_N\to \En$, we conclude that $\En(t)=0$ for any $t\ge 0$,
and $v_1 \equiv v_2$.
This finally proves uniqueness in the energy space $\XX_T^0$, and also convergence of the whole sequence $\{v_N\}_{N \in \N}$ in $\XX_T^\dl \subset \XX_T^0$.
\end{proof}

In particular, Proposition \ref{PROP:flow} implies that the solution
\begin{align*}
u_N = \<1>_\gm + z + v_N
\end{align*}
to \eqref{heat2} converges in measure to $u = \<1>_\gm + z + v$ in $\XX_T^\dl$, which proves Theorem~\ref{THM:GWP}~(i). Note that $u$ is then unique in the class
\begin{align*}
\<1>_\gm + z + \XX_T^\dl \subset C([0,T];H^s_0(\M)\oplus\R).
\end{align*}

\subsection{Invariance of the LQG measure}
\label{SUBSEC:Gibbs1}
In this subsection, 
 we finally conclude the proof of Theorem \ref{THM:GWP} (ii). We begin by proving the invariance of the truncated Gibbs measure $\rho_{N,\gm}$ \eqref{LQGN} under the flow $\wt u_N$ of \eqref{heat1} given by \eqref{Girsanov2}, where $u_N$ denotes the solution of \eqref{heat2} constructed in the previous subsection.
 \begin{lemma}\label{LEM:inv}
For any $N\in\N$, any $t\ge 0$ and any $F\in C_b(H^s_0(\M)\oplus\R)$, it holds
\begin{align*}
\int_{H^s_0(\M,\gm)}\int_\R \E\Big[F\big(\wt u_N(t,X_\gm,\cj X,\om)\big)\Big]d\rho_{N,\gm}(X_\gm,\cj X) = \int_{H^s_0(\M,\gm)}\int_\R F\big(X_\gm+\cj X\big)d\rho_{N,\gm}(X_\gm,\cj X).
\end{align*} 
 \end{lemma}

\begin{proof}
Fix $N\in\N$. For any $M\in\N$, let $\Pi_M$ be the projection on the space $\mathrm{Vect}\big\{\varphi_n, n=0,...,d_M\big\}\simeq \R^{d_M}$, with $d_M = \#\{n\ge 0, \ld_n \le M\}$. Similarly as in \eqref{heat1}, we then look at the dynamics
\begin{multline}\label{heat4}
\dt \wt u_{N,M} =\frac1{4\pi}\Dlg \wt u_{N,M}  -\frac{Q}{8\pi}\Rg_\gm+ \frac12\nu\be e^{-\pi\be^2C_\P}N^{-\frac{\be^2}2}\Pi_M\P_Ne^{\be \Pi_M\P_N\wt u_{N,M}}\\+\frac12\sum_{\ell=1}^La_\ell\Pi_M\P_N\dl_{x_\ell}+ \xi_\gm,
\end{multline}
with initial data $\wt u_0$ randomly distributed by $\rho_{N,M,\gm}$, where
\begin{align*}
d\rho_{N,M,\gm}(X_\gm,\cj X) = d\rho_{N,M,\gm}^F\otimes \big[\Id-\Pi_M\big]_\star \mu_\gm.
\end{align*}
Here the finite dimensional measure is given by
\begin{align*}
d\rho_{N,M,\gm}^F \deff \ZZ_{N,M}^{-1}e^{-E_{N,M}(U_0,...,U_{d_M-1})}\prod_{n=0}^{d_M-1}dU_n
\end{align*}
for the truncated energy
\begin{align*}
E_{N,M}(U_0,...,U_{d_M-1}) &= \sum_{n=0}^{d_M-1}\bigg\{\frac1{4\pi} \ld_n^2e^{-2N^{-2}\ld_n^2}U_n^2 + \frac{Q}{4\pi}U_n\langle\Rg_\gm,\varphi_n\rangle\gm-\sum_{\ell=1}^La_\ell U_0\\
&\qquad+\nu e^{-2N^{-2}\ld_n^2} \int_{\M}e^{-\pi\be^2C_\P}N^{-\frac{\be^2}2}e^{\be\sum_{m=0}^{d_M-1}e^{-2N^{-2}\ld_m^2}U_m\varphi_m}d\Vg\\
&\qquad-\sum_{\ell=1}^La_\ell e^{-2N^{-2}\ld_n^2}U_n\varphi_n(x_\ell)-\frac{a_\ell^2}2\big(\log N +2\pi C_\P\big)\bigg\}.
\end{align*}

For each fixed $N\in\N$, the dynamics \eqref{heat4} is defined in $C([0,T];H^s_0(\M)\oplus\R)$ for any $T>0$ by a standard argument since the nonlinearity is Lipschitz and the deterministic source terms are smooth. Moreover it holds $\wt u_{N,M}(t)\to \wt u_N(t)$ as $M\to\infty$ in law in $H^s_0(\M)\oplus\R$ for any $t\ge 0$.

We thus have for any $t\ge 0$
\begin{align*}
&\int_{H^s_0(\M,\gm)}\int_\R \E\Big[F\big(\wt u_N(t,X_\gm,\cj X,\om)\big)\Big]d\rho_{N,\gm}(X_\gm,\cj X)\\
&\qquad = \lim_{M\to\infty}\int_{H^s_0(\M,\gm)}\int_\R \E\Big[F\big(\wt u_{N,M}(t,X_\gm,\cj X,\om)\big)\Big]d\rho_{N,M,\gm}(X_\gm,\cj X)\\
&\qquad = \lim_{M\to\infty}\int_{H^s_0(\M,\gm)}\int_\R \E\Big[F\big(\Pi_M\wt u_{N,M}(t,X_\gm,\cj X,\om)+(1-\Pi_M)\wt u_{N,M}(t,X_\gm,\cj X,\om)\big)\Big]\\
&\hspace*{5cm}\times d\rho_{N,M,\gm}^F(X_\gm,\cj X)\otimes\big[\Id-\Pi_M\big]_\star\mu_\gm(X_\gm).
\end{align*}

We see that $\wt u_{N,M}^{\perp}=(\Id-\Pi_M)\wt u_{N,M}$ solves the \emph{linear} stochastic equation
\begin{align*}
\dt \wt u_{N,M}^\perp -\frac1{4\pi}\Dlg \wt u_{N,M}^\perp = (\Id-\Pi_M)\xi_\gm
\end{align*}
with initial data distributed by $\big[\Id-\Pi_M\big]_\star\mu_\gm$, so that $\big[\Id-\Pi_M\big]_\star\mu_\gm$ is the unique invariant measure for $\wt u_{N,M}^\perp$ (see e.g. \cite[Section 11.3]{DPZ}).

On the other hand, we can write $\Pi_M\wt u_{N,M}(t) = \sum_{n=0}^{d_M-1}U_n(t)\varphi_n$, so that $U_n$ solves the system of SDEs
\begin{align}\label{SDE}
dU_n = -\frac12\frac{\partial}{\partial U_n}E_{N,M}(U_0,...,U_{d_M-1})dt + dB_n, \hspace*{1cm}n=0,...,d_M-1,
\end{align}
with $B_n$ as in \eqref{Bn}.

The infinitesimal generator for \eqref{SDE} is given by
\begin{align*}
\L_{N,M}f(U_0,...,U_{d_M-1}) &= \sum_{n=0}^{d_M-1}-\frac12\frac{\partial}{\partial U_n}E_{N,M}(U_0,...,U_{d_M-1})\frac{\partial}{\partial U_n}f(U_0,...,U_{d_M-1})\\
&\qquad\qquad + \frac12\frac{\partial^2}{\partial (U_n)^2}f(U_0,...,U_{d_M-1})
\end{align*}
for any test function $f\in C^2_0(\R^{d_M})$. In particular, we have by integrating by parts
\begin{align*}
\int_{\R^{d_M}}\L_{N,M}fd\rho_{N,M,\gm}^F &= \int_{\R^{d_M}}\Big[\sum_{n=0}^{d_M-1}-\frac12\frac{\partial}{\partial U_n}E_{N,M}\frac{\partial}{\partial U_n}f + \frac12\frac{\partial^2}{\partial (U_n)^2}f\Big]e^{-E_{N,M}}dU\\
&=\sum_{n=0}^{d_M-1}-\frac12\int_{\R^{d_M}}\frac{\partial}{\partial U_n}E_{N,M}\frac{\partial}{\partial U_n}fe^{-E_{N,M}}dU\\
&\qquad\qquad + \frac12\int_{\R^{d_M}}\frac{\partial}{\partial U_n}E_{N,M}\frac{\partial}{\partial U_n}fe^{-E_{N,M}}dU\\
&=0.
\end{align*}
This shows that $\Pi_M\wt u_{N,M}$ also leaves $\rho_{N,M,\gm}^F$ invariant. 

All in all, we deduce that
\begin{align*}
&\int_{H^s_0(\M,\gm)}\int_\R \E\Big[F\big(\wt u_N(t,X_\gm,\cj X,\om)\big)\Big]d\rho_{N,\gm}(X_\gm,\cj X)\\
&\qquad = \lim_{M\to\infty}\int_{H^s_0(\M,\gm)}\int_\R \E\Big[F\big(\Pi_M\wt u_{N,M}(t,X_\gm,\cj X,\om)+(1-\Pi_M)\wt u_{N,M}(t,X_\gm,\cj X,\om)\big)\Big]\\
&\hspace*{5cm}\times d\rho_{N,M,\gm}^F(X_\gm,\cj X)\otimes\big[\Id-\Pi_M\big]_\star\mu_\gm(X_\gm)\\
&\qquad = \lim_{M\to\infty}\int_{H^s_0(\M,\gm)}\int_\R F\big(\Pi_M(X_\gm+\cj X)+(1-\Pi_M)X_\gm\big)\\
&\hspace*{5cm}\times d\rho_{N,M,\gm}^F(X_\gm,\cj X)\otimes\big[\Id-\Pi_M\big]_\star\mu_\gm(X_\gm)\\
&\qquad=\int_{H^s_0(\M,\gm)}\int_\R F\big(X_\gm+\cj X\big)d\rho_{N,\gm}(X_\gm,\cj X).
\end{align*}
\end{proof}

To conclude the proof of Theorem \ref{THM:GWP} (ii), we observe that by convergence in measure of $v_N$ in $\XX_T^\dl$ given by Theorem \ref{THM:GWP} (i) and definition of $\wt u_N$ and $u_N$, we have the convergence in law $\wt u_N(t)\to\wt u(t)$ in $H^s_0(\M)\oplus\R$ for any $t\ge 0$. With the weak convergence of $\rho_{N,\gm}$ to $\rho_{\{a_\ell,x_\ell\},\gm}$ given by Theorem \ref{THM:LQG} along with the invariance of $\rho_{N,\gm}$ given by Lemma \ref{LEM:inv}, it thus holds
\begin{align*}
&\int_{H^s_0(\M,\gm)}\int_{\R} \E\Big[ F\big(\wt u(t,X_\gm,\cj X,\om)\big)\Big]d\rho_{\{a_\ell,x_\ell\},\gm}(X_\gm,\cj X)\\&\qquad = \lim_{N\to\infty}\int_{H^s_0(\M,\gm)}\int_{\R} \E\Big[ F\big(\wt u_N(t,X_\gm,\cj X,\om)\big)\Big]d\rho_{N,\gm}(X_\gm,\cj X)\\
&\qquad= \lim_{N\to\infty}\int_{H^s_0(\M,\gm)}\int_{\R} F\big(X_\gm+\cj X\big)d\rho_{N,\gm}(X_\gm,\cj X)\\
&\qquad=\int_{H^s_0(\M,\gm)}\int_\R F(X_\gm+\cj X)d\rho_{\{a_\ell,x_\ell\},\gm}(X_\gm,\cj X).
\end{align*}

This concludes the proof of Theorem~\ref{THM:GWP}.

\begin{remark}\label{REM:delta}\rm
Strictly speaking, the solution given by Theorem \ref{THM:GWP} is not a strong solution (in the probability sense) of the original SPDE \eqref{heat1} since we did a change of probability space by the Girsanov transform \eqref{Girsanov2}. It would be very interesting though to be able to deal directly with the singular equation \eqref{heat1}.
 \end{remark}

\begin{remark}\label{REM:LQGwave}\rm
It would be interesting to also establish the analogue of Theorem \ref{THM:GWP} for the \emph{canonical} stochastic quantization of the LQG measure, namely prove global well-posedness of the stochastic {\it damped wave} equation
\begin{align*}
(\dt^2 -\Dlg + \dt) \wt u_N +\frac{Q}{8\pi}\Rg_\gm + \frac12\nu\be e^{-\pi\be^2C_\P}N^{-\frac{\be^2}2}\P_N\Big\{e^{\be \P_N\wt u_N}\Big\}\\
 =\frac12\sum_{\ell=1}^La_\ell \P_N\dl_{x_\ell} + \sqrt{2}\xi_\gm,
\end{align*}
 and invariance of the LQG measure (coupled with the white noise measure on $\dt \wt u$) under this flow. Although the \emph{local} well-posedness should follow from a straightforward adaptation of \cite[Theorem 1.6]{ORW} combined with the arguments of the present paper, it is not clear to us how to apply Bourgain's invariant measure argument to globalize this dynamics as in \cite[Section 6]{ORW}. Indeed, the argument in \cite{BO94,BO96} seems to require the densities $R_N$ to be in $L^p(d\mu_\gm\otimes d\cj X)$ uniformly in $N\in\N$ for some finite $p>1$, which does not hold since the use of Girsanov transform for $(R_N )^p$, $p>1$, produces a divergent constant $e^{p(p-1)\frac{\s_N}{2}\sum_{\ell=1}^La_\ell}$. 
\end{remark}

\appendix

\section{Remarks on the LQG and $\exp(\Phi)_2$ measures with negative cosmological constant}\label{SEC:A}
Theorem \ref{THM:LQG} discusses the construction of the  $L$-points correlations in the case where the cosmological constant $\nu$ is positive (the free case $\nu=0$ just corresponding to $d\rho_{\{a_\ell,x_\ell\},\gm} = d\mu_\gm\otimes d\cj X$). As is clear from \eqref{bd-RN(X,m)}, in the case $\nu<0$, the change of variable 
\begin{align*}
\cj X\mapsto\tau = -\nu e^{\be \cj X}\int_{\M}e^{\be\sum_{\ell=1}^La_\ell\P_N^2G_0(x_\ell,x)+\be\P_NX_0(x) -\frac{\be^2}2 \s_N(x;\gm_0)}dV_0(x)
\end{align*}
now gives the expression for the truncated partition function
\begin{align*}
\ZZ_N &= \GG_{N,0}\Xi\be^{-1} (-\nu)^{\be^{-1}\big[2\pi \chi(\M)Q)-\cj a\big]}\int_{0}^{\infty}\tau^{\be^{-1}\big[\cj a -2\pi \chi(\M)Q\big]-1}e^{+\tau}d\tau\notag\\&
 \qquad \times\int \Y_N(\M)^{\be^{-1}\big[2\pi \chi(\M)Q-\cj a\big]} d\mu_0(X_0)\\
 &=+\infty,
\end{align*}
with or without the Seiberg bounds. This proves that there are no well-defined correlation functions in this case.

In \cite{ORW}, we looked at the closely related $\exp(\Phi)_2$ (or H\o egh-Krohn \cite{Hoegh}) measure on the flat torus $\M=\T^2$ given by 
\begin{align*}
d\wt\rho =\lim_{N\to\infty} e^{-\nu \int_{\M}e^{\be \P_N\wt X - \frac{\be^2}2\wt\s_N}d\Vg}d\wt\mu(\wt X) ,
\end{align*}
where now $\wt\mu$ is the \emph{massive} Gaussian free field, i.e. under $\wt\mu$ we have
\begin{align}\label{mGFF}
\wt X(x) = \sum_{n\ge 0}\frac{h_n(\om)}{\sqrt{1+\ld_n^2}}\varphi_n(x)
\end{align}
for $h_n\sim\mathcal{N}(0,1)$ on $(\O,\Prob)$, and $\wt\s_N(x) = \E\big[|\P_N\wt X(x)|^2\big]$. In the defocusing case $\nu>0$ on $\M=\T^2$ with the flat metric\footnote{Note that in the massive case there is no conformal invariance anymore. However our proof of the construction of $\wt\rho$ (without curvature term compared to \eqref{SL}-\eqref{LQG}) readily extends to any compact surface $(\M,\gm)$ through the same arguments as in this work.}, we proved that the measure $\wt\rho$ above is indeed well-defined as a by-product of the construction of the GMC $:e^{\be \wt X}:$ and the integrability of the whole density which follows in this case from the trivial bound $e^{-\nu \int_{\M}:e^{\be \wt X}:d\Vg} \le 1$ when $\nu>0$. Namely, since the zero-th Fourier mode of $\wt X$ is now normally distributed (compared to \eqref{GFF}), the partition function (i.e. with $L=0$) is always finite. As for the focusing case $\nu<0$, we prove the following result.
\begin{proposition}\label{PROP:exp}
Let $\nu<0$. Then for any $\be\neq 0$, the H\o egh-Krohn partition function is not finite. More precisely, for any smooth approximations $\wt X_N$ of $\wt X$ such that $\wt X_N$ is still centred and Gaussian with variance $\wt\s_N + O(1)$, it holds
\begin{align*}
\sup_{N\in\N}\E\Big[\exp\Big(-\nu\int_{\M}e^{\be \wt X_N(x)-\frac{\be^2}2\wt\s_N}d\Vg(x)\Big)\Big] = +\infty.
\end{align*}
\end{proposition}
Note that here this holds for a large class of approximations $\wt X_N$, including the one considered in Section \ref{SEC:background}.
\begin{proof}
We begin by recalling that the Cameron-Martin space for $\wt\mu$ is given by $H^1(\M,\gm)$; in particular, Cameron-Martin's theorem (see e.g. \cite[Proposition 2.26]{DPZ}) states that for $\wt X$ as in \eqref{mGFF} and any $f\in H^1(\M,\gm)$, $\wt X - f$ is also a massive GFF under $e^{\langle f,\wt X\rangle_{H^1} - \frac12\|f\|_{H^1}^2}d\Prob$. In particular, for any $F\in C_b(H^s(\M))$, it holds
\begin{align}\label{CM}
\E\Big[F(\wt X)\Big] &= \E_f\Big[F(\wt X)e^{\frac12\|f\|_{H^1}^2-\langle f,\wt X\rangle_{H^1}}\Big] = \E\Big[F(\wt X+f)e^{\frac12\|f\|_{H^1}^2-\langle f, \wt X+f\rangle_{H^1}}\Big],
\end{align} 
where $\E_f$ is the expectation associated with $\Prob_f \deff e^{\langle f,\wt X\rangle_{H^1} - \frac12\|f\|_{H^1}^2}\Prob$.

We also recall the following Moser-Trudinger's inequality \cite[Theorems 2.46 and 2.50]{Aubin}: for $\theta>0$, there exists $C>0$ that for all $f\in H^1(\M)$,
\begin{align*}
\int_{\M}e^{f}d\Vg \le Ce^{\theta \|f\|_{H^1}^2}
\end{align*}
\emph{if and only if} $\theta\ge \frac1{16\pi}$. Thus for any $\be\neq 0$ and $N\in\N$, there exists $f_N\in H^1(\M)$ such that
\begin{align}\label{MT}
\int_{\M}e^{\be f_N}d\Vg \ge Ne^{\frac{\be^2}{10^5}\|f_N\|_{H^1}^2}.
\end{align}
Thus, using \eqref{CM}, we have for any $N\in\N$
\begin{align*}
&\E\Big[\exp\Big(-\nu\int_{\M}e^{\be \wt X_N-\frac{\be^2}2\wt\s_N}d\Vg\Big)\Big]\\
 &= \E\Big[\exp\Big(-\frac12\|f_N\|_{H^1}^2-\jb{f_N,\wt X_N}_{H^1}-\nu\int_{\M}e^{\be \wt X_N+\be f_N-\frac{\be^2}2\wt\s_N^2}d\Vg\Big)\Big]\intertext{Using Jensen's inequality, that $\wt X_N$ is centered and that $\E[e^{\be\wt X_N - \frac{\be^2}2\wt\s_N}]\sim 1$, along with \eqref{MT} and that $\nu<0$, we then continue with}
 &\ge \exp\Big[\E\Big(-\frac12\|f_N\|_{H^1}^2-\jb{f_N,\wt X_N}_{H^1}-\nu\int_{\M}e^{\be \wt X_N+\be f_N-\frac{\be^2}2\wt\s_N^2}d\Vg\Big)\Big]\\
 &\ge \exp\Big[-\frac12\|f_N\|_{H^1}^2-\nu\int_{\M}e^{\be f_N+c}d\Vg\Big]\ge \exp\Big[-\frac12\|f_N\|_{H^1}^2-\nu CNe^{\frac{\be^2}{10^5}\|f_N\|_{H^1}^2}\Big]
\end{align*}
for some $C>0$.

Note that, since $\nu<0$, the function $x\in [0,\infty)\mapsto (-\nu)CNe^{\frac{\be^2}{10^5}x}-\frac12x$ is strictly increasing on $[0,\infty)$ for any $N$ large enough (depending on $\be$), so that we have for  $N$ large enough
\begin{align*}
\E\Big[\exp\Big(-\nu\int_{\M}e^{\be \wt X_N-\frac{\be^2}2\wt\s_N}d\Vg\Big)\Big] & \ge \exp\big[-\nu CN\big]\too \infty
\end{align*}
as $N\to\infty$. This proves Proposition \ref{PROP:exp}.
\end{proof}

\begin{ackno}\rm
The authors are very grateful to R\'emi Rhodes and Vincent Vargas for pointing out the relevance of studying the conformal equation \eqref{heat0} and for interesting discussions on LCFT which motivated the writing of this paper. They are also very grateful to Christophe Garban, R\'emi Rhodes, Vincent Vargas and Younes Zine for helpful comments on a previous version of this work.

\noi
T.O.~was supported by the European Research Council (grant no.~637995 ``ProbDynDispEq''
and grant no.~864138 ``SingStochDispDyn"). T.R.~was supported by the DFG through the CRC 1283 ``Taming uncertainty and profiting from randomness
and low regularity in analysis, stochastics and their applications.''

\end{ackno}

\end{document}